\documentclass[a4paper,10pt,twoside]{article}

\usepackage{geometry}
\usepackage[T1]{fontenc}
\usepackage[utf8]{inputenc}
\usepackage[inline]{enumitem}
\usepackage{amsmath,amsfonts,amssymb,amsthm,xcolor,tikz,hyperref,bera,microtype,indentfirst,mathtools,bbm}
\usetikzlibrary{decorations.pathreplacing}
\hypersetup{
  pdfborder={0 0 0},
  pdfstartview={FitH},
  hypertexnames=false,
  breaklinks,
}

\numberwithin{equation}{section}

\newtheorem{theorem}{Theorem}[section]
\newtheorem{assumption}[theorem]{Assumption}
\newtheorem{claim}[theorem]{Claim}

\newtheorem{corollary}[theorem]{Corollary}
\newtheorem{definition}[theorem]{Definition}
\newtheorem{lemma}[theorem]{Lemma}
\newtheorem{notation}[theorem]{Notation}
\newtheorem{proposition}[theorem]{Proposition}

\theoremstyle{definition}
\newtheorem{remark}[theorem]{Remark}

\newcommand{\norm}[1]{\left\lVert#1\right\rVert}
\newcommand{\bignorm}[1]{\bigl\lVert{#1}\bigr\rVert}
\newcommand{\bigpar}[1]{\bigl(#1\bigr)}
\newcommand{\Bigpar}[1]{\Bigl(#1\Bigr)}
\newcommand{\indic}[1]{\mathbbm{1}_{\{\,{#1}\,\}}}
\newcommand{\bbone}{\mathbbm{1}}
\newcommand{\prob}{\mathbb{P}}
\newcommand{\ev}{\mathbb{E}}
\newcommand{\ibf}{\mathbf{i}}
\newcommand{\jbf}{\mathbf{j}}
\newcommand{\ppp}[2]{\underline{\mathcal{{#1}}}^n_{{#2}}}
\newcommand{\rbrwppp}[2]{\underline{\widehat{\mathcal{{#1}}}}^n_{{#2}}}
\newcommand{\modppp}[2]{\underline{\widetilde{\mathcal{{#1}}}}^n_{{#2}}}
\newcommand{\edge}[1]{\tau^n(#1)}
\newcommand{\rbrwedge}[1]{\widehat{\tau}(#1)}
\newcommand{\exptree}{\mathcal{T}^n_{\hat{t}}}
\newcommand{\rbrwtree}{\widehat{\mathcal{T}}^n_{\hat{t}}}
\newcommand{\modtree}{\widetilde{\mathcal{T}}^n_{\hat{t}}}
\newcommand{\emb}[1]{v^n(#1)}
\newcommand{\rbrwemb}[1]{\hat{v}^n(#1)}
\newcommand{\rt}{\varnothing}
\newcommand{\nb}[1]{N({#1})}
\newcommand{\neunb}[1]{\mathcal{N}({#1})}
\newcommand{\actnb}[1]{\mathcal{A}({#1})}
\newcommand{\expnb}[1]{\mathcal{E}({#1})}
\newcommand{\rbrwnb}[1]{\widehat{N}({#1})}
\newcommand{\rbrwneunb}[1]{\widehat{\mathcal{N}}({#1})}
\newcommand{\rbrwactnb}[1]{\widehat{\mathcal{A}}({#1})}
\newcommand{\rbrwexpnb}[1]{\widehat{\mathcal{E}}({#1})}
\newcommand{\nbh}[3]{\mathcal{B}^{#1, R}_{#2}\left(#3\right)}
\newcommand{\refCl}[1]{Claim~\ref{#1}}
\newcommand{\refC}[1]{Corollary~\ref{#1}}
\newcommand{\refD}[1]{Definition~\ref{#1}}
\newcommand{\refP}[1]{Proposition~\ref{#1}}
\newcommand{\refL}[1]{Lemma~\ref{#1}}
\newcommand{\refN}[1]{Notation~\ref{#1}}
\newcommand{\refS}[1]{Section~\ref{#1}}
\newcommand{\refT}[1]{Theorem~\ref{#1}}
\definecolor{FigBlue}{HTML}{000fff}
\definecolor{gray1}{rgb}{0.3,0.3,0.3}
\definecolor{gray2}{rgb}{0.45,0.45,0.45}
\definecolor{gray3}{rgb}{0.6,0.6,0.6}
\definecolor{gray4}{rgb}{0.75,0.75,0.75}
\definecolor{gray5}{rgb}{0.9,0.9,0.9}

\title{Local Limit of the Random Degree Constrained Process}

\author{%
Bal\'azs~R\'ath \thanks{Department of Stochastics, Institute of Mathematics, Budapest University of Technology and Economics,
M\H{u}egyetem rkp.\ 3., H-1111 Budapest, Hungary; HUN-REN--BME Stochastics Research Group, Budapest University of Technology and Economics, M\H{u}egyetem rkp.\ 3., H-1111 Budapest, Hungary;
HUN-REN Alfr\'ed R\'enyi Institute of Mathematics, Re\'altanoda utca 13-15, H-1053 Budapest,
Hungary.
Email:~\href{mailto:rathb@math.bme.hu}{\nolinkurl{rathb@math.bme.hu}}, \href{mailto:szokemp@math.bme.hu}{\nolinkurl{szokemp@math.bme.hu}}.}
\and 
M\'arton~Sz\H{o}ke\footnotemark[1]
\and
Lutz Warnke\thanks{Department of Mathematics, University of California San Diego, La Jolla CA 92093, USA.
Email:~\href{mailto:lwarnke@ucsd.edu}{\nolinkurl{lwarnke@ucsd.edu}}}}

\begin{document}

\maketitle

\begin{abstract}
In this paper we show that the random degree constrained process (a time-evolving random graph model with degree constraints) has a local weak limit, provided that the underlying host graphs are high degree almost regular.
We, moreover, identify the limit object as a multi-type branching process, by combining coupling arguments with the analysis of a certain recursive tree process. 
Using a spectral characterization, we also give an asymptotic expansion of the critical time when the giant component emerges in the so-called random $d$-process, resolving a problem of Warnke and Wormald for large~$d$.

\vspace{0.5em}
\end{abstract}

\section{Introduction}
In this paper we determine the local weak limit of a time-evolving random graph model with degree constraints. 
The model we study is the random degree constrained process (RDCP): starting with the empty graph, the edges (of some given host graph) are added one-by-one in a random order, but a new edge is only added if it does not violate the degree constraints at its end-vertices. 
This model is a generalization of the random \mbox{$d$-process} (in which each degree constraint is a fixed integer~$d$) that has been widely studied since the 1980s in many areas of science, including chemistry and physics~\cite{BQ87, BK, KQ}. 
Since then many interesting combinatorial properties of the random \mbox{$d$-process} have been investigated, including `static' properties of the final graph such as the number of edges~\cite{RW92}, connectedness~\cite{RW02}, the distribution of short cycles~\cite{RW97}, and Hamiltonicity~\cite{TWZ}, as well as `dynamic' properties of the time-evolving \mbox{$d$-process} graph such as the degree distribution~\cite{H24a,RW23,W95,W99a} and the `giant' component phase transition~\cite{EW,Se13,WW,WW2}. 

Local weak convergence describes the limiting local structure of the neighborhood of a typical vertex (often a tree structure). 
This is a useful concept, because numerous graph properties and parameters can be investigated using local convergence, including the matching number~\cite{BLS13}, maximum weight independent sets and matchings~\cite{GNS}, the number of spanning trees~\cite{L05}, weighted spanning subgraphs~\cite{Sa13}, the spectrum~\cite{BES, BL, BLS11}, the rank of the adjacency matrix~\cite{BLS11}, and the densest subgraph~\cite{AS16}.
Here a key conceptual point is that the limiting objects (often branching processes) are much simpler to study than the original model. 

For random graphs, most of the existing results establish local weak convergence for `static' models where there is no time-evolution, 
like Erd\H os--R\'enyi random graphs, random graphs with a given degree sequence, and inhomogeneous random graphs (see~\cite{BBSY,H24b} and Section~\ref{subsec:related}). 
Far fewer results establish local weak convergence for `dynamic' models where there is time-evolution, like in the random degree constrained process we study in this paper. 
In fact, most of these dynamic results exploit some `special' property of the time-evolving model that leads to a well-behaved limit object, for example that their evolution is closely linked to a P\'olya-urn process~\cite{BBCS}, 
that they have uniform distribution after conditioning on the degree sequence~\cite{JW,MHZ}, or that they are inhomogeneous random graphs after conditioning on certain information~\cite{CRY,RT}. 
The fact that in this paper we can determine the limit object of the `dynamic' random degree constrained process without exploiting such a special property\footnote{Such a special property does not seem to be available for the RDCP, which we think of as a genuine time-evolving model. As some partial evidence towards this, we remark that by~\cite{MSW} we know that the final graph of the RDCP is not contiguous to the configuration model with the same degree sequence (provided that the degree constraints are not nearly regular).} 
is one of our main conceptual~contributions.

Our main results for the random degree constrained process (RDCP) are as follows:
{\vspace{-0.25em}\begin{enumerate}[label={(\arabic*)}]
\itemsep 0.125em \partopsep=0pt \parsep 0em 
\item \textbf{Local Weak Convergence:} 
We determine the local weak limit of the random degree constrained process on regular host graphs~$G^n$ with degree~$r_n \to \infty$; see Theorem~\ref{thm:local-limit1}. 
The local limit object is universal in the sense that it does not depend on the sequence of underlying host graphs~$G^n$ (and we can even relax the regularity condition slightly by only assuming that the sequence of host graphs is \emph{high degree almost regular}, see Definition~\ref{def:hdar}).
\item \textbf{Identification of Limit Object:} 
We explicitly identify the local weak limit of the random degree constrained process in three different ways:
(i)~in terms of the RDCP on the Poisson weighted infinite tree (PWIT) (cf.\ Section~\ref{subsec:rdcp_on_pwit}), (ii)~in terms of a so-called recursive tree process (RTP), i.e., a special case of the framework introduced in \cite{AB} (cf.\ Section~\ref{sss:additional_RTP}), and (iii)~as a multi-type branching process (MTBP) (cf.\ Section~\ref{sss:additional_MTBP}).
In our case the response variable of a vertex of the RTP as well as the type of a vertex in the MTBP is a positive real number (the so-called phantom saturation time of a~vertex, see Definition~\ref{def:phantom-sat}). 
\item \textbf{Critical Time:} 
To demonstrate the applicability of our local limit results, we consider the random \mbox{$d$-process}, i.e., the RDCP on the complete $n$-vertex host graph~$G^n=K_n$, where each vertex~$v$ has the same degree constraint~$d(v)=d \ge 3$. 
Warnke and Wormald~\cite{WW} showed that there exists a critical time~${t_c=t_c(d)}$ such that a `giant' component of size $\Theta(n)$ emerges after adding around~$t_c n$ edges. 
Using the MTBP characterization of the local limit, we give a new spectral characterization of the critical time; see Theorem~\ref{thm:PF-eigenvalue}. 
From this we derive the asymptotic expansion ${t_c(d) = 1/2 + (1+o(1))/(\mathrm{e} \cdot d!)}$ of the critical time (as~${d \to \infty}$), and deduce that~${t_c(d)>1/2}$ for sufficiently large~$d$, resolving a problem of Warnke and Wormald~\cite[Section~6]{WW} for large~$d$. Actually, our asymptotic expansion of the critical time is more generally valid for the RDCP with high degree constraints; see Theorem~\ref{thm:tc-asymptotics} and~\eqref{eq:discrete-tc:dprocess}. 
\vspace{-0.25em}\end{enumerate}}

The remainder of this introduction is organized as follows. 
In Section~\ref{subsec:setup} we formally define the continuous-time random degree constrained process (RDCP) on high degree almost regular host graphs.
In Section~\ref{subsec:rdcp_on_pwit} we define the limit object that arises in the local weak limit: the RDCP on the PWIT.
In Section~\ref{subsec:local_limit} we state our first main result, \refT{thm:local-limit1}, about the local weak convergence of the RDCP graph.
In Section~\ref{subsec:discrete} we discuss how this result transfers to the discrete-time variant of the RDCP. 
In Section~\ref{subsec:intro_tc} we state our other main result, Theorem~\ref{thm:tc-asymptotics} about the asymptotic expansion of the critical time~$t_c$. 
In Section~\ref{subsec:auxiliary_results} we give two further equivalent descriptions of the limit object: one that uses recursive tree processes (RTPs) (cf.\ Section~\ref{sss:additional_RTP}) and another one that uses multi-type branching processes (MTBPs) (cf.\ Section~\ref{sss:additional_MTBP}).
In Section~\ref{sss:additional_spectral} we provide a new spectral characterization of the critical time~$t_c$ in terms of the branching operator of the MTBP representation that allows us to prove our main results about the asymptotic expansion of~$t_c$.
In Section~\ref{subsec:related} we briefly discuss some related work on the topics of RDCP, local limits and the PWIT. Finally, in Section~\ref{subsec:outline_of_rest_of_paper} we outline the structure of the rest of this paper.

\subsection{Random degree constrained process (RDCP)}\label{subsec:setup}
In this section we formally define the random degree constrained process (RDCP) on a host graph~$G^n$, which requires some setup.
 
In previous work the host graph~$G^n$ was mostly\footnote{For example, the papers~\cite{H24a,MSW,RW92,RW97,RW02,RW23,Se13,WW} study the RDCP with host graph~${G^n=K_n}$. We remark that degree constrained percolation models have also been investigated on some infinite host graphs, including the square lattice~\cite{LSSST, SSS, SS21}, $m$-ary trees~\cite{LSSST}, and the hypercubic lattice~\cite{APS24,HL,HS}.} the complete $n$-vertex graph ${G^n=K_n}$. 
In this paper we allow for sparser host graphs~$G^n$ that are almost~$r_n$ regular in the sense of the following definition from~\cite[Definition~1.3]{NP}; note that we assume~$r_n \to \infty$ as~$n \to \infty$ (and that we do not assume~$G^n$ to have~$n$~vertices). 
\begin{definition}[High degree almost regularity]\label{def:hdar}
A sequence of finite, simple, connected graphs $\{\,G^n\,\}$ is called \textbf{high degree almost regular} if there exist non-negative sequences\footnote{Note that w.l.o.g.\ we could have replaced $a_n$, $b_n$ and $c_n$ by $\max\{a_n, \: b_n, \: c_n\}$ in Definition~\ref{def:hdar}. We chose to denote $a_n$, $b_n$ and $c_n$ with separate symbols, because this makes it easier to track which condition is used in calculations.}
$r_n$, $a_n$, $b_n$ and $c_n$ such that
$r_n \to \infty$ and $\max\{a_n, \: b_n, \: c_n\} \to 0$ as~$n \to \infty$, 
and that, for any sufficiently large $n \in \mathbb{N}$, we have 
{\vspace{-0.25em}\begin{enumerate}[label=(\alph*)]
\itemsep 0.125em \partopsep=0pt \parsep 0em 
\item\label{item:hdar_regular} at least $(1-a_n) \cdot |V(G^n)|$ vertices of $G^n$ have degree in the interval $(1 \pm b_n)\cdot r_n$, and 
\item\label{item:hdar_degree_sum} the sum of degrees in $G^n$ is in the interval $(1 \pm c_n) \cdot r_n \cdot |V(G^n)|$.
\vspace{-0.25em}\end{enumerate}}%
\end{definition}
We call $r_n$ the typical degree of $G^n$.

We next discuss the (random) degree constraints of the vertices of the host graph~$G^n$. 
\begin{assumption}[Degree constraint distribution (mild)]\label{assump:degree_constraint_mild} Let $p_k \geq 0$ for each $k=1,2,\dots$ and let us assume that $\sum_{k=1}^{\infty} p_k =1$. 
Let us denote $\underline{p} = (p_k)_{k=1}^\infty$.
\end{assumption}

For each vertex $v \in V(G^n)$ of the host graph~$G^n$, we independently let $d(v) \sim \underline{p}$ be the random \textbf{degree constraint} of~$v$.
Hence~$p_k$ is approximately the fraction of vertices with degree constraint~$k$. We assume $p_0=0$, because a vertex $v$ with $d(v)=0$ would remain isolated in the graph that we produce (so there would be no point in including it).

Some of our results are proved under a more strict condition on $\underline{p}$:

\begin{assumption}[Degree constraint distribution (strict)]\label{assump:degree_constraint_strict} Let $\Delta \in \mathbb{N}_+$, $2 \leq \Delta <\infty$. Let $p_k \geq 0$ for each $k=2, \dots, \Delta$ and let us assume $p_2+\dots+p_\Delta=1$. Let
\begin{equation} \label{eq:upper-bound-assumption}
\underline{p} = (p_k)_{k=2}^\Delta, \quad \text{i.e.,} \quad
p_k = 0 \quad \text{for} \quad k=0 \quad \text{and} \quad k=1 \quad \text{and for any} \quad k > \Delta.
\end{equation}
\end{assumption}

Observe that ${\Delta \in \mathbb{N}_+}$ is an upper bound on the degree constraints (and eventually the degrees of the random graph that we produce, cf.\ Definition~\ref{def:rdcp}).

We will prove our local weak convergence result and the various equivalent characterizations (RTP, MTBP) of the limit object under the mild Assumption~\ref{assump:degree_constraint_mild}.
We use the strict Assumption~\ref{assump:degree_constraint_strict} only to prove our results related to the critical time.

Intuitively, the RDCP starts with the empty graph on the vertex set of the host graph~$G^n$, and then sequentially tries to add the edges of~$G^n$ in random order, only adding those that do not violate any degree constraint. 
For mathematical convenience, in this paper we use the natural continuous-time version of the RDCP defined below (see Section~\ref{subsec:discrete} for a discrete-time variant), where we tacitly assume that $\{\,G^n\,\}$ is a sequence of high degree almost regular graphs with typical degree $r_n$. 
Recall that each vertex~$v \in V(G^n)$ has its own independent degree constraint $d(v) \sim \underline{p}$, where $\underline{p}$ satisfies the mild Assumption~\ref{assump:degree_constraint_mild}.

\begin{definition}[RDCP on a high degree almost regular graph]\label{def:rdcp}
The random degree constrained process $\bigpar{G^n_{\underline{p}}(t)}_{t \geq 0}$ is a time-evolving random graph process on $G^n$, where $t$ is a \textbf{continuous time} parameter and $G^n_{\underline{p}}(0)$ is the empty graph with vertex set~$V(G^n)$. 
For each edge $e \in E(G^n)$, we independently let $X_e \sim \text{EXP}\left(1/r_n\right)$ be the \textbf{activation time} of~$e$. 
As time evolves, we attempt to add edges to the graph one-by-one, in increasing order of their activation times. 
An edge $e=vw$ is added at its activation time $X_e$ if and only if the degrees of vertices $v$ and~$w$ in the RDCP graph $G^n_{\underline{p}}(X_e^-)$ right before time~$X_e$ are strictly less than their degree constraints $d(v)$ and~$d(w)$, respectively.
If no more edges can be added to the graph without violating any degree constraint, we arrive at the \textbf{final graph} of the RDCP on $G^n$, which we denote by $G^n_{\underline{p}}(\infty)$.
\end{definition}

We say that vertex $v$ is \textbf{saturated} in the RDCP on $G^n$ at time $t$ if its degree in $G^n_{\underline{p}}(t)$ is equal to $d(v)$, i.e., no more edges containing vertex~$v$ can be added. 
The saturation time of a vertex is the time when it saturates. 
Note that the final graph can potentially contain unsaturated vertices (if $G^n$ has no edges that connect unsaturated vertices).

\subsection{RDCP on the Poisson weighted infinite tree (PWIT)}\label{subsec:rdcp_on_pwit}

One of our goals is to show that the sequence of graphs $G^n_{\underline{p}}(t)$ (cf.\ Definition~\ref{def:rdcp}) locally converges as $n \to \infty$. 
In this subsection we describe the limit object $G^\infty_{\underline{p}}(t)$ by adapting the random degree constrained process (RDCP) to the Poisson weighted infinite tree (PWIT).

\begin{remark}[Motivation for the definition of PWIT]\label{rem:mot_pwit} In order to understand the local limit of $G^n_{\underline{p}}(t)$ as $n \to \infty$, we want to describe the structure of neighbors, the neighbors of neighbors, etc.\ of a typical vertex $v$ of $G^n$. If $n \gg 1$ then $r_n \gg 1$.
The degree of~$v$ is close to~$r_n$ and the activation times of the edges emanating from~$v$ are i.i.d.\ with $\text{EXP}\left(1/r_n\right)$ distribution, so the first few activation times of edges adjacent to~$v$ can be approximated by the arrival times of a homogeneous Poisson point process (PPP) with unit intensity.
Also note that only these early edges matter in the construction of the RDCP, since the edges that activate after $v$ saturates never get added. 
Also note that if~$w$ is an early neighbor of~$v$ in~$G^n$, then the point process of the first few activation times that we can observe on the set of edges that connect~$w$ to its other neighbors is again approximately a unit intensity PPP, which is approximately independent of the PPP corresponding to the activation times of the early neighbors of~$v$. 
Also note that it follows from $r_n \gg 1$ that the subgraph of $G^n$ spanned by such early edges will look locally like a tree: it is unlikely to see a short cycle of early edges that contains~$v$.
\end{remark}

With the above motivation in mind, let us define the PWIT. We will discuss the history and other applications of the PWIT in Section~\ref{subsec:related}.

\begin{definition}[PWIT]\label{def:pwit}
The \textbf{Poisson weighted infinite tree} (PWIT) is an infinite tree with random edge labels. The vertices are denoted by finite strings of positive integers. If $\ibf$ is such a string and $j$ is a positive integer, then we say that the vertex~$\ibf j$ is a child of vertex~$\ibf$.
The root of the tree is the empty string~$\rt$. For each vertex~$\ibf$ we define a unit intensity Poisson point process (PPP) $\underline{\tau}^{\ibf} = (\tau^{\ibf j})_{j=1}^\infty$ on $[0,\infty)$, and we assume that these PPPs are independent of each other. We say that~$\tau^{\ibf j}$ is the \textbf{label of the edge}~$\{\ibf, \ibf j\}$.
\end{definition}
We may also interpret the label~$\tau^{\ibf j}$ as the activation time of the edge that connects~$\ibf$ and~$\ibf j$. Let us stress that the label of an edge that connects a parent $\ibf$ to a child $\ibf j$ might be bigger than the label of an edge that connects $\ibf j$ to a grandchild $\ibf jl$ of $\ibf$. In other words, we can have $\tau^{\ibf jl}<\tau^{\ibf j}$.
See Figure~\ref{fig:pwit} for an illustration of the PWIT.

\begin{figure}[t]
\centering
\begin{tikzpicture}[font=\small, node distance={1.0cm}, thick, main/.style = {draw, circle}]

\node[main, fill=black, scale=0.3,label=left:{$\scriptstyle{\varnothing}$}] at (-3,-1)(0) {};
\node[main, fill=black, scale=0.3, label=left:{$\scriptstyle{1}$}] at (-6,1) (1) {};
\node[main, fill=black, scale=0.3, label=left:{$\scriptstyle{2}$}] at (-3,1) (2) {};
\node[main, fill=black, scale=0.3, label=left:{$\scriptstyle{3}$}] at (0,1) (3) {};
\node[main, fill=black, scale=0.3, label=left:{$\scriptstyle{\dots}$}] at (1.5,1) (4) {};
\node[main, fill=black, scale=0.3, label=above:{$\scriptstyle{11}$}] at (-7,2) (11) {};
\node[main, fill=black, scale=0.3, label=above:{$\scriptstyle{12}$}] at (-6,2) (12) {};
\node[main, fill=black, scale=0.3, label=above:{$\scriptstyle{13}$}] at (-5,2) (13) {};
\node[main, fill=black, scale=0.3, label=above:{$\scriptstyle{\dots}$}] at (-4.5,2) (14) {};
\node[main, fill=black, scale=0.3, label=above:{$\scriptstyle{21}$}] at (-4,2) (21) {};
\node[main, fill=black, scale=0.3, label=above:{$\scriptstyle{22}$}] at (-3,2) (22) {};
\node[main, fill=black, scale=0.3, label=above:{$\scriptstyle{23}$}] at (-2,2) (23) {};
\node[main, fill=black, scale=0.3, label=above:{$\scriptstyle{\dots}$}] at (-1.5,2) (24) {};
\node[main, fill=black, scale=0.3, label=above:{$\scriptstyle{31}$}] at (-1,2) (31) {};
\node[main, fill=black, scale=0.3, label=above:{$\scriptstyle{32}$}] at (0,2) (32) {};
\node[main, fill=black, scale=0.3, label=above:{$\scriptstyle{33}$}] at (1,2) (33) {};
\node[main, fill=black, scale=0.3, label=above:{$\scriptstyle{\dots}$}] at (1.5,2) (34) {};

\draw[gray5] (0)--(4) node [midway, xshift=-0.02cm,yshift=-0.1cm] {$\scriptstyle{\dots}$};
\draw[gray3] (0)--(3) node [midway, xshift=-0.25cm] {$\scriptstyle{\tau}^{\scriptscriptstyle{3}}$};
\draw[gray1] (0)--(2) node [midway, xshift=-0.2cm] {$\scriptstyle{\tau}^{\scriptscriptstyle{2}}$};
\draw (0)--(1) node [midway, xshift=-0.36cm] {$\scriptstyle{\tau}^{\scriptscriptstyle{1}}$};

\draw[gray5] (1)--(14) node [midway, xshift=0.18cm]
{$\scriptstyle{\dots}$};
\draw[gray3] (1)--(13) node [midway, xshift=-0.12cm,yshift=0.1cm] {$\scriptstyle{\tau}^{\scriptscriptstyle{13}}$};
\draw[gray1] (1)--(12) node [midway, xshift=-0.22cm,yshift=0.1cm] {$\scriptstyle{\tau}^{\scriptscriptstyle{12}}$};
\draw (1)--(11) node [midway, xshift=-0.39cm,yshift=0.1cm] {$\scriptstyle{\tau}^{\scriptscriptstyle{11}}$};

\draw[gray5] (2)--(24) node [midway, xshift=0.18cm]
{$\scriptstyle{\dots}$};
\draw[gray3] (2)--(23) node [midway, xshift=-0.12cm,yshift=0.1cm] {$\scriptstyle{\tau}^{\scriptscriptstyle{23}}$};
\draw[gray1] (2)--(22) node [midway, xshift=-0.22cm,yshift=0.1cm] {$\scriptstyle{\tau}^{\scriptscriptstyle{22}}$};
\draw (2)--(21) node [midway, xshift=-0.39cm,yshift=0.1cm] {$\scriptstyle{\tau}^{\scriptscriptstyle{21}}$};

\draw[gray5] (3)--(34) node [midway, xshift=0.18cm]
{$\scriptstyle{\dots}$};
\draw[gray3] (3)--(33) node [midway, xshift=-0.12cm,yshift=0.1cm] {$\scriptstyle{\tau}^{\scriptscriptstyle{33}}$};
\draw[gray1] (3)--(32) node [midway, xshift=-0.22cm,yshift=0.1cm] {$\scriptstyle{\tau}^{\scriptscriptstyle{32}}$};
\draw (3)--(31) node [midway, xshift=-0.39cm,yshift=0.1cm] {$\scriptstyle{\tau}^{\scriptscriptstyle{31}}$};

\node[main, fill=black, scale=0.3] at (-3,-1)(0) {};
\node[main, fill=black, scale=0.3] at (-6,1) (1) {};
\node[main, fill=black, scale=0.3] at (-3,1) (2) {};
\node[main, fill=black, scale=0.3] at (0,1) (3) {};
\node[main, fill=black, scale=0.3] at (1.5,1) (4) {};
\node[main, fill=black, scale=0.3] at (-7,2) (11) {};
\node[main, fill=black, scale=0.3] at (-6,2) (12) {};
\node[main, fill=black, scale=0.3] at (-5,2) (13) {};
\node[main, fill=black, scale=0.3] at (-4.5,2) (14) {};
\node[main, fill=black, scale=0.3] at (-4,2) (21) {};
\node[main, fill=black, scale=0.3] at (-3,2) (22) {};
\node[main, fill=black, scale=0.3] at (-2,2) (23) {};
\node[main, fill=black, scale=0.3] at (-1.5,2) (24) {};
\node[main, fill=black, scale=0.3] at (-1,2) (31) {};
\node[main, fill=black, scale=0.3] at (0,2) (32) {};
\node[main, fill=black, scale=0.3] at (1,2) (33) {};
\node[main, fill=black, scale=0.3] at (1.5,2) (34) {};
\end{tikzpicture}
\caption{Illustration of the PWIT. Vertices and edge labels are denoted as in \refD{def:pwit}. Color strength of the edges indicates the importance of the given edge in the RDCP on the PWIT: we have $\tau^{\ibf 1}<\tau^{\ibf 2}<\tau^{\ibf 3}<\dots$, and edges with a bigger label (lighter gray) tend to be less important than edges with a smaller label (darker gray) 
since edges with later activation times have a smaller chance of being added in the RDCP than edges with earlier activation times.}
\label{fig:pwit}
\end{figure}

\begin{definition}[RDCP on PWIT]\label{def:rdcp-on-pwit}
We decorate the vertices of the PWIT with i.i.d.\ random degree constraints with distribution~$\underline{p}$ satisfying the mild Assumption~\ref{assump:degree_constraint_mild}, writing~$d(\ibf)$ for the degree constraint of vertex~$\ibf$. 
By identifying the activation time of an edge with its label (see \refD{def:pwit}), the usual definition of the RDCP time evolution described in \refD{def:rdcp} then carries over to the PWIT: we attempt to add the edges in increasing order of their labels, only adding them if this does not violate any degree constraint. 
For fixed $t \in \mathbb{R}_+ \cup \{+\infty\}$ let $G^\infty_{\underline{p}}(t)$ denote the connected component of the root~$\rt$ in the subgraph of the PWIT that consists of the edges that got added in the RDCP by time~$t$.
Hence $G^\infty_{\underline{p}}(\infty)$ denotes the connected component of the root in the final graph of the RDCP on the PWIT.
\end{definition}

As before, we say that vertex $\ibf$ \textbf{saturates} when its degree reaches $d(\ibf)$, so that no more edges containing vertex~$\ibf$ can be added. 
Since each vertex~$\ibf$ of the PWIT has its own independent PPP~$\underline{\tau}^{\ibf}$ (the points are the labels of the edges that connect~$\ibf$ to its children) and since the PWIT has infinitely many vertices, the set of all edge labels is dense in~$\mathbb{R}_+$. 
In particular, the instructions are not well-ordered chronologically.
Therefore, it is not obvious that the RDCP on the PWIT is well-defined: we clarify this in the next claim.
\begin{claim}[RDCP on PWIT is well-defined]\label{cl:rdcp-pwit-welldefined} Let
$\underline{p}$ satisfy the mild Assumption~\ref{assump:degree_constraint_mild}.
The algorithm described in \refD{def:rdcp-on-pwit} almost surely determines the random rooted tree~$G^\infty_{\underline{p}}(t)$ for any~$t \in \mathbb{R}_+ \cup \{+\infty\}$.
\end{claim}

We will prove Claim~\ref{cl:rdcp-pwit-welldefined} in \refS{subsec:properties_rbrw_tree}, but let us sketch the gist of the idea here. 

\begin{remark}[Monotone decreasing labels and causality]\label{rem:mon_dec_causality}
It is enough to show that we only need to explore an almost surely finite subtree of the PWIT in order to determine whether an edge~$e$ of the PWIT can be added when it activates. If we condition on the value of the activation time of~$e$ (say it is equal to~$t$), then the two subtrees of the PWIT that we obtain by erasing this edge are actually two independent copies of the PWIT (rooted at the two endpoints of~$e$): let us focus on the left subtree. The edge set of this left PWIT can be partitioned into generations according to the length of the (unique) path that connects an edge to~$e$.
Each edge has a unique parent in the previous generation. 
In order to decide whether we can add~$e$ at its activation time~$t$, it is enough to focus on those edges of the first generation whose activation time is less than~$t$, since only these edges can potentially have an effect on our decision. Let us call these edges the crucial edges of the first generation.
We will also determine whether we can add these crucial edges at their respective activation times (since given this information and the degree constraint of the left endpoint of~$e$, we can decide whether or not the left endpoint is saturated when we try to add~$e$ at time~$t$). In order to decide about a first generation crucial edge, we only need to look at those children of it whose activation times are smaller than the activation time of the parent edge. Let us call these edges the crucial edges of the second generation, etc. Observe that if we look at the unique path from~$e$ to a crucial edge, then the sequence of activation times of the edges of this path are monotone decreasing. Using this, one can check that the expected number of crucial edges in the~$r^{\text{th}}$ generation is~${t^r/r!}$ (where the~$r!$ in the denominator accounts for the fact that only one out of the $r!$ possible rearrangements of the~$r$ edge labels of the path has the monotone decreasing property), and thus the expected number of crucial edges in the left PWIT is~$ \sum_{r=0}^{\infty} t^r/r!=\mathrm{e}^t$, thus the subtree of crucial edges that we need to explore is almost surely finite.
\end{remark}

Let us note that a similar localization idea (involving paths with decreasing activation times) appeared independently in~\cite[Section~2]{APS24} for a degree constrained percolation model on an infinite~lattice.

\subsection{Local weak convergence}\label{subsec:local_limit}
One main result of this paper is that the random degree constrained process (RDCP) has a local weak limit. 
Let us briefly recall the notion of the local weak convergence of random graphs that we shall use.

\begin{definition}[Local weak convergence in probability]\label{def:local:weak:conv}
Let $\{\,\mathcal{G}^n\,\}$ be a sequence of random graphs. Let $V(\mathcal{G}^n)$ denote the vertex set of~$\mathcal{G}^n$. 
Let $\mathcal{B}^{n,R}(w)$ denote the $R$-neighborhood in $\mathcal{G}^n$ of a vertex $w \in V(\mathcal{G}^n)$ with respect to graph distance, viewed as a graph rooted at~$w$. 
Then we say that $\mathcal{G}^n$ \textbf{converges locally in probability} to the random rooted graph~${(\mathcal{G},u)}$, if for every rooted graph $(H,v)$ and every integer $R \ge 1$ we have
\begin{equation}\label{eq:local-weak-in-prob}
\frac{1}{|V(\mathcal{G}^n)|} \sum \limits_{w \in V(\mathcal{G}^n)} \indic{\mathcal{B}^{n,R}(w) \simeq (H,v)} \stackrel{\prob}{\longrightarrow} \prob\bigpar{\mathcal{B}^R(u) \simeq (H,v) } \quad \text{as } n \to \infty,
\end{equation}
where $\mathcal{B}^R(u)$ is the $R$-neighborhood of the root $u$ in $\mathcal{G}$, and the relation $\simeq$ denotes the root-preserving isomorphism of rooted graphs. 
\end{definition}

\begin{remark}[Local weak convergence in distribution]
Sometimes a slightly weaker convergence notion than~\eqref{eq:local-weak-in-prob} is studied (cf.~the discussion in~\cite[Remarks~2.12-2.13]{H24b}). 
Indeed, if we let $u^n$ denote a vertex that is independent of $\mathcal{G}^n$ and uniformly distributed in $V(\mathcal{G}^n)$, then~\eqref{eq:local-weak-in-prob} implies
\begin{equation}\label{eq:local-weak-dist}
\prob \left(\mathcal{B}^{n,R}(u^n) \simeq (H,v) \right) \longrightarrow \prob \left(\mathcal{B}^R(u) \simeq (H,v) \right) \quad \text{ as } n \to \infty
\end{equation}
for every rooted graph $(H,v)$, which is also called \emph{convergence in distribution in the local weak sense}; see~\cite[Definition~2.11]{CRY} or~\cite[Definition~2.11]{H24b}.
\end{remark}

We are now ready to state out first main result: 
in concrete words, Theorem~\ref{thm:local-limit1} says that the local weak limit of the RDCP on high degree almost regular host graphs~$G^n$ (which includes complete graphs~$K_n$, complete bipartite graphs~$K_{n,n}$, and also $r_n$-regular graphs with degree~$r_n\to \infty$ as~$n \to \infty$, just to name a~few) is the RDCP on the PWIT.

\begin{theorem}[Local weak limit of RDCP]\label{thm:local-limit1}
Let~$\{\,G^n\,\}$ be a sequence of high degree almost regular graphs. Let $\underline{p}$ satisfy the mild Assumption~\ref{assump:degree_constraint_mild}. Let us fix~$\hat{t} \in \mathbb{R}_+ \cup \{+\infty\}$.
Then
\begin{equation}\label{eq:local-limit}
G^n_{\underline{p}}(\hat{t}) \longrightarrow \Bigpar{G^\infty_{\underline{p}}(\hat{t}), \rt } 
\quad \text{locally in probability as } n \to \infty,
\end{equation}
where $\bigpar{G^\infty_{\underline{p}}(\hat{t}), \rt }$ was introduced in Definition~\ref{def:rdcp-on-pwit}.
\end{theorem}

Note that the local weak limit is universal in the sense that the limit object $G^\infty_{\underline{p}}(\hat{t})$ does not depend on the underlying host graphs~$G^n$ (whereas different host graphs can potentially have very different global properties). 
Notably, in this paper we give two further equivalent characterizations of the local weak limit object $G^\infty_{\underline{p}}(\hat{t})$: 
(i)~in terms of a recursive tree process (RTP) and
(ii)~as a multi-type branching process (MTBP); 
see Sections~\ref{sss:additional_RTP} and~\ref{sss:additional_MTBP} for the details. 

We will prove \refT{thm:local-limit1} in \refS{sec:section_local_limit} using a coupling argument that makes the heuristics outlined in Remark~\ref{rem:mot_pwit} rigorous. We sketch some further proof ideas (which involve a \emph{two-phase graph exploration}) in Remark~\ref{rem:heu_two_phase}. A more 
detailed and self-contained heuristic explanation as to why the local weak limit of the RDCP exists at any time $\hat{t}$ can be found in Section~1.6 of an earlier version~\cite{RSzWArXivV1} of our paper.

\subsection{Discrete-time RDCP}\label{subsec:discrete}
The random degree constrained process (RDCP) is usually studied on the complete~$n$-vertex graph $G^n=K_n$ and the papers~\cite{H24a,MSW,RW92,RW97,RW02,RW23,Se13,WW} define the RDCP process using a slightly different time parametrization: edges are added one by one at discrete time steps, see the definition below.
As the reader can guess, our main local weak convergence result Theorem~\ref{thm:local-limit1} carries over to the discrete-time version of the process (see Corollary~\ref{cor:local-limit-discrete} below), but this requires some extra~setup. 

\begin{definition}[RDCP in discrete time]\label{def:rdcp_discrete_time}
We define the discrete-time random degree constrained process $\bigpar{G^{n,k}_{\underline{p}}}_{k \ge 0}$, where $k \in \mathbb{N}$ is a step parameter and~$G^{n,0}_{\underline{p}}$ is the empty graph with vertex set~$V(G^n)$. 
For $k \in \mathbb{N}_+$ we then obtain $G^{n,k}_{\underline{p}}$ by adding exactly one new edge to $G^{n,k-1}_{\underline{p}}$. 
This new edge is chosen uniformly from the set of edges not in~$G^{n,k-1}_{\underline{p}}$, for which both end-vertices are unsaturated in $G^{n,k-1}_{\underline{p}}$. 
If no more edges can be added to the graph without violating any degree constraint, we arrive at the final graph. 
We denote by $M_n(\underline{p})$ the random number of edges in the final graph $G^{n,M_n(\underline{p})}_{\underline{p}}$. 
For mathematical convenience, we also define $G^{n,k}_{\underline{p}}:=G^{n,M_n(\underline{p})}_{\underline{p}}$ for $k > M_n(\underline{p})$. 
\end{definition}
The following claim requires no further justification.
\begin{claim}[Correspondence between RDCP in discrete and continuous time]\label{cl:final_graphs}
Note that $G^{n,k}_{\underline{p}}$ has exactly $k$~edges when~$k \leq M_n(\underline{p})$, 
Furthermore, $G^{n,k}_{\underline{p}}$ has the same distribution as $G^n_{\underline{p}}(\hat{\tau}_{n,k})$, where $\hat{\tau}_{n,k}$ denotes the time when the $k^{\text{th}}$ edge is added in the (c.\`a.d.l.\`a.g.\ version of) the continuous-time process $\bigpar{G^n_{\underline{p}}(t)}_{t \geq 0}$. 
In particular, the continuous- and discrete-time final graphs $G^n_{\underline{p}}(\infty)$ and $G^{n,M_n(\underline{p})}_{\underline{p}}$ have the same~distribution. 
\end{claim}

\begin{remark}[Time-change]\label{rem:time_change}
To translate results from the continuous-time RDCP to the discrete-time RDCP, 
by Claim~\ref{cl:final_graphs} it suffices to (i)~show concentration of the number of edges in~$G^n_{\underline{p}}(t)$, 
and (ii)~determine which time~$t=t(s)$ ensures that~$G^n_{\underline{p}}(t)$ has approximately~$sn$ edges, i.e., an average degree approximately equal to~$2s$. 
By the local weak convergence result of Theorem~\ref{thm:local-limit1}, this effectively reduces to studying the expected number of neighbors of the root in the random tree $G^\infty_{\underline{p}}(t)$, 
which is the content of the auxiliary result \refCl{cl:F(t)-inverse} below. 
\end{remark}

In the following~$D$ denotes a random variable with distribution $\underline{p}$ satisfying the mild Assumption~\ref{assump:degree_constraint_mild}. Let $\ev(D)=\sum_{k=1}^{\infty} k \cdot p_k \in \mathbb{R}_+ \cup \{+\infty\}$.

\begin{definition}[Mean number of neighbors of the root]\label{def:F(t)}
Given $ t \in \mathbb{R}_+ \cup \{+\infty\}$, let $F_{\underline{p}}(t)$ denote the expected number of neighbors of the root in $G^\infty_{\underline{p}}(t)$ (cf.\ Definition~\ref{def:rdcp-on-pwit}).
\end{definition}
\begin{claim}[Inverse of $F_{\underline{p}}(t)$]\label{cl:F(t)-inverse}
Let $\underline{p}$ satisfy the mild Assumption~\ref{assump:degree_constraint_mild}. The function $F_{\underline{p}}$ has an inverse function $F_{\underline{p}}^{-1}\colon \,[0, \ev(D)) \to \mathbb{R}_+$, which is continuous and strictly increasing.
Furthermore, $F_{\underline{p}}^{-1}(s) \to \infty$ as~$s \nearrow \ev(D)$. 
\end{claim}
We will prove Claim~\ref{cl:F(t)-inverse} in Section~\ref{subsec:bopint}.

We are now ready to formally state that our main local weak convergence result Theorem~\ref{thm:local-limit1} carries over to the discrete-time version of 
the RDCP, including the final graph (cf.~the~$\hat{t}=+\infty$ case of~Theorem~\ref{thm:local-limit1}). 
\begin{corollary}[Local weak convergence of RDCP in discrete time]\label{cor:local-limit-discrete}
Given~${s \in \left[0, \frac{\ev(D)}{2} \right)}$, assume that~$k_n(s) \in \mathbb{N}$ satisfies $\lim_{n \to \infty} k_n(s)/|V(G^n)|=s$. Let $\underline{p}$ satisfy the mild Assumption~\ref{assump:degree_constraint_mild}.
Then the discrete-time random degree constrained process $G_{\underline{p}}^{n, k_n(s) }$ on $G^n$ after~$k_n(s)$ steps
converges locally in probability to a random rooted tree $\bigpar{G^\infty_{\underline{p}}\bigpar{F_{\underline{p}}^{-1}(2s)}, \rt }$ as $n \to \infty$, i.e., 
\begin{equation*}
G_{\underline{p}}^{n, k_n(s) } \longrightarrow \Bigpar{G^\infty_{\underline{p}}\bigpar{F_{\underline{p}}^{-1}(2s)}, \rt } 
\quad \text{locally in probability as } n \to \infty,
\end{equation*}
where $\bigpar{G^\infty_{\underline{p}}(\hat{t}), \rt}$ is the same as in Theorem~\ref{thm:local-limit1}. 
Furthermore, the sequence of final graphs $G^{n,M_n(\underline{p})}_{\underline{p}}$ converges locally in probability
to $\bigpar{G^\infty_{\underline{p}}(\infty), \rt}$.
\end{corollary}

We will prove \refC{cor:local-limit-discrete} in \refS{subsec:local_limit_extension} by making the time-change argument of Remark~\ref{rem:time_change} rigorous. 
Although this local weak limit transfer from the continuous-time RDCP to the discrete-time RDCP is fairly straightforward and not surprising, 
we decided to include it here explicitly since the discrete-time version appears in many papers 
(so the formulation of \refC{cor:local-limit-discrete} may potentially be useful in applications).

\subsection{Critical time}\label{subsec:intro_tc}

Warnke and Wormald~\cite[Theorem~1.6]{WW} proved that if $\underline{p}$ satisfies $p_k=0$ for all $k > \Delta$, then the discrete-time RDCP on the complete graph $G^n=K_n$ undergoes a \textbf{phase transition}:
there exists a \textbf{critical time}~${t_c = t_c(\underline{p})}$ such that
\begin{itemize}
\item if $s<t_c(\underline{p})$ is fixed and~${n \gg 1}$, then the size of the largest connected component of the discrete-time RDCP graph $G_{\underline{p}}^{n, \lfloor n \cdot s \rfloor }$ is $\mathcal{O}(\ln(n))$ with high probability, but 
\item 
if $s>t_c(\underline{p})$ is fixed and $n \gg 1$, then $G_{\underline{p}}^{n, \lfloor n \cdot s \rfloor }$ has a giant connected component (i.e., a component of size~$\Theta(n)$) with high probability. 
\end{itemize}

A main result of this paper concerns the asymptotics of the critical time of the random degree constrained process. 
Warnke and Wormald~\cite{WW} used differential equations to identify the critical time $t_c = t_c(\underline{p})$ as the blow-up point of the susceptibility, i.e., the expected component size of a randomly chosen vertex (in the discrete-time RDCP). 
Using the time-change of \refC{cor:local-limit-discrete}, we can analogously characterize the critical time using the local weak limit of the RDCP, i.e., as the blow-up point of the susceptibility of the `limiting' random rooted tree $G^\infty_{\underline{p}}(\hat{t})$ introduced in Definition~\ref{def:rdcp-on-pwit}.

\begin{definition}[Critical time of continuous-time RDCP]\label{def:t-crit}
We define~$\hat{t}_c(\underline{p})$ as
\begin{equation}\label{eq:t_crit-def}
\hat{t}_c(\underline{p}) := \sup \left\{ \, \hat{t} \ge 0 \: : \: \ev \bigpar{|V(G^\infty_{\underline{p}}(\hat{t}))|} < \infty \, \right\}.
\end{equation}
\end{definition}

Our next result states that the global characterization \cite[Theorem~1.6]{WW} of the phase transition and the critical time of the discrete-time RDCP corresponds to our local characterization of the phase transition and the critical time of the continuous-time RDCP under the appropriate time-change.

\begin{proposition}[Equivalence of critical time definitions]\label{prop:tc-equivalence}
Let $\underline{p}$ satisfy the strict Assumption~\ref{assump:degree_constraint_strict}.
The time
\begin{equation}\label{eq:discrete-tc-def}
t_c(\underline{p}) := F_{\underline{p}}\bigpar{\hat{t}_c(\underline{p})}/2
\end{equation}
equals the critical time~$t_c(\underline{p})$ 
for the discrete-time RDCP determined in~\cite[Theorem~1.6]{WW}.
\end{proposition}

We will prove \refP{prop:tc-equivalence} in \refS{sec:branching_operator} using a new spectral characterization of the critical time~$\hat{t}_c(\underline{p})$ given by Theorem~\ref{thm:PF-eigenvalue}, 
as later discussed in Section~\ref{sss:additional_spectral}.

We will discuss the reasons for making the strict Assumption~\ref{assump:degree_constraint_strict} after the statement of Theorem~\ref{thm:PF-eigenvalue}, but let us note that the assumptions of ~\cite[Theorem~1.6]{WW} (where $t_c(\underline{p})$ is defined) also include that $\underline{p}$ satisfies $p_k=0$ for all $k > \Delta$. 
On the other hand, let us stress that the only point of this paper where we rely on the results of \cite{WW} is the proof of \refP{prop:tc-equivalence}; otherwise, our results are self-contained and our proof strategy differs from that of \cite{WW}.

To demonstrate the applicability of our local limit results, we shall next state a novel result \eqref{eq:discrete-tc:dprocess} about the critical time $t_c = t_c(d)$ of the \textbf{random \mbox{$d$-process}}, 
i.e., the discrete-time RDCP on the complete $n$-vertex host graph~$G^n=K_n$, where each vertex~$v$ has the same degree constraint~$d(v)=d \ge 3$, i.e., the special case of the RDCP where $\underline{p}$ satisfies~$p_d=1$.
Motivated by the desire to better understand analytical properties of the critical time, 
a problem of Warnke and Wormald~\cite[Section~6]{WW} asks for a proof of ${t_c(d) > 1/2}$ for any~$d \ge 3$. 
With the large~$d$ case in mind, this motivates us to study the critical time of the RDCP, when the typical degree constraint is large. 
\begin{notation}
Let $D^m,\, m \in \mathbb{N}$ denote a sequence of non-negative, integer-valued random variables. 
We write $D^m \Rightarrow \infty$ if $D^m$ converges in distribution to~$\infty$ as~$m \to \infty$, i.e., if~$\prob(D^m = k) \to 0$ as $m \to \infty$ for any $k \in \mathbb{N}_+$.
\end{notation}
Given~$\underline{p}^m = (p^m_k)_{k=2}^{\Delta_m}$, suppose that the sequence of random variables ${D^m \sim \underline{p}^m}$ satisfies ${D^m \Rightarrow \infty}$ as $m \to \infty$. 
If we run the discrete-time RDCP using the probability distribution $\underline{p}=\underline{p}^m$ (and a suitable upper bound~$\Delta=\Delta_m$, where $\Delta_m \to \infty$ as $m \to \infty$), 
then for large~$m$ the degree constraint of most vertices is very large, and thus very few vertices saturate before the phase transition occurs.
The discrete-time Erd\H{o}s--R\'enyi graph process (where there are no degree~constraints) is known to have critical time~$1/2$, 
so we expect the first order asymptotics $t_c(\underline{p}^m) \to 1/2$ as $m \to \infty$, which in fact follows from~\cite[Appendix~A]{WW}.
The following theorem gives the second order asymptotics of the critical time, 
by showing that $t_c(\underline{p}^m) \approx 1/2+\ev \bigpar{1/D^m!}/\mathrm{e}$ for large~$m$. 
\begin{theorem}[Asymptotics of critical times $\hat{t}_c$ and $t_c$]\label{thm:tc-asymptotics}
Let $\Delta_m \in \mathbb{N}_+$, and for each $m = 1,2,\dots$ let $\underline{p}^m = (p^m_k)_{k=2}^{\Delta_m}$ be a probability distribution on $\{\,2, \dots, \Delta_m \, \}$ that satisfies the strict Assumption~\ref{assump:degree_constraint_strict}. Moreover, let us assume that the random variables~$D^m \sim \underline{p}^m$ satisfy $D^m \Rightarrow \infty$ as $m \to \infty$. 
Then the continuous and discrete critical times~satisfy
\begin{align}
\label{eq:tc}
\frac{\hat{t}_c(\underline{p}^m)-1}{\frac{2}{\mathrm{e}} \cdot \ev \left( \frac{1}{D^m!} \right)} \; &\stackrel{m \to \infty}{\longrightarrow} \; 1,\\
\label{eq:discrete-tc}
\frac{t_c(\underline{p}^m)-\frac 12}{\frac{1}{\mathrm{e}} \cdot \ev \left( \frac{1}{D^m!} \right)} \; &\stackrel{m \to \infty}{\longrightarrow} \; 1.
\end{align}%
\end{theorem}
We will prove \refT{thm:tc-asymptotics} in \refS{sec:t_crit_asymp} using the asymptotic analysis of the solution of a second order linear differential equation that arises from a spectral characterization of~$\hat{t}_c(\underline{p})$ given by Theorem~\ref{thm:PF-eigenvalue}, 
which in turn exploits the multi-type branching process characterization of the local weak limit that we will discuss in Section~\ref{sss:additional_MTBP}. 

In the case of the random \mbox{$d$-process}, the degree constraint variables $d(v), \, v \in V(K_n)$ are all equal to the deterministic number $d$, i.e., the probability distribution $\underline{p}^d = (p^d_k)_{k=2}^{\Delta_d}$ satisfies $p^d_d=1$ and~$\Delta_{d}=d$, say. 
From~\eqref{eq:discrete-tc} of Theorem~\ref{thm:tc-asymptotics}, it follows that the critical time $t_c=t_c(d)$ of the random \mbox{$d$-process} satisfies
\begin{equation}\label{eq:discrete-tc:dprocess}
t_c(d) = \frac{1}{2} + \frac{1+o(1)}{\mathrm{e} \cdot d!} \quad \text{as} \quad d \to \infty.
\end{equation}
Hence there exists $d_0$ such that $t_c(d) > 1/2$ holds for $d \ge d_0$, answering the above-mentioned open problem of Warnke and Wormald for large~$d$. 
Conceptually speaking, this shows that (even if we measure time by the number of successfully added edges) the `giant' component of size~$\Theta(n)$ emerges significantly later in the random \mbox{$d$-process} than in the Erd\H{o}s--R\'enyi graph process, where there are no degree~constraints. 
On an intuitive level this might not be surprising, since vertices with large degree have a bigger role in the creation of the giant of the Erd\H{o}s--R\'enyi compared to vertices with a lower degree, and the degree constraints of the random $d$-process prevent the formation of such large degree hubs. However, $t_c(d) > 1/2$ appears to be harder to prove than one might suspect (which is one of the reasons why Warnke and Wormald asked about~this).

\subsection{Further characterizations of the local limit object and~\texorpdfstring{$\hat{t}_c$}{tc}}\label{subsec:auxiliary_results}
We have already identified the limit object of Theorem~\ref{thm:local-limit1} as the RDCP on the PWIT. In Section~\ref{subsec:auxiliary_results} we provide alternative equivalent descriptions of this limit object.
In Section~\ref{sss:additional_RTP} we define the notion of phantom saturation times and show that they solve a `recursion from infinity', i.e., we put the RDCP on the PWIT in the framework of recursive tree processes (RTPs) established by Aldous and Bandyopadhyay in \cite{AB}.
In Section~\ref{sss:additional_MTBP} we give another equivalent description of this limit object, in terms of a multi-type branching process (MTBP).
Finally, in Section~\ref{sss:additional_spectral} we use the MTBP description to give a useful spectral characterization of the critical time~$\hat{t}_c(\underline{p})$.

\subsubsection{Characterization of the limit object using recursive tree processes (RTPs)}\label{sss:additional_RTP}

In Section~\ref{sss:additional_RTP} we give a new characterization of the RDCP on the PWIT in terms of the framework of recursive tree processes (RTPs). The key idea is to use that the PWIT has a simple recursive structure: the subtrees of the PWIT that emanate from the children of the root are i.i.d.\ copies of the PWIT.

Recall the notion of the RDCP on the PWIT from Definition~\ref{def:rdcp-on-pwit}.
\begin{definition}[Phantom saturation time]\label{def:phantom-sat}
The \textbf{phantom saturation time} of vertex~$\ibf$ of the PWIT is the time when vertex~$\ibf$ saturates in the RDCP on the PWIT rooted at~$\ibf$ (i.e., we `delete' the edge that connects $\ibf$ to its parent, we focus on vertex~$\ibf$ and the PWIT of its descendants and run the RDCP construction on this copy of the PWIT). We denote the phantom saturation time of vertex~$\ibf$ by $T_\ibf$.
\end{definition}

In Definition~\ref{def:rde} and Claim~\ref{claim:rtp} we will describe a recursion (that `starts at infinity' and progresses towards the root) for phantom saturation times. Once we know the phantom saturation times, we will describe another recursion (that starts at the root and progresses towards infinity) for the (real) saturation times in Proposition~\ref{prop:rtp_satu}.

Note that the (real) saturation time of $\rt$ is equal to $T_\rt$, but the real saturation time of a non-root vertex of the PWIT might be smaller than its phantom saturation time.
In particular, the saturation time of a non-root vertex of $G^\infty_{\underline{p}}(\infty)$ (cf.\ Definition~\ref{def:rdcp-on-pwit}) is strictly smaller than its phantom saturation time (cf.\ Remark~\ref{remark:pst_sat_comparison} for more on this).

\begin{definition}[Recursion for phantom saturation times]\label{def:rde}
Let~$T$ denote an $\mathbb{R}_+ \cup \{ \infty\}$-valued random variable with a given distribution on~${\mathbb{R}_+ \cup \{ \infty\}}$, and let $T_1, \, T_2,\, \dots$ denote i.i.d.\ random variables with the same distribution as~$T$. Let ${\underline{\tau} = (\tau^j)_{j=1}^\infty}$ be a homogeneous unit intensity PPP and $D$ be a random variable with distribution $\underline{p}$ satisfying the mild Assumption~\ref{assump:degree_constraint_mild}. Let us assume that all of these random variables are independent of each other.
Let~$\underline{\tau}_\bullet$ denote the thinned point process where we only keep the $j^{\text{th}}$ point~$\tau^j$ of~$\underline{\tau}$ if ${\tau^j < T_j}$. 
Let $\tau^1_{\bullet}<\tau^2_{\bullet} <\dots$ denote the points of $\underline{\tau}_\bullet$, i.e., $\underline{\tau}_\bullet= (\tau^\ell_{\bullet})_{\ell=1}^\infty$. If~$\underline{\tau}_\bullet$ only has finitely many (say~$L$) points, let us define $\tau^\ell_{\bullet}=+\infty$ for all~${\ell > L}$.
Let us define the function~$\chi$~as
\begin{equation}\label{eq:rde-recursion}
\chi[\underline{\tau},D](T_1, T_2, \dots) := \tau_{\bullet}^D.
\end{equation}

\noindent 
We say that the distribution of~$T$ \textbf{solves the RDCP recursive distributional equation (RDE)} if the distribution of $\chi[\underline{\tau},D](T_1, T_2, \dots)$ is the same as the distribution of~$T$.
\end{definition}

\begin{remark}\label{remark:T_finite_infinite_exp}
It is easy to check that $\underline{\tau}_\bullet$ has infinitely many points if and only if ${\mathbb{E}(T)=+\infty}$. We will later see (cf.\ Remark~\ref{remark:rde_ode_sketch}, Claim~\ref{cl:phantom-sat-finite} and Lemma~\ref{lem:unique_rde}) that the RDCP RDE has a unique solution and that it satisfies ${\mathbb{P}(T<+\infty)=1}$ and ${\mathbb{E}(T)=+\infty}$.
\end{remark}

Recall from Definition~\ref{def:phantom-sat} that $T_{\ibf j}$ denotes the phantom saturation time of the $j^{\text{th}}$ child of vertex $\ibf$ of the PWIT. Recall from Definition~\ref{def:pwit} that $\underline{\tau}^{\ibf} = (\tau^{\ibf j})_{j=1}^\infty$, where 
 $\tau^{\ibf j}$ denotes the activation time of the edge that connects $\ibf$ to its $j^{\text{th}}$ child.

\begin{definition}[Thinned point processes on the PWIT]\label{def:thinned_pp_on_pwit}
Let $\underline{\tau}^{\ibf}_\bullet=(\tau^{\ibf \ell}_{\bullet})_{\ell=1}^\infty$ denote the thinned point process where we only keep the $j^{\text{th}}$ point $\tau^{\ibf j}$ of $\underline{\tau}^{\ibf}$ if $\tau^{\ibf j} < T_{\ibf j}$. 
\end{definition}

\begin{claim}[Phantom saturation times solve the RDE]\label{cl:phantom-sat-time-solves-rde}
The distribution of the phantom saturation times in the RDCP on the PWIT solves the RDE of \refD{def:rde}.
\end{claim}

\begin{proof} In this proof we work in the PWIT rooted at $\ibf$ (cf.\ Definition~\ref{def:phantom-sat}).
By \refD{def:rdcp-on-pwit}, we know that $T_{\ibf 1}, \, T_{\ibf 2}, \, \dots$ are i.i.d., since each child of vertex $\ibf$ has its own independent PWIT of descendants and the value of the phantom saturation time $T_{\ibf j}$ only depends on the PWIT of descendants of the $j^\text{th}$ child of $\ibf$, moreover, $T_{\ibf j}$ has the same distribution as~$T_\ibf$. If the edge connecting $\ibf$ to one of its children arrives before that vertex phantom saturates (i.e., if $\tau^{\ibf j} < T_{\ibf j}$), then we add that edge, otherwise we don't. We keep connecting $\ibf$ with new neighbors until it (phantom) saturates (i.e., until it acquires $d(\ibf)$ child neighbors). 
According to Definition~\ref{def:thinned_pp_on_pwit}, this happens at time $\tau^{\ibf d(\ibf) }_{\bullet}$. Thus $T_\ibf=\tau^{\ibf d(\ibf) }_{\bullet}$, and we obtain that $T_{\ibf} = \chi[\underline{\tau}^{\ibf},d(\ibf)](T_{\ibf 1}, T_{\ibf 2}, \dots)$ holds. This proves Claim~\ref{cl:phantom-sat-time-solves-rde}.
\end{proof}

In a \textbf{recursive tree process} (RTP) \cite{AB} the vertices have innovation and response variables such that if we know the innovation of a given vertex and the responses of its children, then we can determine the response of the vertex.

\begin{claim}[RTP representation of phantom saturation times]\label{claim:rtp}
 The structure of phantom saturation times of the RDCP on the PWIT can be represented as a recursive tree process (RTP), where the response variable of vertex $\ibf$ is its phantom saturation time $T_{\ibf}$ and the innovation of vertex $\ibf$ is 
 $(\underline{\tau}^{\ibf},d(\ibf))$, where 
 $\underline{\tau}^{\ibf} = (\tau^{\ibf j})_{j=1}^\infty$ is the point process of activation times of edges between vertex $\ibf$ and its children and $d(\ibf)$ is the degree constraint of vertex $\ibf$.
In particular, we have
\begin{equation}\label{eq:phantom-sat-recursion}
T_{\ibf} = \chi[\underline{\tau}^{\ibf},d(\ibf)](T_{\ibf 1}, T_{\ibf 2}, \dots).
\end{equation}
\end{claim}
\begin{proof}
We have already seen that \eqref{eq:phantom-sat-recursion} holds in the proof of \refCl{cl:phantom-sat-time-solves-rde}.
\end{proof}

\begin{remark}[Endogeny of the phantom saturation time RTP]
The recursive relation \eqref{eq:phantom-sat-recursion} determines the phantom saturation times of the $r^{\text{th}}$ generation using the phantom saturation times of generation $(r+1)$. At first, this recursion from infinity seems potentially ill-defined, but using that a chain of causality can only involve monotone decreasing edge labels (cf.\ Remark~\ref{rem:mon_dec_causality}), we can see that we don't have to look infinitely far if we want to determine the value of $T_{\ibf}$. More rigorously, \refCl{cl:rdcp-pwit-welldefined} implies that the phantom saturation times of all vertices are measurable with respect to the $\sigma$-algebra generated by all of the degree constraints and edge labels, thus by \cite[Definition 7]{AB} the RTP that appears in \refCl{claim:rtp} is \emph{endogenous}.
\end{remark}

\begin{remark}[The solution of the RDE]\label{remark:rde_ode_sketch} In Section~\ref{subsec:rtp_mtbp} we will show that the RDCP RDE (cf.\ Definition~\ref{def:rde}) identifies the distribution of phantom saturation times using the following ideas: let
$X_t := \sum_{\ell=1}^{\infty} \indic{\tau^\ell_{\bullet} \leq t} $ denote the number of children of $\rt$ at time~$t$ in the RDCP on the PWIT if we ignore the degree constraint of $\rt$. Then $X_t$ has $\mathrm{POI}(\lambda(t))$ distribution, where
$\lambda(t)=\int_0^t \mathbb{P}(T>s)\, \mathrm{d}s$. Since $T$ has the same distribution as $\chi[\underline{\tau},D](T_1, T_2, \dots)$, we have 
$\mathbb{P}(T>t)=\mathbb{P}(X_t<D)$. These observations, together with $\lambda'(t)=\mathbb{P}(T>t)$ and $\mathbb{P}(X_t<D)=\sum_{k=1}^{\infty} \mathbb{P}(D=k)\cdot \mathbb{P}(X_t \leq k-1) $ give the ODE 
\begin{equation}\label{lambda_ode_first_intro}
\lambda'(t)= \sum \limits_{k=1}^{\infty} p_k \cdot \sum \limits_{l=0}^{k-1} \mathrm{e}^{-\lambda(t)} \cdot \frac{\lambda(t)^l}{l!}, \qquad \lambda(0)=0,
\end{equation}
which identifies the function $\lambda(\cdot)$ and the distribution of phantom saturation times. In Lemma~\ref{lem:unique_rde} we will show that the solution of~\eqref{lambda_ode_first_intro} satisfies $\lambda'(t) \to 0$ and $\lambda(t) \to \infty$ as $t \to \infty$ and that this implies ${\mathbb{P}(T<+\infty)=1}$ and ${\mathbb{E}(T)=+\infty}$, as we have already mentioned in Remark~\ref{remark:T_finite_infinite_exp}.
\end{remark}

\begin{definition}[Saturation times in the PWIT]\label{def:sat_times_pwit}
Let us denote by $T^{\lozenge}_{\ibf}$ the saturation time of vertex $\ibf$ in the RDCP on the PWIT (cf.\ Definition~\ref{def:rdcp-on-pwit}). Let $\eta^{\lozenge}_{\ibf j}$ denote the indicator of the event that the edge between $\ibf$ and $\ibf j$ is added in the RDCP on the PWIT. Let $\eta^{\lozenge}_{\rt}:=0$.
\end{definition}
This way $\eta^{\lozenge}_{\ibf}$ is defined for each vertex $\ibf$ of the PWIT, and it can be interpreted as the indicator of the event that the edge between $\ibf$ and its parent is added in the RDCP on the PWIT (and we have $\eta^{\lozenge}_{\rt}=0$ since the root has no parent).

The recursion of Claim~\ref{claim:rtp} determines the phantom saturation times of the $r^{\text{th}}$ generation using the phantom saturation times of generation $(r+1)$ (and the innovations of the RTP). 
In the next proposition we recursively determine the saturation times of generation $(r+1)$ using the saturation times of the $r^{\text{th}}$ generation (and the innovations as well as the responses of the RTP).

\begin{proposition}[The RTP determines the saturation times]\label{prop:rtp_satu} Let $\ibf$ be a vertex of the PWIT. If we already know $\eta^{\lozenge}_{\ibf}$, then we can determine $T^{\lozenge}_{\ibf}$ and $\eta^{\lozenge}_{\ibf j}$ for each $j \in \mathbb{N}_+$ as follows. Let us define the point process $\underline{\tau}^{\ibf}_{\lozenge} = (\tau^{\ibf \ell}_{\lozenge})_{\ell=1}^\infty$ by
\begin{equation}\label{eq:tau_lozenge_def}
\underline{\tau}^{\ibf}_{\lozenge}:= \begin{cases}
\underline{\tau}^{\ibf}_{\bullet} & \text{if } \;\;\; \eta^{\lozenge}_{\ibf}=0, \\
\underline{\tau}^{\ibf}_{\bullet} \cup \{ \tau^{\ibf} \} & \text{if } \;\;\; \eta^{\lozenge}_{\ibf}=1,
\end{cases} 
\end{equation}
where the point process $ \underline{\tau}^{\ibf}_{\bullet} \cup \{ \tau^{\ibf} \}$ is obtained from $ \underline{\tau}^{\ibf}_{\bullet}$ by inserting the extra point $\tau^{\ibf}$ (i.e., the activation time of the edge that connects $\ibf$ to its parent). Then we have
\begin{equation}\label{eq:sat_time_explicit}
T^{\lozenge}_{\ibf}=\tau_{\lozenge}^{\ibf d(\ibf)},
\qquad
\eta^{\lozenge}_{\ibf j} = \indic{\tau^{\ibf j} < T_{\ibf j} } \cdot \indic{\tau^{\ibf j} \leq T^{\lozenge}_{\ibf}}.
\end{equation}
\end{proposition}

We will prove Proposition~\ref{prop:rtp_satu} in Section~\ref{subsec:rtp_mtbp}. 
\begin{remark}[About $T^{\lozenge}_{\ibf}$ and $T_{\ibf}$]\label{remark:pst_sat_comparison} Note that it follows from $T_\ibf=\tau^{\ibf d(\ibf) }_{\bullet}$ and $T^{\lozenge}_{\ibf}=\tau_{\lozenge}^{\ibf d(\ibf)}$ that $T^{\lozenge}_{\ibf} \leq T_\ibf$ for all $\ibf$. If $\eta^{\lozenge}_{\ibf}=0$ then 
$T^{\lozenge}_{\ibf} = T_\ibf$, but if $\eta^{\lozenge}_{\ibf}=1$ then 
$T^{\lozenge}_{\ibf} < T_\ibf$.
\end{remark}

\subsubsection{Characterization of the limit object as a MTBP}\label{sss:additional_MTBP}
In this subsection we will define a \textbf{multi-type branching process (MTBP)}, 
which has the same law as the limit object $G^\infty_{\underline{p}}(\hat{t})$ from Theorem~\ref{thm:local-limit1}. The types of our MTBP will be the phantom saturation times (cf.\ Definition~\ref{def:phantom-sat}). 
Our goal is to generate the labels of the edges and the types of vertices of generation~$(r+1)$ given the types of generation~$r$.
\begin{notation}\label{not:lambda_q_z}
Let us introduce the notation
\begin{align}
\label{eq:def-qk}
q_k &:= \sum \limits_{d=k}^\infty p_d, && k= 1,2,\dots, \\
\label{eq:def-zk}
z_k^t & := \mathrm{e}^{-\lambda(t)} \cdot \frac{\lambda(t)^k}{k!}\cdot p_{k+1}, && k = 0,1,2,\dots,
\end{align}
where $\lambda(t)$ is the solution of the initial value problem
\begin{align}
\lambda'(t) &= \mathrm{e}^{-\lambda(t)} \cdot \sum \limits_{k=0}^{\infty} \frac{\lambda(t)^k}{k!} q_{k+1}, \label{eq:lambda-ivp} \\
\lambda(0) &= 0.\label{eq:lambda-ivp-null}
\end{align}
\end{notation}
\begin{remark}[Two equivalent forms of the same ODE] 
One checks that the ODE \eqref{lambda_ode_first_intro} is the same as the ODE
\eqref{eq:lambda-ivp} by changing the order of summations in \eqref{lambda_ode_first_intro}, thus the two initial value problems have the same solution $\lambda(\cdot)$.
\end{remark}
\begin{claim}[Properties of phantom saturation times of RDCP on PWIT]\label{cl:phantom-sat-finite}
Let
$\underline{p}$ satisfy the mild Assumption~\ref{assump:degree_constraint_mild}.
The phantom saturation time~$T_\rt$ of the root~$\rt$ of the PWIT is almost surely finite, but $\ev(T_\rt) = \infty$. 
Moreover, $T_\rt$ has a probability density function $f(t)$ such that
\begin{equation}\label{eq:f-with-lambda-rdcp}
f(t) = -\lambda''(t) = \lambda'(t)\mathrm{e}^{-\lambda(t)} \cdot \sum \limits_{k=0}^\infty \frac{\lambda(t)^k}{k!} \cdot p_{k+1} \stackrel{\eqref{eq:def-zk}}{=} \lambda'(t) \cdot \sum \limits_{k=0}^\infty z_k^t,
\end{equation}
where $\lambda(t)$ is the solution of the initial value problem $\{\eqref{eq:lambda-ivp},\, \eqref{eq:lambda-ivp-null}\}$.
\end{claim}

We will prove \refCl{cl:phantom-sat-finite} in \refS{subsec:rtp_mtbp}.
Note that $\prob(T_\rt < \infty) = 1$ (see \refCl{cl:phantom-sat-finite}) and \refT{thm:local-limit1} together imply that the fraction of unsaturated vertices in the final graph~$G^n_{\underline{p}}(\infty)$ goes to $0$ as $n \to \infty$. 
In concrete words, this means that asymptotically almost all vertices of the final graph are saturated.

\begin{notation}\label{not:f(t)-A}
In the following $f(t)$ denotes the probability density function of $T_{\rt}$, cf.~\eqref{eq:f-with-lambda-rdcp}.
We introduce the planar domain
\begin{equation}\label{eq:A-def}
A_{t_0} := \{\,(t,s) \in \mathbb{R}_+^2 \, | \, 0 \le t \le t_0,\, t \le s \,\} \quad \text{for } t_0 \in \mathbb{R}_+.
\end{equation}
\end{notation}

\begin{remark}[The role of $A_{t_0}$] If $\{\ibf, \ibf j\}$ is an edge of~$G^\infty_{\underline{p}}(\infty)$ (cf.~Definition~\ref{def:rdcp-on-pwit}) 
and ${T_{\ibf}=t_0}$, then the pair 
$(\tau^{\ibf j}, T_{\ibf j})$ falls in the domain~$A_{t_0}$, as we now explain. It is enough to verify that $\tau^{\ibf j} \leq T_{\ibf}$ and $ \tau^{\ibf j} \leq T_{\ibf j}$ both hold. We have $T^{\lozenge}_{\ibf} \leq T_{\ibf}$ and $T^{\lozenge}_{\ibf j} \leq T_{\ibf j}$ (cf.~Remark~\ref{remark:pst_sat_comparison}), i.e., the saturation time of a vertex is less than or equal to its phantom saturation time. Now $\tau^{\ibf j} \leq T^{\lozenge}_{\ibf} \wedge T^{\lozenge}_{\ibf j}$ must hold by the rules of the RDCP: the edge $\{\ibf, \ibf j\}$ can only be added if both of its end-vertices are unsaturated when it activates.
\end{remark}

Now we redefine $\tau^{\mathbf{i}j}$ and $T_{\mathbf{i}}$ in the context of multi-type branching processes, but this clash between old and new notation should not cause much confusion, because Theorem~\ref{thm:mtbp-reproduces} below states that the new tree decorated with the new notation has the same law as the old tree decorated with the old notation.

\begin{definition}[MTBP tree that will reproduce RDCP on PWIT]\label{def:mtbp}
We define an \textbf{edge-labelled MTBP} where the vertices are denoted by finite strings of integers (similarly to \refD{def:pwit}), the type of a vertex~$\ibf$ is a positive real number $T_{\ibf}$ and the label of the edge $(\ibf,\ibf j)$ is denoted by the positive real number $\tau^{\ibf j}$. Moreover, each vertex~$\ibf$ has an integer-valued degree constraint $d(\ibf)$.

The type of the root $T_\rt$ has probability density function~$f(t)$, where $f$ is defined as in~\eqref{eq:f-with-lambda-rdcp}. 
If the type of the root is~$T_\rt=t_0$, then we generate its degree constraint $d(\rt)$ with probability mass function
\begin{equation}\label{eq:d-mass}
p_k^{t_0} := \frac{z_{k-1}^{t_0}}{\sum \limits_{l=0}^\infty z_l^{t_0}} \quad \text{for } k =1,2,\dots,
\end{equation}
where $z_k^t$ is defined in \eqref{eq:def-zk}.

Conditional on $T_\rt=t_0$ and $d(\rt)=d$, the edge label $\tau^d$ of the edge that connects the root to the $d^{\text{th}}$ child of the root is $\tau^d=t_0$ and the type $T_d$ of the $d^{\text{th}}$ child of the root is a random variable with probability density function
\begin{equation}\label{eq:dth-dens}
\frac{f(s)}{\int \limits_{t_0}^\infty f(u) \, \mathrm{d}u} \cdot \indic{s>t_0}.
\end{equation}
Given $T_\rt=t_0$ and $d(\rt)=d$, let $(\breve{\tau}^j, \breve{T}_j),\, j=1,\dots, d-1$ denote conditionally i.i.d.\ $\mathbb{R}^2_+$-valued random variables that are also independent of $T_d$, and conditionally on $T_\rt=t_0$ and $d(\rt)=d$ let the pair $(\breve{\tau}^j, \breve{T}_j)$ have joint probability density function
\begin{equation}\label{eq:mtbp-joint-dens}
g_{t_0}(t,s) = \frac{f(s)}{\lambda(t_0)} \cdot \indic{(t,s) \in A_{t_0}},
\end{equation}
where $A_{t_0}$ is defined in \eqref{eq:A-def}. Now let $\sigma\colon \{1,\dots,d-1\} \to \{1,\dots,d-1\}$ denote the (almost surely unique) permutation for which $\breve{\tau}^{\sigma(1)}<\dots<\breve{\tau}^{\sigma(d-1)}$ and let $\tau^j:=\breve{\tau}^{\sigma(j)}$ and $T_j:=\breve{T}_{\sigma(j)}$ for all $j=1,\dots,d-1$.

For further generations of the MTBP, if we already know the type of a vertex $\ibf$ (say it is equal to $t_0$), we can generate its degree constraint $d(\ibf)$ with mass function \eqref{eq:d-mass}. Then this vertex will have $d(\ibf)-1$ children. We generate the pairs $(\breve{\tau}^{\ibf j}, \breve{T}_{\ibf j})$ for these children with joint density function \eqref{eq:mtbp-joint-dens} independently of each other and of the random variables generated earlier, and again we define $\left((\tau^{\ibf j}, T_{\ibf j})\right)_{j=1}^{d(\ibf)-1}$ by rearranging the list $\left((\breve{\tau}^{\ibf j}, \breve{T}_{\ibf j})\right)_{j=1}^{d(\ibf)-1}$ of pairs in a way that their first coordinates are in increasing order.
\end{definition}

\begin{definition}[Connected component of $\rt$ at time $\hat{t}$]\label{def:mtbp-at-t}
Fix~$\hat{t} \in \mathbb{R}_+ \cup \{+\infty\}$. Let us consider the subgraph of the MTBP tree described in \refD{def:mtbp} that consists of the edges with label less than $\hat{t}$. 
We denote by $\mathcal{M}_{\underline{p}}(\hat{t})$ the connected component in this subgraph that contains the root, 
decorated with types $T_\ibf$, degree constraints~$d(\ibf)$ and edge labels~$\tau^{\ibf}$.
\end{definition}

Note that by \refD{def:mtbp}, it follows that $\mathcal{M}_{\underline{p}}(\hat{t})$ is also (the family tree of) a multi-type branching process, 
since the type of a vertex determines the distribution of the number of its children and their types.

\begin{theorem}[MTBP gives an alternative characterization of the local weak limit]\label{thm:mtbp-reproduces}
Fix~$\hat{t} \in \mathbb{R}_+ \cup \{+\infty\}$. 
Let
$\underline{p}$ satisfy the mild Assumption~\ref{assump:degree_constraint_mild}.
Then 
$\mathcal{M}_{\underline{p}}(\hat{t})$ with types~$T_\ibf$ and edge labels~$\tau^{\ibf}$ 
has the same distribution as 
$G^\infty_{\underline{p}}(\hat{t})$ from \refD{def:rdcp-on-pwit} (that appears as the limit object in Theorem~\ref{thm:local-limit1}) 
with phantom saturation times~$T_\ibf$ and edge labels~$\tau^{\ibf}$.
\end{theorem}

We will prove \refT{thm:mtbp-reproduces} in \refS{subsec:rtp_mtbp}. 
Note that in the case of the RTP (cf.\ Claim~\ref{claim:rtp}) one recursively determines phantom saturation times of the $r^{\text{th}}$ generation from generation $(r+1)$, while in the case of the MTBP one recursively generates the phantom saturation times (i.e., the types) of generation $(r+1)$ with the help of the $r^{\text{th}}$ generation. 
The proof of Theorem~\ref{thm:mtbp-reproduces} amounts to showing that the outputs produced by these two definitions have the same law.

\subsubsection{Characterization of critical time in terms of MTBP}\label{sss:additional_spectral}
In this subsection we exploit that the limit object $G^\infty_{\underline{p}}(\hat{t})$ from Theorem~\ref{thm:local-limit1}
has a multi-type branching process description (cf.\ Theorem~\ref{thm:mtbp-reproduces}) in order to give a new spectral characterization of the critical time $\hat{t}_c(\underline{p})$ (cf.\ Definition~\ref{def:t-crit}).

\begin{definition}[Branching operator]
Let $S$ be the type space of a MTBP and $\varphi\colon S \to \mathbb{R}$ be a bounded, measurable function defined on the type space. We define the \textbf{branching operator} $B$ of the MTBP as follows. We define $B\varphi \colon S \to \mathbb{R}$ such that for any $t_0 \in S$ we~have
\begin{equation}\label{eq:branch-op-def}
(B\varphi)(t_0) = \ev \left( \sum \limits_{i=1}^M \varphi(t_i) \right),
\end{equation}
where $M$ is the (random) number of children of a vertex with type $t_0$ in the MTBP, and $t_1, \, t_2, \, \dots, \, t_M$ are the (random) types of these children.
\end{definition}

\begin{definition}[Branching operator at time $\hat{t}$]\label{def:branching-operator-at-t}
Let
$\underline{p}$ satisfy the mild Assumption~\ref{assump:degree_constraint_mild}.
Recall the notion of the MTBP tree~$\mathcal{M}_{\underline{p}}(\hat{t})$ that we introduced in \refD{def:mtbp-at-t}. Recall from Definition~\ref{def:mtbp} that the offspring distribution of the root is different from that of any other (i.e., non-root) vertex of the MTBP.
We denote the branching operator of non-root vertices of the MTBP tree~$\mathcal{M}_{\underline{p}}(\hat{t})$ by $B_{\hat{t}}$.
\end{definition}

In Section~\ref{subsec:bopint} we will show that $B_{\hat{t}}$ is an integral operator of form 
\begin{equation}\label{eq:bo_integral_intro}
\left(B_{\hat{t}} \varphi\right) (t) = \frac{\sum_{k=1}^{\infty} (k-1)\cdot p_k^t} {\lambda(t)} \int \limits_0^\infty f(s) \cdot \left(\hat{t} \wedge t \wedge s \right) \cdot \varphi(s) \, \mathrm{d}s ,
\end{equation}
where $\lambda(\cdot)$, $f(\cdot)$ and $p_k^t$ were defined in Notation~\ref{not:lambda_q_z}, \eqref{eq:f-with-lambda-rdcp} and \eqref{eq:d-mass}, respectively.

\begin{definition}[Sub- and supercritical MTBP]\label{def:critical}
We say that a MTBP tree~$\mathcal{M}$ is \textbf{subcritical} if its susceptibility is finite, i.e., if the expected size $\ev( |\mathcal{M}| )$ of the branching process tree $\mathcal{M}$ is finite. 
We say that a MTBP tree~$\mathcal{M}$ is \textbf{supercritical} if it survives with positive probability, i.e., if $\prob( |\mathcal{M}| = \infty ) > 0$. 
\end{definition}

The following theorem gives a new and useful spectral characterization of the critical time $\hat{t}_c(\underline{p})$ defined in~\eqref{eq:t_crit-def}, 
in terms of the branching operator $B_{\hat{t}}$ of the multi-type branching process tree $\mathcal{M}_{\underline{p}}(\hat{t})$ (cf.\ Definition~\ref{def:branching-operator-at-t}).
\begin{theorem}[Spectral characterization of critical time~$\hat{t}_c$]\label{thm:PF-eigenvalue}
Let 
$\underline{p}$ satisfy the strict Assumption~\ref{assump:degree_constraint_strict}.
For any $\hat{t} \in \mathbb{R}_+$ we have:%
{\vspace{-0.25em}\begin{enumerate}
\itemsep 0.125em \partopsep=0pt \parsep 0em 
\item The branching operator $B_{\hat{t}}$ is self-adjoint with respect to $L^2(\mathbb{R}_+, \, \rho)$ with an appropriate choice of the measure $\rho$ (cf.\ Definition~\ref{def:auxiliary_rho}), and the $L^2(\mathbb{R}_+, \, \rho)$ operator norm $\norm{B_{\hat{t}}}$ of $B_{\hat{t}}$ is~finite.
\item \label{item:subcrit_eigen_statement} If $\norm{B_{\hat{t}}} < 1$, then $\mathcal{M}_{\underline{p}}(\hat{t})$ is subcritical (cf.~Definition~\ref{def:critical}).
\item \label{item:supcrit_eigen_statement} If $\norm{B_{\hat{t}}} > 1$, then $\mathcal{M}_{\underline{p}}(\hat{t})$ is supercritical (cf.~Definition~\ref{def:critical}).
\item The function $\hat{t} \mapsto \norm{B_{\hat{t}}}$ is strictly increasing and continuous on $[0,+\infty)$.
\item We have $\bignorm{B_{\hat{t}_c(\underline{p})}}=1$, i.e., the norm of the branching operator at the critical time~$\hat{t}_c(\underline{p})$ equals~one.
\vspace{-0.25em}\end{enumerate}}%
\end{theorem}

We will prove \refT{thm:PF-eigenvalue} in Section~\ref{sec:branching_operator}, but some of the technical ingredients (specifically, the proof of the Hilbert--Schmidt property of $B_{\hat{t}}$ and the proof of statements~\ref{item:subcrit_eigen_statement}.\ and~\ref{item:supcrit_eigen_statement}.\ of \refT{thm:PF-eigenvalue}) are deferred to Appendix~\ref{appendix_main}. 
These are the proofs where we use that $p_k =0$ holds for any $k > \Delta$, as stated in the strict Assumption~\ref{assump:degree_constraint_strict}. 

Note that~$\norm{B_{\hat{t}}}$ is the principal eigenvalue of $B_{\hat{t}}$. We will give an equivalent characterization of this eigenvalue (and the corresponding eigenfunction) in terms of a Sturm--Liouville problem (i.e., a second order linear differential equation with well-chosen boundary conditions), cf.\ Lemma~\ref{lem:w-properties}. In Remark~\ref{rem:p1_zero} we explain why we assume $p_1=0$ (as stated in the strict Assumption~\ref{assump:degree_constraint_strict}) in this proof.
This ODE-based characterization of $\norm{B_{\hat{t}}}$ provides us with a new characterization of $\hat{t}_c(\underline{p})$ that will serve as a tool in the proof of our results about the asymptotics of $\hat{t}_c(\underline{p})$ stated in Theorem~\ref{thm:tc-asymptotics}. 

We point out that our ODE-based characterization of the critical time (in terms of the MTBP) is fundamentally different from the ODE-based characterization of the critical time given in \cite[Section~3.1]{WW}, where $t_c$ is shown to be equal to the blow-up point of a system of differential equations that characterizes the susceptibility. 

\begin{remark}[Multiple giant components are possible]
Proposition~\ref{prop:tc-equivalence} and Theorem~\ref{thm:PF-eigenvalue} together imply that if $\hat{t} > \hat{t}_c(\underline{p})$, then the local limit~$G^\infty_{\underline{p}}(\hat{t})$ has infinitely many vertices with positive probability. 
Perhaps surprisingly, there are high degree almost regular graph sequences~$\{\,G^n\,\}$ where the fraction of vertices in the largest connected component of the RDCP graph $G^n_{\underline{p}}(\hat{t})$ does not converge in probability to ${\mathbb{P}( |V(G^\infty_{\underline{p}}(\hat{t}))|=+\infty)}$, as one might expect. 
For example, this convergence fails if we define $G^n$ to be the disjoint union of two copies of $K_n$ connected by a single edge, 
in which case for $\hat{t} > \hat{t}_c(\underline{p})$ the RDCP graph $G^n_{\underline{p}}(\hat{t})$ will typically have two `giant' components of size~$\Theta(n)$ by the results of Warnke and Wormald~\cite{WW} (one in each copy of~$K_n$). 
In fact, this convergence can even fail for complete host graphs~$G^n=K_n$ when every vertex has degree constraint~$2$ (i.e., $\underline{p} = (p_k)_{k=2}^\Delta$ with~$p_2=1$): 
in this case equation~(4.12) in~\cite[Section~4.2]{WW} implies that the final graph $G^n_{\underline{p}}(\infty)$ can have multiple `giant' components of size $\Theta(n)$ with positive probability. 
\end{remark}

\subsection{Related work}\label{subsec:related}
The concept of local weak convergence of finite graphs was introduced in the early 2000s by Benjamini and Schramm~\cite[Section~1.2]{BS} and also by Aldous and Steele~\cite[Section~2.3]{AS04}, 
with some of the underlying ideas tracing back to the early 1990s, see~\cite{A91}.
Below we we briefly discuss some (well-known or recent) results on the topic; for a more detailed overview of local weak convergence we refer 
to the textbook~\cite{H24b} by van der Hofstad and the expository article~\cite{BBSY} by Banerjee, Bhamidi, Shen and Young. 

\smallskip 

\textbf{Classical random graphs.} The local weak limits of many random graph models have been studied for a long time (though the results were originally not phrased that way). 
It is a classical result that the Erd\H os--R\'enyi graph with average degree~$\lambda$ converges locally in probability to a branching process with a Poisson offspring distribution~$\text{POI}(\lambda)$ with mean~$\lambda$ (see e.g.~\cite[Section~2.4.5]{H24b} or~\cite[Section~3.5.1]{Bor}). 
Another well-known result concerns random graphs with a given degree sequence: 
under some regularity conditions, the configuration model~\cite{Bol} converges locally in probability to a unimodular branching process tree (see e.g.~\cite[Section~4.2]{H24b}, \cite[Section~3.5.2]{Bor} and~\cite[Theorem~25]{BR15}). 
Bollob\'as, Janson and Riordan \cite{BJR} studied the phase transition in so-called inhomogeneous random graphs: their results imply that the local weak limit of sequences of such graphs is the family tree of a multi-type branching process (see also~\cite[Section~3.5]{H24b} and~\cite[Lemma~11.11]{BJR}).

\smallskip 

\textbf{Spanning trees.} The local weak limit of spanning trees has rich literature. It follows from \cite[Theorem~3]{G80} that the local weak limit of the uniform spanning tree of~$K_n$ is a critical Galton--Watson tree with Poisson offspring distribution, conditioned to survive forever.
Nachmias and Peres~\cite[Theorem~1.4]{NP} show that the local weak limit is the same if the host graphs form a high degree almost regular graph sequence (as in Definition~\ref{def:hdar}).
Hladk\'y, Nachmias and Tran~\cite[Theorem~1.3]{HNT} showed that (a multi-type version of) this result remains true when $K_n$ is replaced by a dense graph sequence that converges to a graphon.
D'Achille, Enriquez and Melotti~\cite[Theorem~1]{DEM} studied the local weak limit of $\lambda_n$-massive spanning forest of $K_n$, where $\lambda_n$ is a given sequence: they showed that the limit depends on the growth rate of $\lambda_n$. 

\smallskip 

\textbf{Preferential attachment models.} Local weak limits of the uniform and the preferential attachment models have also been investigated. These models are very natural, the uniform attachment tree (or random recursive tree) occurred already in \cite{DKS}, while the preferential attachment model was introduced by Barab\'asi and Albert in \cite{BA}, but it has already appeared in~\cite{Y} and later was observed in different contexts in \cite{P, S55}. The limit of the uniform attachment tree was studied in \cite[Section~3.2]{A91} (see also \cite[Theorem~3.10]{BBSY}). 
 Berger, Borgs, Chayes and Saberi defined a MTBP, the P\'olya-point graph model (see \cite[Section~2.3]{BBCS}) and \cite[Theorem~2.2]{BBCS} shows that this branching process is the local weak limit of the preferential attachment model. Rudas, T\'oth and Valk\'o studied the limit of the tree versions of these models~\cite[Theorem~2]{RTV}, i.e., each new vertex is only connected to one other vertex of the graph. These results were further generalized by Banerjee, Deka and Olvera-Cravioto \cite[Corollary 1.1]{BDO} for collapsed branching processes and by Garavaglia, Hazra, van der Hofstad and Ray~\cite[Theorem~1.5]{GHHR} for the case when the new vertex is connected to the graph with random number of new edges. See also~\cite[Theorems~5.8,~5.26]{H24b} for the local weak limit of different versions of the preferential attachment model.

\smallskip 

\textbf{Local weak limits of other random graphs.} The local weak limit of uniformly chosen connected and $2$-connected planar graphs \cite[Theorem~1.1 and Theorem~1.2]{S21}, and of spatial inhomogeneous random graphs \cite[Theorem~1.11]{HHM} have also been determined. The mean field forest fire model \cite{RT} at a fixed time is in fact an inhomogeneous random graph (in the sense of~\cite{BJR}), and thus the local weak limit is a multi-type branching process~\cite{CRY}, just like in the case of inhomogeneous random graphs~\cite{BJR} mentioned above.

\smallskip

\textbf{PWIT.} One characterization of our local limit result uses the so-called Poisson weighted infinite tree (PWIT), which is a fundamental limit object in the theory of randomly weighted graphs. It was introduced by Aldous~\cite[Section~3.1]{A92} in the early 1990s, and was first called PWIT in~\cite[Section~4.1]{A01}. It is the local weak limit of~$K_n$ decorated with independent and exponentially distributed edge weights \cite{A92} (we tried to summarize this idea in Remark~\ref{rem:mot_pwit} above). The PWIT has several applications in the asymptotic characterization of the solutions of various combinatorial optimization problems on large random weighted graphs, e.g.\ for random assignment problems \cite{A92, SS}, random matrices \cite{BCC11a, BCC11b, BC, J18}, invasion percolation \cite{AGK}, minimal and stable matchings \cite{AB, HMP} and minimal spanning trees \cite{S02}. In particular, the local weak limit of the minimal spanning tree of $K_n$ with i.i.d.\ edge weights \cite[Theorem~1.1]{A13} is constructed using the PWIT (this result is a generalization of \cite[Theorem~1]{A90}). Branching random walk on the PWIT \cite{AF}, the length of longest path in subgraphs of the PWIT \cite{G16} as well as wired minimal spanning forests of the PWIT \cite{AS23,NT} have also been studied. A general overview of the PWIT with applications can be found in the survey~\cite{AS04}.

\subsection{Outline of the rest of this paper}\label{subsec:outline_of_rest_of_paper}

In Section~\ref{sec:section_local_limit} we prove our local limit result using graph explorations and couplings. In Section~\ref{subsec:rtp_mtbp} we prove the RTP and MTBP characterizations of RDCP on the PWIT.
In Section~\ref{sec:branching_operator} we study the branching operator of our MTBP and prove our spectral characterization of the critical time (deferring some of the proofs to Appendix~\ref{appendix_main}). Finally, in Section~\ref{sec:t_crit_asymp} we prove our result about the asymptotic expansion of the critical time.

\section{Proof of local weak limit result}\label{sec:section_local_limit}

\paragraph{}In this section we prove Theorem~\ref{thm:local-limit1}, i.e., that the local weak limit of the RDCP on~$G^n$ is the RDCP on the PWIT (see \refD{def:rdcp-on-pwit}). 
First, we heuristically outline our two-phase exploration algorithm on $G^n$ that only explores the activation times in a small neighborhood of a typical vertex $v$ (but the explored part will be sufficient for us to determine the $R$-neighborhood of $v$ in $G^n_{\underline{p}}(\hat{t})$). We make the definition of the two-phase exploration rigorous in \refS{subsec:two_phase_algo}. In \refS{subsec:coupling_to_rbrw} we define another process called the regularized branching random walk (RBRW) and construct a coupling of the two-phase exploration processes and the RBRW. Then in \refS{subsec:coupling_properties} we state the main properties of this coupling. In \refS{subsec:properties_rbrw_tree} we study the properties of the output of the RBRW, and in \refS{subsec:alarm_prob} we show that the probability that the two coupled models differ is small. Then we can deduce Theorem~\ref{thm:local-limit1} for $\hat{t}<\infty$. Finally, in \refS{subsec:local_limit_extension} we extend our proof to the final graph (i.e., when $\hat{t}=\infty$), and also prove \refC{cor:local-limit-discrete}.

\medskip 

In the following, $G^n_{\underline{p}}(\hat{t})$ denotes the RDCP graph on $G^n$ at time $\hat{t}$ (cf.\ Definition~\ref{def:rdcp}), where the sequence of the base graphs $\{\,G^n\,\}_{n=1}^\infty$ satisfies the high degree almost regularity condition (\refD{def:hdar}) and $\underline{p}$ satisfies the mild Assumption~\ref{assump:degree_constraint_mild}. We denote the vertex set and the edge set of $G^n$ by $V(G^n)$ and $E(G^n)$, respectively.

\medskip 

Let us fix a radius $R \in \mathbb{N}$, a point $0 \le \hat{t} < \infty$ in time, an error probability threshold ${\varepsilon>0}$ and a graph sequence $\{ \,G^n \, \}_{n=1}^\infty$. We will show that there exists $n_0 = n_0(R,\, \hat{t},\, \varepsilon,\, \{ \, G^n \, \})$ such that for any $n \ge n_0$ if we pick a vertex uniformly in $G^n_{\underline{p}}(\hat{t})$ and consider its $R$-radius neighborhood in $G^n_{\underline{p}}(\hat{t})$ (cf.\ Definition~\ref{def:rdcp}), then it can be coupled with the $R$-radius neighborhood of the root of $G_{\underline{p}}^\infty(\hat{t})$ (cf.\ Definition~\ref{def:rdcp-on-pwit}) in a way that with probability at least ($1-\varepsilon$) they will be isomorphic.

\begin{remark}[Heuristic description of the two-phase exploration]\label{rem:heu_two_phase}
Recall from Remark~\ref{rem:mot_pwit} that if $n$ is large, then the local structure of `early' edges of~$G^n$ is similar to the structure of the PWIT.
However, this similarity is not perfect, so if we aim for a perfect coupling of the explorations in the two settings, then it is crucial to explore only a small part of these graphs in order to minimize the chance that we encounter a difference between the two coupled worlds as we explore them. On the other hand, the structure that we explore around a uniformly chosen vertex $v$ of $G^n$ should contain enough information to allow us to identify the $R$-neighborhood of~$v$ in~$G^n_{\underline{p}}(\hat{t})$. Our exploration has two phases. The first phase is called the \textbf{breadth-first search (BFS)} phase: we explore the $R$-neighborhood of~$v$ in the subgraph of $G^n$ spanned by edges with activation time at most $\hat{t}$. We call this graph the BFS graph. The $R$-neighborhood of $v$ in~$G^n_{\underline{p}}(\hat{t})$ will surely be a subgraph of the BFS graph, but to be able to decide exactly which edges of the BFS graph are added, we also want to figure out whether the vertices on the boundary of the BFS graph saturate before $\hat{t}$ and if they do, we want to figure out when. This is what we do in the second phase of our exploration which we call the \textbf{monotone decreasing phase}. The strategy is analogous to the one outlined in Remark~\ref{rem:mon_dec_causality}: from each vertex on the boundary of the BFS graph we start a new exploration, but we only explore edges that activate earlier than their parent edge, since only these edges can affect our decision when we try to add the parent edge. This way the labels of edges decrease monotonically along a path that we explore.
We refer the reader to Figure~\ref{fig:two_phase} for a rough image of the two-phase exploration.
Somewhat naively we can estimate the size of the graph that we explore starting from $v$ by $ \sum_{r=1}^R \hat{t}^{r} + \hat{t}^R \cdot \mathrm{e}^{\hat{t}}$, where $\mathrm{e}^{\hat{t}}=\sum_{\ell=0}^{\infty} \hat{t}^\ell/\ell!$, cf.\ Remark~\ref{rem:mon_dec_causality}.
The rigorous bound will be worse than this, but it will be enough to guarantee that the explored region is `small enough' if we fix $R<+\infty$ and $\hat{t}<+\infty$. 
\end{remark}

\begin{figure}[t]
\centering
\begin{tikzpicture}[font=\small, node distance={1.0cm}, thick, main/.style = {draw, circle}]

\draw[dashed, gray5] (0.0, 0.0) -- (6.0, 3.0);
\draw[dashed, gray5] (0.0, 0.0) -- (6.0, -3.0);

\draw[dashed, gray5] (1.0, -0.5) -- (1.0, 0.5);
\draw[dashed, gray5] (2.0, -1.0) -- (2.0, 1.0);
\draw[dashed, gray5] (3.0, -1.5) -- (3.0, 1.5);
\draw[dashed, gray5] (4.0, -2.0) -- (4.0, 2.0);
\draw[dashed, gray5] (5.0, -2.5) -- (5.0, 2.5);
\draw[dashed, gray5] (6.0, -3.0) -- (6.0, 3.0);
\node[label=below:1] at (1.0, -0.4) {};
\node[label=below:2] at (2.0, -0.9) {};
\node[label=below:{\rotatebox{-25}{\raisebox{1pt}{$\cdots$}}}] at (4, -1.9) {};
\node[label=below:$R$] at (6, -2.9) {};

\draw[dashed, gray5] (8.1,2) ellipse (2.1 and 0.9);
\draw[dashed, gray5] (7.1,0) ellipse (1.1 and 0.35);
\draw[dashed, gray5] (7.6,-2) ellipse (1.6 and 0.8);

\node[main, fill=black, scale=0.3, label=below:$\varnothing$] (v0) at (0,0) {};
\node[main, fill=black, scale=0.3] (v11) at (1.0, 0.0) {};
\node[main, fill=black, scale=0.3] (v21) at (2.0, 0.5) {};
\node[main, fill=black, scale=0.3] (v22) at (2.0, -0.5) {};
\node[main, fill=black, scale=0.3] (v31) at (3.0, 0.0) {};
\node[main, fill=black, scale=0.3] (v32) at (3.0, -1.0) {};
\node[main, fill=black, scale=0.3] (v41) at (4.0, 1) {};
\node[main, fill=black, scale=0.3] (v42) at (4.0, 0) {};
\node[main, fill=black, scale=0.3] (v43) at (4.0, -1) {};
\node[main, fill=black, scale=0.3] (v51) at (5.0, 1.5) {};
\node[main, fill=black, scale=0.3] (v52) at (5.0, 0.5) {};
\node[main, fill=black, scale=0.3] (v53) at (5.0, -1.5) {};
\node[main, fill=black, scale=0.3] (vR1) at (6, 2) {};
\node[main, fill=black, scale=0.3] (vR2) at (6, 0) {};
\node[main, fill=black, scale=0.3] (vR3) at (6, -2) {};

\draw[black] (v0) -- (v11);
\draw[black] (v11) -- (v21);
\draw[gray5] (v11) -- (v22);
\draw[black] (v22) -- (v31);
\draw[gray4] (v22) -- (v32);
\draw[black] (v31) -- (v41);
\draw[gray2] (v31) -- (v42);
\draw[gray5] (v31) -- (v43);
\draw[gray1] (v41) -- (v51);
\draw[gray4] (v41) -- (v52);
\draw[gray3] (v43) -- (v53);
\draw[gray5] (v51) -- (vR1);
\draw[gray3] (v52) -- (vR2);
\draw[gray5] (v53) -- (vR3);

\node[main, fill=black, scale=0.3] (m11) at (7.0, 2.0) {};
\node[main, fill=black, scale=0.3] (m12) at (7.0, 0.0) {};
\node[main, fill=black, scale=0.3] (m13) at (7.0, -2.0) {};

\node[main, fill=black, scale=0.3] (m21) at (8.0, 2.5) {};
\node[main, fill=black, scale=0.3] (m22) at (8.0, 1.5) {};
\node[main, fill=black, scale=0.3] (m23) at (8.0, 0.0) {};
\node[main, fill=black, scale=0.3] (m24) at (8.0, -1.5) {};
\node[main, fill=black, scale=0.3] (m25) at (8.0, -2.5) {};

\node[main, fill=black, scale=0.3] (m31) at (9.0, 2.5) {};
\node[main, fill=black, scale=0.3] (m32) at (9.0, 2.0) {};
\node[main, fill=black, scale=0.3] (m33) at (9.0, -2.0) {};

\node[main, fill=black, scale=0.3] (m41) at (10.0, 2.0) {};

\draw[gray4] (vR1) -- (m11);
\draw[gray2] (vR2) -- (m12);
\draw[gray4] (vR3) -- (m13);
\draw[gray2] (m11) -- (m21);
\draw[gray3] (m11) -- (m22);
\draw[black] (m12) -- (m23);
\draw[gray1] (m13) -- (m24);
\draw[gray3] (m13) -- (m25);
\draw[black] (m21) -- (m31);
\draw[gray1] (m21) -- (m32);
\draw[black] (m25) -- (m33);
\draw[black] (m32) -- (m41);

\draw[decorate, decoration={brace, amplitude=10pt, mirror}] (-0.2, -3.5) -- (6.15, -3.5) node[midway, below=10pt] {BFS phase};
\draw[decorate, decoration={brace, amplitude=10pt, mirror}] (6.25, -3.5) -- (10.2, -3.5) node[midway, below=10pt] {Monotone decreasing phase};

\node[main, fill=black, scale=0.3] (v0) at (0,0) {};
\node[main, fill=black, scale=0.3] (v11) at (1.0, 0.0) {};
\node[main, fill=black, scale=0.3] (v21) at (2.0, 0.5) {};
\node[main, fill=black, scale=0.3] (v22) at (2.0, -0.5) {};
\node[main, fill=black, scale=0.3] (v31) at (3.0, 0.0) {};
\node[main, fill=black, scale=0.3] (v32) at (3.0, -1.0) {};
\node[main, fill=black, scale=0.3] (v41) at (4.0, 1) {};
\node[main, fill=black, scale=0.3] (v42) at (4.0, 0) {};
\node[main, fill=black, scale=0.3] (v43) at (4.0, -1) {};
\node[main, fill=black, scale=0.3] (v51) at (5.0, 1.5) {};
\node[main, fill=black, scale=0.3] (v52) at (5.0, 0.5) {};
\node[main, fill=black, scale=0.3] (v53) at (5.0, -1.5) {};
\node[main, fill=black, scale=0.6] (vR1) at (6, 2) {};
\node[main, fill=black, scale=0.6] (vR2) at (6, 0) {};
\node[main, fill=black, scale=0.6] (vR3) at (6, -2) {};

\node[main, fill=black, scale=0.3] (m11) at (7.0, 2.0) {};
\node[main, fill=black, scale=0.3] (m12) at (7.0, 0.0) {};
\node[main, fill=black, scale=0.3] (m13) at (7.0, -2.0) {};
\node[main, fill=black, scale=0.3] (m21) at (8.0, 2.5) {};
\node[main, fill=black, scale=0.3] (m22) at (8.0, 1.5) {};
\node[main, fill=black, scale=0.3] (m23) at (8.0, 0.0) {};
\node[main, fill=black, scale=0.3] (m24) at (8.0, -1.5) {};
\node[main, fill=black, scale=0.3] (m25) at (8.0, -2.5) {};
\node[main, fill=black, scale=0.3] (m31) at (9.0, 2.5) {};
\node[main, fill=black, scale=0.3] (m32) at (9.0, 2.0) {};
\node[main, fill=black, scale=0.3] (m33) at (9.0, -2.0) {};
\node[main, fill=black, scale=0.3] (m41) at (10.0, 2.0) {};
\end{tikzpicture}
\caption{A schematic image of the output of the two-phase exploration of the PWIT (noting that if the coupling is perfect, then the output of the exploration on $G^n$ will exactly be the same). 
The vertices on the boundary of the BFS graph are marked with larger dots.
Similarly to Figure~\ref{fig:pwit}, darker gray edges activate earlier, lighter gray edges activate later. All edges on this figure activate earlier than $\hat{t}$. Note that the edges explored in the BFS phase can activate later than their parent edge, but in the monotone decreasing phase each edge must activate earlier than its parent edge.}
\label{fig:two_phase}
\end{figure}

\begin{definition}[PPP of an edge]\label{def:edge-ppp}
On each edge $e \in E(G^n)$ we consider i.i.d.\ homogeneous Poisson point processes with intensity ${1/r_n}$. We denote the sequence of arrival times of the PPP of edge $e$ by
\begin{equation}\label{eq:ppp-Xe-def}
\ppp{X}{e} = \left( \mathcal{X}_e^{n,k} \right)_{k=1}^\infty,
\end{equation}
i.e., $\mathcal{X}_e^{n,k}$ is the $k^{\text{th}}$ point of the PPP $\ppp{X}{e}$. We denote by $G^n_*$ the graph $G^n$ decorated with the PPPs $\ppp{X}{e}, \, e \in E(G^n)$. 
In the RDCP on $G^n$ we identify the activation time of edge $e$ (cf.\ \refD{def:rdcp}) with $\mathcal{X}_e^{n,1}$, i.e., the first point of its PPP.
\end{definition}

Thus the activation times are indeed i.i.d.\ random variables with common distribution $\text{EXP}\left(1/r_n\right)$, as required by Definition~\ref{def:rdcp}. We need the remaining points of the PPPs of Definition~\ref{def:edge-ppp} to facilitate the coupling between the two-phase exploration of $G^n_*$ (to be rigorously described in Section~\ref{subsec:two_phase_algo}) and the two-phase exploration of the Poisson weighted infinite tree (to be described in Section~\ref{subsec:coupling_to_rbrw}): since the latter is naturally defined in terms of PPPs (cf.\ Definition~\ref{def:pwit}), it will be beneficial to treat an $\text{EXP}\left(1/r_n\right)$ distributed edge activation time as the first point of a PPP.

\subsection{Two-phase exploration algorithm}\label{subsec:two_phase_algo}

In Section~\ref{subsec:two_phase_algo} we define the two-phase exploration outlined in Remark~\ref{rem:heu_two_phase}. We begin by describing those aspects of the exploration that are common to the two phases, then we describe the BFS phase, finally we describe the monotone decreasing phase.

\paragraph{Exploration process on a graph}
Now we describe an exploration process of $G_*^n$. In some cases the exploration algorithm fails and gives an alarm as an output. Otherwise, the output of the exploration process is an \textbf{abstract rooted, edge-labelled tree} $\exptree$ and an injective \textbf{embedding}
\begin{equation}\label{eq:emb-def}
v^n\colon \, V(\exptree) \to V(G^n).
\end{equation}

The vertices of $\exptree$ are denoted by strings of integers, just like in the case of the vertices of the PWIT (cf.\ Definition~\ref{def:pwit}): this will facilitate the coupling between the two trees.

The \textbf{root} of $\exptree$ is denoted by the empty string ${\rt}$ and it forms the $0^{\text{th}}$ \textbf{generation} of~$\exptree$. The vertex $\emb{{\rt}} \in V(G^n)$, i.e., the embedded image of ${\rt}$ is a uniformly chosen vertex of $G^n$. Every vertex~${\ibf} \in V(\exptree)$ has finite number of children, which are denoted by $\ibf 1,\, \ibf 2,\,$ etc. The children of vertices of the~$r^{\text{th}}$ generation of $\exptree$ form generation~$(r+1)$, i.e., the generation of a vertex is equal to its distance from the root~${\rt}$. The \textbf{incoming edge} of a vertex ${\ibf} \in V(\exptree) \setminus \{{\rt} \}$, i.e., the edge connecting vertex~${\ibf}$ to its parent, is denoted by~$e_{{\ibf}}$.

At each step of the exploration process, each vertex of $G^n$ can be \textbf{neutral}, \textbf{active} or \textbf{explored}. At the beginning, $\emb{{\rt}}$ is active, every other vertex is neutral. We also say that ${\ibf \in V(\exptree)}$ is neutral, active or explored if $\emb{\ibf}$ is neutral, active or explored, respectively. The algorithm terminates when either there is no active vertex of $\exptree$ or we get an alarm. Until then, in each step of the algorithm we select an active vertex~${\ibf \in V(\exptree)}$ and discover some of the neutral neighbors of $\emb{\ibf}$. They will correspond to the children of vertex ${\ibf} \in V(\exptree)$ and we make them active. Vertex~${\emb{\ibf} \in V(G^n)}$ (and the corresponding~${\ibf \in V(\exptree)}$) becomes explored. We call this step the \textbf{exploration of} ${\ibf}$.

If ${\ibf} \in V(\exptree)$, we denote by $\nb{\ibf} \subset V(G^n)$ the set of neighbors of $\emb{{\ibf}}$ in $G^n$. We also introduce the notation $\neunb{\ibf}$, $\actnb{\ibf}$ and $\expnb{\ibf}$ for the set of neighbors that are neutral, active and explored right before the exploration of ${\ibf}$. Note that $\nb{\ibf}$, $\neunb{\ibf}$, $\actnb{\ibf}$ and~$\expnb{\ibf}$ depend on $n$, but we omit the superscripts for ease of notation. Note also that ${\nb{\ibf} = \neunb{\ibf} \cup \actnb{\ibf} \cup \expnb{\ibf}}$.

Let ${\ibf} \in V(\exptree)$, $v \in \nb{\ibf}$ and let us define a point process
\begin{equation}\label{eq:ppp-X-expl-def}
\ppp{X}{{\ibf},v} = \left( \mathcal{X}_{{\ibf},v}^{n,k} \right)_{k=1}^\infty,
\quad \text{where} \quad
\ppp{X}{{\ibf},v} = \left\{\,
\begin{tabular}{ll}
$\ppp{X}{\{\emb{{\ibf}},v \} }$ & if $v \in \neunb{\ibf} \cup \actnb{\ibf}$, \\
$\underline{0}$ & if $v \in \expnb{\ibf}$,
\end{tabular}
\right.
\end{equation}
where $\ppp{X}{\{\emb{{\ibf}},v \} }$ is defined in \eqref{eq:ppp-Xe-def} and $\underline{0}$ denotes the empty point process.

\begin{claim}[Independent edge PPPs]\label{cl:indep-ppp}
Conditional on what we have explored so far, the point processes $\left(\ppp{X}{{\ibf},v} \right)_{v \in \neunb{\ibf} \cup \actnb{\ibf}}$ are i.i.d.\ PPPs on $\mathbb{R}_+$ with intensity~${1/r_n}$.
\end{claim}

\begin{proof}
It follows from the definition of the point processes~$\ppp{X}{{\ibf},v}$ and Definition~\ref{def:edge-ppp}.
\end{proof}

\begin{definition}[Union of point processes]
If $\underline{\mathcal{X}}$ and $\underline{\mathcal{X}}'$ are point processes on $\mathbb{R}_{+}$, then we define $\underline{\mathcal{X}} \cup \underline{\mathcal{X}}'$ as the point process that we obtain if we merge the point processes~$\underline{\mathcal{X}}$ and $\underline{\mathcal{X}}'$, i.e., it consists of the ordered union of the points of $\underline{\mathcal{X}}$ and $\underline{\mathcal{X}}'$.
\end{definition}

Note that almost surely, we will never encounter points with multiplicity higher than~$1$.

\begin{definition}[Point process of a vertex]\label{def:vertex-ppp}
For any vertex ${\ibf} \in V(\exptree)$ we define a point process
\begin{equation}\label{eq:ppp-Y-expl-def}
\ppp{Y}{{\ibf}} = \left( \mathcal{Y}_{{\ibf}}^{n,k} \right)_{k=1}^\infty,
\quad \text{where} \quad
\ppp{Y}{{\ibf}} := \bigcup_{v \in \nb{\ibf}} \ppp{X}{{\ibf},v}.
\end{equation}

For each point of the point process $\ppp{Y}{{\ibf}}$ we store the associated vertex of $G^n$, i.e., we denote by $v \left(\mathcal{Y}_{{\ibf}}^{n,k}\right)$ the vertex $v \in V(G^n)$ for which $\mathcal{Y}_{{\ibf}}^{n,k}$ is a point of the PPP $\ppp{X}{{\ibf},v}$.
\end{definition}

\begin{definition}[Cycle alarm]
We say that there is a \textbf{cycle alarm} when we explore vertex~${\ibf \in V(\exptree)}$ if
\begin{equation}\label{eq:cycle-alarm-def}
\left( \bigcup_{v \in \actnb{\ibf}} \ppp{X}{\ibf,v} \right) \cap [0, \hat{t}] \neq \emptyset \quad \text{ or } \quad \exists \, v \in \nb{\ibf} \text{ such that } \big| \ppp{X}{\ibf,v} \cap [0, \hat{t}] \big| \ge 2.
\end{equation}
\end{definition}
In words, the cycle alarm rings if we close a cycle or if we find two points in the interval $[0, \hat{t}]$ in the PPP of one of the edges
that we explore.

Now we define the children of the vertex ${\ibf} \in V(\exptree)$ that is currently explored. Before that, we need to define the labels in $\exptree$.

\begin{definition}[Edge labels]
If ${\ibf j} \in V(\exptree)$ is the $j^{\text{th}}$ child of vertex ${\ibf} \in V(\exptree)$, then the \textbf{label of the incoming edge} $e_{{\ibf j}}$ of vertex ${\ibf j}$ is $\mathcal{Y}_{{\ibf}}^{n,j}$, the $j^{\text{th}}$ point of the PPP~$\ppp{Y}{{\ibf}}$, i.e., the label is equal to the arrival time. We denote it by $\edge{{\ibf j}}$ and we also call it the \textbf{label of the vertex} ${\ibf j}$.
\end{definition}

\begin{remark}
Note that the root ${\rt} \in V(\exptree)$ has no label.
\end{remark}

\begin{claim}[Label with edge PPPs]
If there is no cycle alarm, then
\begin{equation}\label{eq:edge-label-ppp}
\edge{{\ibf j}} = \mathcal{X}_{{\ibf},\emb{{\ibf j}}}^{n,1},
\end{equation}
i.e., the label of the edge $e_{{\ibf j}}$ is equal to the first point of the point process~$\ppp{X}{{\ibf},\emb{{\ibf j}}}$.
\end{claim}

\begin{proof}
Equation \eqref{eq:edge-label-ppp} follows from the definition of $\ppp{Y}{{\ibf}}$ \eqref{eq:ppp-Y-expl-def} and \eqref{eq:cycle-alarm-def}.
\end{proof}

The exploration algorithm has two phases:
\begin{itemize}
\item \textbf{Breadth-first search (BFS) phase:} In the first $R$ generations we proceed with breadth-first search (BFS), i.e., at each step we choose an active vertex ${\ibf} \in V(\exptree)$ from the lowest, say $r^{\text{th}}$, generation where there exists an active vertex. If $r=0$ we choose the root ${\rt} \in V(\exptree)$. Otherwise, from the set of active vertices of the $r^{\text{th}}$ generation, we choose the vertex ${\ibf} \in V(\exptree)$ with the lowest label $\edge{{\ibf}}$, i.e.,
\begin{equation}\label{eq:bfs-vertex-selection}
{\ibf} := \arg\, \min \{\, \edge{{\jbf}} \, | \, {\jbf} \in V(\exptree) \text{ is an active vertex of generation } r\,\}.
\end{equation}

Now we explore the chosen vertex ${\ibf}$. If there is no cycle alarm at vertex ${\ibf}$, then the points of
\begin{equation}\label{eq:bfs-children-points}
\ppp{Y}{{\ibf}} \cap [0,\hat{t}]
\end{equation}
(all of which came from edges connecting $\emb{{\ibf}}$ to a neutral neighbor) generate the children of ${\ibf}$. The $j^{\text{th}}$ child of ${\ibf}$ is denoted by ${\ibf j}$ and its embedded image is
\begin{equation}\label{eq:ij-embedded-image-with-ppp}
\emb{{\ibf j}} := v \left(\mathcal{Y}_{{\ibf}}^{n,j}\right),
\end{equation}
where the function $v$ on the right-hand side of \eqref{eq:ij-embedded-image-with-ppp} is introduced in \refD{def:vertex-ppp}.

For each child ${\ibf j}$, the vertex $\emb{{\ibf j}} \in V(G^n)$ becomes active (and the corresponding vertex ${\ibf j}$ of $\exptree$ becomes active as well). Note that we have
\begin{equation}\label{eq:edge-label-bfs}
\edge{{\ibf j}} < \hat{t}.
\end{equation}

At the same time $\emb{{\ibf}}$ (and the corresponding vertex ${\ibf}$ of $\exptree$) becomes explored.

\item \textbf{Monotone decreasing phase:} If there is no active vertex in the first $R$ generations, from generation $(R+1)$ we start the monotone decreasing phase. Similarly to the BFS, at each step we select an active vertex. But now we always select the active vertex ${\ibf}$ with the largest label $\edge{\ibf}$, regardless of the generation level of the vertices, i.e.,
\begin{equation}\label{eq:md-vertex-selection}
{\ibf} := \arg\, \max \{\, \edge{\jbf} \, | \, {\jbf} \text{ is an active vertex} \,\}.
\end{equation}

The other difference is that if there is no cycle alarm at vertex ${\ibf}$, then (in contrast with \eqref{eq:bfs-children-points}) only the points of
\begin{equation}\label{eq:md-children-points}
\ppp{Y}{{\ibf}} \cap [0,\edge{\ibf}]
\end{equation}
generate the children of ${\ibf}$. Otherwise, everything is the same as it was in the BFS phase. (E.g., the vertices determined by the points of \eqref{eq:md-children-points} become active.)

Therefore, in this phase the label $\edge{\ibf j}$ of ${\ibf j}$ is less than the label $\edge{\ibf}$ of ${\ibf}$:
\begin{equation}\label{eq:edge-label-md}
\edge{\ibf j} < \edge{\ibf}.
\end{equation}
\end{itemize}

\begin{definition}[Output of the exploration of $G^n_*$]\label{def:exploration-output}
The output of the two-phase exploration algorithm is a cycle alarm or the abstract rooted, edge-labelled tree $\exptree$ and the embedding $v^n$ of its vertex set into~$V(G^n)$.
\end{definition}

\begin{remark}$ $
\begin{itemize}
\item Note that from generation $R+1$ we do not use BFS, instead we switch to a second phase that we call the monotone decreasing phase, because the labels of the selected vertices are monotone decreasing as we move further away from the root along the branches of the tree, see \eqref{eq:edge-label-md}.
\item We recall that if ${\ibf} \in V(\exptree)$ and $v \in \expnb{{\ibf}}$, then $\ppp{X}{{\ibf},v}$ is defined to be an empty point process (see the second line of \eqref{eq:ppp-X-expl-def}). The reason behind this choice is that the edge connecting $\emb{\ibf}$ and $v$ was already revealed in the step when $v$ became explored (this follows from the order in which we explore vertices in the BFS phase and the monotone decreasing phase, respectively).
\item The choice of the active vertex that we explore in both phases only depends on the explored part of~$G^n_*$.
\item Note that if there is no cycle alarm during the exploration process of $G^n_*$, then the embedding $v^n$ is injective (see \eqref{eq:cycle-alarm-def}).
\item Note that the embedded image $\emb{V(\exptree)}$ spans a small subgraph of $G^n$, but, as we will see in \refL{lem:R-neighborhood}, the information that we explored in the two-phase algorithm is sufficient to determine the $R$-radius neighborhood of vertex $\emb{\rt}$ in the RDCP on $G^n$ at time $\hat{t}$.
\end{itemize}
\end{remark}

\subsection{Coupling with regularized branching random walk (RBRW)}\label{subsec:coupling_to_rbrw}

In Section~\ref{sss:rbrw} we define another process on $G^n$ that we call the \emph{regularized branching random walk (RBRW)}. In Section~\ref{sss:coupling_expl_rbrw} we will couple the RBRW to the exploration process of $G^n_*$ defined in Section~\ref{subsec:two_phase_algo}. We will see that the output of the RBRW can be represented as the output of an analogous exploration process on a copy of the PWIT together with the embedding of the explored vertices of the PWIT in $G^n$.

\begin{remark}
Loosely speaking, in the RBRW, each individual has $\text{POI}(\hat{t})$ many children, which are then placed independently at uniformly chosen neighbors of the parent particle. In fact, our construction of the RBRW can also be viewed as an embedding of a certain Galton--Watson tree into $G^n$ (the details are more delicate, e.g.\ the RBRW construction will also have a monotone decreasing phase). We will explain the meaning of the adjective \emph{regularized} in the name of the RBRW in Remark~\ref{rem:why_regularized}.
\end{remark}

\subsubsection{Regularized branching random walk (RBRW)}\label{sss:rbrw}

The output of the RBRW will be an abstract rooted edge-labelled tree $\rbrwtree$ (which we call the family tree) and an embedding~$\hat{v}^n$ of $V(\rbrwtree)$ into $V(G^n)$:
\begin{equation}\label{eq:emb-rbrw}
\hat{v}^n\colon \, V(\rbrwtree) \to V(G^n).
\end{equation}
The vertices of $\rbrwtree$ are denoted by strings of integers. We call the vertices of the family tree $\rbrwtree$ individuals and their embedded image into $G^n$ particles.

The root of the RBRW is ${\rt} \in V(\rbrwtree)$. The embedded image $\rbrwemb{{\rt}} \in V(G^n)$ of ${\rt}$ is chosen from $V(G^n)$ with probability proportional to its degree in $G^n$, i.e., $\rbrwemb{{\rt}} \sim \underline{\nu}^n= \left( \nu^n_v \right)_{v \in V(G^n)}$, where
\begin{equation}\label{eq:nu-def}
\nu^n_v = \frac{\deg(v)}{\sum \limits_{w \in V(G^n)} \deg(w)} \quad \forall \, v \in V(G^n).
\end{equation}
Note that $\underline{\nu}^n$ is the stationary distribution of the simple random walk on $G^n$.

Every individual~${\ibf} \in V(\rbrwtree)$ has finite number of children, which are denoted by ${\ibf 1,\, \ibf 2,}$~etc. The algorithmic construction of the RBRW will be divided into rounds.

In order to be able to compare the RBRW to the exploration of $G^n_*$, in each round of the construction of the RBRW we classify the individuals of $\rbrwtree$ as either currently active or explored and we classify the vertices of $G^n$ as currently neutral, active or explored.
At the beginning, $\rbrwtree$ has one individual ${\rt}$ that is active. The algorithm terminates when there is no active individual of $\rbrwtree$. Until then, in each step of the algorithm we select an active individual ${\ibf} \in V(\rbrwtree)$, which produces some children that will be embedded to some of the neighbors of $\rbrwemb{{\ibf}}$ and we make these new individuals active. Individual~${\ibf \in V(\rbrwtree)}$ becomes explored. We call the step when we generate the children of vertex~${\ibf \in V(\rbrwtree)}$ the exploration of~${\ibf}$.

In each round of the construction of the RBRW, a vertex $v \in V(G^n)$ is (currently) neutral if there is no particle in $v$, i.e., none of the discovered individuals is embedded into $v$. The vertex $v$ is active if there is at least one active particle located at $v$. Otherwise, $v$ is explored. If ${\ibf} \in V(\rbrwtree)$, we denote by $\rbrwnb{\ibf} \subset V(G^n)$ the set of neighbors of $\rbrwemb{{\ibf}}$ in~$G^n$, and we also introduce the notation $\rbrwneunb{\ibf}$, $\rbrwactnb{\ibf}$ and $\rbrwexpnb{\ibf}$ for the set of neutral, active and explored neighbors of $\rbrwemb{{\ibf}}$ in~$G^n$ right before the exploration of ${\ibf}$, respectively. Note that $\rbrwnb{\ibf}$ is the disjoint union of $\rbrwneunb{\ibf}$, $\rbrwactnb{\ibf}$ and $\rbrwexpnb{\ibf}$.

Let ${\ibf} \in V(\rbrwtree)$. Given what we have constructed so far, for every $v \in \rbrwnb{\ibf}$ let
\begin{equation}\label{eq:ppp-X-rbrw-def}
\rbrwppp{X}{{\ibf},v} = \left( \widehat{\mathcal{X}}_{{\ibf},v}^{n,k} \right)_{k=1}^\infty
\end{equation}
denote a PPP with intensity $1/\deg(\rbrwemb{{\ibf}})$. Moreover, let us assume that the PPPs $\rbrwppp{X}{{\ibf},v},\, {v \in \rbrwnb{\ibf}}$ are conditionally independent of each other given what we have constructed
 so far.

For any vertex ${\ibf} \in V(\rbrwtree)$ we define a point process
\begin{equation}\label{eq:ppp-Y-rbrw-def}
\rbrwppp{Y}{{\ibf}} = \left( \widehat{\mathcal{Y}}_{{\ibf}}^{n,k} \right)_{k=1}^\infty, \quad \text{where} \quad
\rbrwppp{Y}{{\ibf}} := \bigcup_{v \in \rbrwnb{\ibf}} \rbrwppp{X}{{\ibf},v}.
\end{equation}

For each point of the point process $\rbrwppp{Y}{{\ibf}}$, we store the associated vertex of $G^n$, i.e., let us denote by $\hat{v} \left(\widehat{\mathcal{Y}}_{{\ibf}}^{n,k}\right)$ the vertex $v \in V(G^n)$ for which $\widehat{\mathcal{Y}}_{{\ibf}}^{n,k}$ is a point of the PPP $\rbrwppp{X}{{\ibf},v}$.

\begin{claim}[Vertex PPP in the RBRW]\label{cl:rbrw-unit-intensity}
$\rbrwppp{Y}{{\ibf}}$ is a unit intensity PPP.
\end{claim}

\begin{proof}
It follows from the merging property of the Poisson point processes.
\end{proof}

\begin{remark}\label{rem:why_regularized} The adjective \emph{regularized} in the name of the RBRW refers to the fact that the intensity of $\rbrwppp{Y}{{\ibf}}$ is exactly the same (namely, one) for all $\ibf$. Although
this property fails for the point processes
 $\ppp{Y}{{\ibf}}$ (defined in
\eqref{eq:ppp-Y-expl-def}), later we will see that the discrepancy between the point processes $\rbrwppp{Y}{{\ibf}}$ and $\ppp{Y}{{\ibf}}$ is small enough to enable a coupling between the two (cf.\ Definition~\ref{def:coupling-rbrw}) that is perfect with high probability on $[0, \hat{t}]$.
\end{remark}

Before defining the children of the vertex ${\ibf} \in V(\rbrwtree)$ that is currently explored, we need to define the labels in $\rbrwtree$.

\begin{definition}[Labels in the RBRW]
If ${\ibf j} \in V(\rbrwtree)$ is the $j^{\text{th}}$ child of vertex ${\ibf} \in V(\rbrwtree)$, then the \textbf{label of the incoming edge} $e_{{\ibf j}}$ of vertex ${\ibf j}$ is $\widehat{\mathcal{Y}}_{{\ibf}}^{n,j}$, the $j^{\text{th}}$ point of the PPP~$\rbrwppp{Y}{{\ibf}}$, i.e., it is equal to the arrival time. We denote it by $\rbrwedge{{\ibf j}}$ and we also call it the \textbf{label of the vertex} ${\ibf j}$.
\end{definition}

The RBRW has two phases, similarly to the exploration process of $G^n_*$:
\begin{itemize}
\item \textbf{Breadth-first search (BFS) phase:} In the first $R$ generations, in each step we choose an active individual of $\rbrwtree$, analogously to the BFS phase of the exploration process of $G^n_*$ (see \eqref{eq:bfs-vertex-selection}), and we take its unit intensity PPP $\rbrwppp{Y}{{\ibf}}$. For every point of the PPP $\rbrwppp{Y}{{\ibf}}$ in the time interval $[0, \hat{t}]$ (cf.\ \eqref{eq:bfs-children-points}), we generate a new individual. These new individuals will be the children of ${\ibf}$. The embedded image of ${\ibf j}$ ($j^{\text{th}}$ child of ${\ibf}$) is
\begin{equation}\label{eq:rbrwemb-def}
\rbrwemb{{\ibf j}} := \hat{v} \left(\widehat{\mathcal{Y}}_{{\ibf}}^{n,j}\right).
\end{equation}

The new individuals ${\ibf j}$ become active and ${\ibf}$ becomes explored.

\item \textbf{Monotone decreasing phase:} From generation $(R+1)$ onward, in each step we select the active individual ${\ibf} \in V(\rbrwtree)$ with the largest label $\rbrwedge{\ibf}$ (analogously to~\eqref{eq:md-vertex-selection}), and we consider the points of the PPP $\rbrwppp{Y}{{\ibf}}$ on the time interval~$[0,\, \rbrwedge{\ibf}]$ (instead of~$[0,\, \hat{t}]$, similarly to \eqref{eq:md-children-points}). Otherwise, the process is the same as in the BFS phase. The construction of the RBRW ends when it dies out, i.e., if there is no active individual in $\rbrwtree$.
\end{itemize}

\begin{remark}
The output of the exploration process of $G^n_*$ (described in Section~\ref{subsec:two_phase_algo}) and the RBRW depend also on the value of $R$ and $\hat{t}$.
\end{remark}

\begin{claim}[Exploration of the PWIT]\label{cl:pwit-expl}
If we use a copy of the PWIT (instead of $G^n_*$) as the input of the two-phase exploration process described in Section~\ref{subsec:two_phase_algo}, then the output that we obtain is a rooted, edge-labelled tree, which has the same distribution as the family tree $\rbrwtree$ of the RBRW.
\end{claim}

\begin{proof}
In both trees the children of a vertex $\ibf$ and the edge labels are generated by a unit intensity PPP restricted to $[0, \hat{t}]$ in the BFS phase and on $[0, \rbrwedge{\ibf}]$ in the monotone decreasing phase (see \refD{def:pwit} and \refCl{cl:rbrw-unit-intensity}), where $\rbrwedge{\ibf}$ denotes the label of the incoming edge of vertex $\ibf$ in the associated tree.
\end{proof}

\subsubsection{Coupling of the graph exploration and the RBRW}
\label{sss:coupling_expl_rbrw}

We now couple the exploration process of $G^n_*$ (defined in \refS{subsec:two_phase_algo}) with the RBRW. For this purpose, we augment the probability space on which the RBRW was defined and construct on it an abstract rooted, edge-labelled tree $\modtree$ and an embedding $\tilde{v}^n$ of~$V(\modtree)$ into~$V(G^n)$ or a cycle alarm, which have the same distribution as the output of the exploration of~$G^n_*$. In particular, we modify the PPPs~$\rbrwppp{X}{{\ibf}, v}$ (which have intensity~$1/{\deg(\rbrwemb{{\ibf}})}$, defined in \eqref{eq:ppp-X-rbrw-def}) of the RBRW by deleting/inserting points to obtain PPPs~$\modppp{X}{{\ibf},v}$ with the same distribution as $\ppp{X}{{\ibf}, v}$ (see \eqref{eq:ppp-X-expl-def}). Our goal is to produce a coupling under which~${\modtree \equiv \rbrwtree}$ (the two edge-labelled trees agree) and~${\tilde{v}^n \equiv \hat{v}^n}$ (the two embeddings agree) with high probability if~${n \gg 1}$.
We will build the two trees simultaneously and if something undesired happens, we abandon our coupling via alarms. The \emph{initialization alarm} occurs if we cannot couple~$\emb{{\rt}}$ (which has uniform distribution) and~$\rbrwemb{{\rt}}$ (which has degree-biased distribution). The \emph{injectivity alarm} occurs if the RBRW visits the same vertex of~$G^n$ twice. The \emph{modification alarm} occurs if we actually insert or delete a point of~${[0,\hat{t}]}$ as we modify~$\rbrwppp{X}{{\ibf}, v}$ to obtain~$\modppp{X}{{\ibf},v}$. If none of these alarms ever goes off, then~${\modtree \equiv \rbrwtree}$ and~${\tilde{v}^n \equiv \hat{v}^n}$ and the cycle alarm will not go off either. In Section~\ref{subsec:alarm_prob} we will show that the alarms stay silent with high probability if~${n \gg 1}$.

\paragraph{}Recall the embeddings $v^n$ and $\hat{v}^n$ from \eqref{eq:emb-def} and \eqref{eq:emb-rbrw}.

\begin{definition}[Initialization alarm]\label{def:initial-alarm}
We say that we have an \textbf{initialization alarm} in the coupling if the embedded images of the roots of the coupled processes differ, i.e.,~$\emb{{\rt}} \neq \rbrwemb{{\rt}}$.
\end{definition}

As we will see later (\refL{lem:alarm1}), it is easy to define a coupling of $\emb{{\rt}}$ and $\rbrwemb{{\rt}}$ such that the probability of an initialization alarm goes to zero as $n \to \infty$ if the sequence of graphs $\{ \, G^n \, \}$ is high degree almost regular.

The next type of alarm is defined purely in terms of the RBRW.
\begin{definition}[Injectivity alarm]\label{def:inject-alarm}
We say that we have an \textbf{injectivity alarm} in the RBRW if the embedding $\hat{v}^n$ is not injective (i.e., if the construction produces two particles in the same vertex of $G^n$).
\end{definition}

\begin{definition}[Coupling the RBRW and the exploration of $G^n_*$]\label{def:coupling-rbrw}
If there is an initialization alarm, then we do not bother with synchronizing our RBRW and the exploration of $G^n_*$,
i.e., we generate a RBRW starting from the vertex $\rbrwemb{{\rt}}$, we let $\tilde{v}^n(\rt):=\emb{{\rt}}$ (which is a uniformly distributed vertex of $G^n$) be the embedded image of the root of $\modtree$ and then we generate $\modtree$ together with the rest of the embedding $\tilde{v}^n$ (or a cycle alarm) with the same distribution as $\exptree$ and $v^n$ (or a cycle alarm) in the exploration process of~$G^n_*$, independently of the rest of the construction of the RBRW.

If there is no initialization alarm, we define the coupling recursively. Let ${\ibf} \in V(\rbrwtree)$ be the next active vertex of the RBRW. Let $v \in \rbrwnb{\ibf}$. We modify $\rbrwppp{X}{{\ibf},v}$ to obtain a Poisson point processes with the same distribution as in \eqref{eq:ppp-X-expl-def} and \refCl{cl:indep-ppp}.
\begin{itemize}
\item If $\frac{1}{r_n} < \frac{1}{\deg(\rbrwemb{{\ibf}})}$ then we thin $\rbrwppp{X}{{\ibf}, v}$, more precisely we keep each point of $\rbrwppp{X}{{\ibf}, v}$ independently with probability~${\deg(\rbrwemb{{\ibf}})}/{r_n}$. We denote the point process of deleted points by
\begin{equation}\label{eq:ppp-U-def}
\ppp{U}{{\ibf}, v} = \left( \mathcal{U}_{{\ibf}, v}^{n,k} \right)_{k=1}^\infty
\end{equation}
and $\modppp{X}{{\ibf},v}$ is defined as the point process of the remaining points.

\item If $\frac{1}{r_n} > \frac{1}{\deg(\rbrwemb{{\ibf}})}$ then conditional on what we have explored so far, we define the~PPP
\begin{equation}\label{eq:ppp-V-def}
\ppp{V}{{\ibf}, v} = \left( \mathcal{V}_{{\ibf}, v}^{n,k} \right)_{k=1}^\infty
\end{equation}
with intensity $\frac{1}{r_n} - \frac{1}{\deg(\rbrwemb{{\ibf}})}$. Then we define $\modppp{X}{{\ibf},v}$ as the superposition of $\rbrwppp{X}{{\ibf}, v}$ and~$\ppp{V}{{\ibf}, v}$:
\begin{equation}
\modppp{X}{{\ibf},v} := \rbrwppp{X}{{\ibf}, v} \cup \ppp{V}{{\ibf}, v}.
\end{equation}
\end{itemize}

Moreover, if $v \in \rbrwexpnb{\ibf}$ then $\modppp{X}{{\ibf},v}$ is defined as an empty point process $\underline{0}$.

After these modifications we obtain the PPPs
\begin{equation}\label{eq:ppp-X-mod-def}
\modppp{X}{{\ibf},v} = \left( \widetilde{\mathcal{X}}_{{\ibf}, v}^{n,k} \right)_{k=1}^\infty, \qquad v \in \rbrwnb{\ibf}.
\end{equation}
For each $v \in \rbrwnb{\ibf}$, the PPP~$\modppp{X}{{\ibf},v}$ has the same conditional intensity as the associated PPP~$\ppp{X}{{\ibf},v}$ of the exploration process of $G^n_*$ (cf.\ \refCl{cl:indep-ppp}), i.e., its intensity is~${1/r_n}$ if~${v \in \rbrwneunb{\ibf} \cup \rbrwactnb{\ibf}}$ or it is the empty point process $\underline{0}$ if $v \in \rbrwexpnb{\ibf}$. Moreover, the PPPs~$\modppp{X}{{\ibf},v}, \, v \in \rbrwnb{\ibf}$ are conditionally independent of each other given what we have explored so far.

\begin{definition}[Modification alarm]\label{def:mod-alarm}
We say that we have a \textbf{modification alarm} at vertex ${\ibf} \in V(\rbrwtree)$ if there exists a vertex $v \in \rbrwneunb{\ibf} \cup \rbrwactnb{\ibf}$ such that $\rbrwppp{X}{{\ibf},v}$ (see~\eqref{eq:ppp-X-rbrw-def}) and~$\modppp{X}{{\ibf},v}$ (see~\eqref{eq:ppp-X-mod-def}) differ on the time interval $[0, \hat{t}]$, i.e., the PPP $\ppp{U}{{\ibf},v}$ or $\ppp{V}{{\ibf},v}$ has at least one point in the time interval $[0,\hat{t}]$.
\end{definition}

The coupling of $\rbrwtree$ and $\modtree$ goes as follows.

\begin{itemize}
\item If there was no modification alarm nor injectivity alarm before exploring vertex~${\ibf \in V(\rbrwtree)}$:

In this case there was no cycle alarm either and we generate the point processes~$\rbrwppp{X}{{\ibf},v}$ (see \eqref{eq:ppp-X-rbrw-def}) and $\modppp{X}{{\ibf},v}$ (see \eqref{eq:ppp-X-mod-def}). We grow the trees $\rbrwtree$, $\modtree$ and extend the embeddings $\hat{v}^n$, $\tilde{v}^n$ using the corresponding PPPs (according to the rules described in Sections~\ref{sss:rbrw} and~\ref{subsec:two_phase_algo}, respectively). If we see a cycle alarm at $\modppp{X}{{\ibf},v}$, the output of the graph exploration is a cycle alarm.

\item If there was a modification or an injectivity alarm before exploring vertex ${\ibf} \in V(\rbrwtree)$:

In this case we grow $\rbrwtree$ according to the PPP $\rbrwppp{X}{{\ibf},v}$. If there was a cycle alarm in the graph exploration, then its output is already defined as a cycle alarm. If there was no cycle alarm in the graph exploration, then $\modtree$ and its embedding $\tilde{v}^n$ are generated as $\exptree$ and $v^n$ (detailed in \refS{subsec:two_phase_algo}) independently of~$\rbrwtree$.
\end{itemize}
\end{definition}

\begin{lemma}[Correct marginal distributions]\label{lem:pwit-marginal}
The construction of Definitions~\ref{def:coupling-rbrw} and~\ref{def:mod-alarm} indeed produces an output which has the same distribution as the output of the two-phase exploration of $G^n_*$ (described in \refD{def:exploration-output}). In particular, the output is either a cycle alarm with the appropriate probability or, given that there was no cycle alarm, the conditional joint distributions of $\left( \exptree,\,v^n \right)$ and $\left( \modtree,\,\tilde{v}^n \right)$ are the same.
\end{lemma}

\begin{proof}
This follows from the definition of the exploration process.
\end{proof}

\subsection{Coupling construction properties}\label{subsec:coupling_properties}

In this section we study the main properties of the coupling construction described in \refD{def:coupling-rbrw}. These properties will imply \eqref{eq:local-limit} if $\hat{t}<\infty$.

\begin{definition}[Alarm]\label{def:alarm}
We say that an \textbf{alarm rings} if we have an initialization (see \refD{def:initial-alarm}), injectivity (see \refD{def:inject-alarm}) or modification alarm (see \refD{def:mod-alarm}) during the coupling algorithm.
\end{definition}

\begin{lemma}[Alarm probability]\label{lem:alarm}
For any choice of $R$ and $\hat{t}$, the probability that the alarm rings during the coupling goes to $0$ as~$n \to \infty$.
\end{lemma}

We prove \refL{lem:alarm} in \refS{subsec:alarm_prob}.

\begin{lemma}[No alarm implies synchronization]\label{lem:sync}
The above described coupling has the property that if the alarm has not rung during the coupling, then the output cannot be a cycle alarm and the two processes generate the same rooted, edge-labelled tree with the same embedded image, i.e., $\modtree \equiv \rbrwtree$ and $\tilde{v}^n \equiv \hat{v}^n$.
\end{lemma}

\begin{proof}
If the alarm did not go off, then the embedding $\hat{v}^n$ of $V(\rbrwtree)$ into $V(G^n)$ is injective. Since there were no initialization and modification alarms, we know that $\rbrwppp{X}{{\ibf},v}$ and $\modppp{X}{{\ibf},v}$ are the same on $[0,\hat{t}]$ for any ${\ibf} \in \rbrwtree$ and for any $v \in \rbrwnb{\ibf}$, therefore, $\rbrwtree \equiv \modtree$ (the two trees agree) and $\hat{v}^n \equiv \tilde{v}^n$ (the two embeddings agree).
\end{proof}

\begin{definition}[Coupling decorated with degree constraints]\label{def:deg-cons-coupling}
Let $\underline{p}$ satisfy the mild Assumption~\ref{assump:degree_constraint_mild}. Conditionally on the output of the RBRW (cf.\ Section~\ref{sss:rbrw}), let us decorate the vertices of $\rbrwtree$ with i.i.d.\ degree constraints $d(\ibf)$ with distribution $\underline{p}$. If there was no alarm during the coupling algorithm, then $\rbrwtree=\modtree$ and for each vertex $\ibf \in V(\modtree)$ we define the degree constraint of~$\ibf$ to be equal to~$d(\ibf)$.
On the other hand, if there was an alarm in the coupling and the output of the construction described in \refD{def:coupling-rbrw} is $\modtree$ (i.e., the output is not a cycle alarm), then we decorate the vertices of $\modtree$ with i.i.d.\ degree constraints with distribution~$\underline{p}$.
\end{definition}

\begin{corollary}[Synchronization and marginals decorated with degree constraints]\label{cor:decorated-synch}
The coupling decorated with degree constraints (see \refD{def:deg-cons-coupling}) preserves the properties seen in Lemmas~\ref{lem:pwit-marginal} and~\ref{lem:sync}, i.e.,
\begin{itemize}
\item the output of the exploration process of $G^n_*$ (described in \refS{subsec:two_phase_algo}) decorated with i.i.d.\ degree constraints and the output of the construction of \refD{def:coupling-rbrw} decorated with degree constraints (\refD{def:deg-cons-coupling}) have the same distribution,
\item if the alarm has not rung during the coupling, then the coupled processes decorated with degree constraints generate the same rooted, edge-labelled tree with the same embedded image.
\end{itemize}

Moreover, the upper bound on the probability of no alarm (and therefore, complete synchronization, including degree constraints) is the same as in the case without degree constraints, and in particular, the conclusion of Lemma~\ref{lem:alarm} will also hold in the case when the vertices are decorated with degree constraints.
\end{corollary}

Recall that we have defined the phantom saturation time $T_{\ibf}$ of a vertex $\ibf$ of the RDCP on the PWIT in \refD{def:phantom-sat}. Now we define it in the exploration process of $G^n_*$.

\begin{definition}[Phantom saturation times of the explored vertices of $G^n$]\label{def:phantom-sat-Gn}
Let us assume that there was no cycle alarm in the exploration of~$G^n_*$ (see \refS{subsec:two_phase_algo}). Let~${\ibf \in V(\exptree) \setminus \{ {\rt} \}}$ and $e_{{\ibf}}$ be the incoming edge of ${\ibf}$ (cf.\ Section~\ref{subsec:two_phase_algo}). Then the \textbf{phantom saturation time} of~${\ibf}$ in the exploration of~$G^n_*$ is the time when ${\ibf}$ saturates if we remove the edge $e_{{\ibf}}$ from~$G^n$. The phantom saturation time of the root ${\rt}$ is defined to be equal to its saturation time. We denote the phantom saturation time of ${\ibf} \in V(\exptree)$ by $T_{{\ibf}}^n$.
\end{definition}

\begin{definition}[Truncated phantom saturation time]\label{def_tr_st}
Let $t \in \mathbb{R}_+$ be fixed and let ${\ibf \in V(\exptree)}$. Then we define the \textbf{$t$-truncated phantom saturation time} $T^{n,t}_{{\ibf}}$ of ${\ibf}$ in the exploration of $G^n_*$ as
\begin{equation}\label{eq:truncated-phantom}
T^{n,t}_{{\ibf}} := T_{{\ibf}}^n \wedge t,
\end{equation}
where $T_{{\ibf}}^n$ is the phantom saturation time of vertex ${\ibf}$ (see \refD{def:phantom-sat-Gn}).
\end{definition}

Recall the notation $G^n_{\underline{p}}(\hat{t})$ and $G^\infty_{\underline{p}}(\hat{t})$ from Definitions~\ref{def:rdcp} and~\ref{def:rdcp-on-pwit}.

\begin{definition}[Neighborhood in the RDCP]\label{def:B-nR}
Let us denote by $\nbh{n}{\hat{t}}{v}$ the $R$-radius neighborhood of vertex $v$ in $G^n_{\underline{p}}(\hat{t})$ and by $\nbh{\infty}{\hat{t}}{\rt}$ the $R$-radius neighborhood of $\rt$ in~${G^\infty_{\underline{p}}(\hat{t})}$.
\end{definition}

\begin{lemma}[Determining the $R$-neighborhood]\label{lem:R-neighborhood}
Let us assume that there was no cycle alarm in the exploration of $G^n_*$ (see \refS{subsec:two_phase_algo}). The edge-labelled tree $\exptree$ equipped with degree constraints $d({\ibf})$ (see \refD{def:deg-cons-coupling}) determines the isomorphism class of the rooted graph $\nbh{n}{\hat{t}}{\emb{{\rt}}}$ (introduced in \refD{def:B-nR}).
\end{lemma}

\begin{proof}
Let us denote by $\mathcal{T}^{n,R}_{\hat{t}}$ the subgraph of $\exptree$ spanned by the first $R$ generations of~$\exptree$.

Note that $\nbh{n}{\hat{t}}{\emb{{\rt}}}$ is a subgraph of the embedded image of $\mathcal{T}^{n,R}_{\hat{t}}$.

Moreover, if we already know the $\hat{t}$-truncated phantom saturation time $T^{n,\hat{t}}_{{\ibf}}$ of each vertex of $\mathcal{T}^{n,R}_{\hat{t}}$, then the edges of $\nbh{n}{\hat{t}}{\emb{{\rt}}}$ can be determined, starting from the root, in the same order as they are generated in $\exptree$. Indeed, the embedded image of an edge~$({\rt}, j)$ of $\mathcal{T}^{n,R}_{\hat{t}}$ emanating from the root ${\rt}$ is an edge of $\nbh{n}{\hat{t}}{\emb{{\rt}}}$ if and only if
\begin{equation}\label{eq:root-edge-kept}
\edge{{j}} \le \min \left( T^{n,\hat{t}}_{{\rt}}, T^{n,\hat{t}}_{{j}} \right),
\end{equation}
noting that $\edge{{j}}=T^{n,\hat{t}}_{{\rt}}$ can occur with positive probability, but we almost surely have $\edge{{j}}\neq T^{n,\hat{t}}_{{j}}$.
 The embedded image of an edge $({\ibf}, {\ibf}j)$ of $\mathcal{T}^{n,R}_{\hat{t}}$ emanating from a vertex ${\ibf} \neq {\rt}$ is an edge of $\nbh{n}{\hat{t}}{\emb{{\rt}}}$ if and only if $\emb{{\ibf}}$ belongs to $\nbh{n}{\hat{t}}{\emb{{\rt}}}$ and
\begin{equation}\label{eq:non-root-edge-kept}
\edge{{\ibf j}} < \min \left( T^{n,\hat{t}}_{{\ibf}}, T^{n,\hat{t}}_{{\ibf j}} \right).
\end{equation}
Note that almost surely we have $\edge{{\ibf j}} \neq T^{n,\hat{t}}_{{\ibf j}}$.
Also note that in \eqref{eq:root-edge-kept} equality is allowed, while in \eqref{eq:non-root-edge-kept} equality is not allowed, since the $d({\ibf})^{\text{th}}$ child cannot be connected to $\emb{{\ibf}}$ in $\nbh{n}{\hat{t}}{\emb{{\rt}}}$ except for the case when $\ibf=\rt$.

Therefore, we only need to show that the $\hat{t}$-truncated phantom saturation time $T^{n,\hat{t}}_{{\ibf}}$ of a vertex ${\ibf} \in V(\mathcal{T}^{n,R}_{\hat{t}})$ can be determined by the edge-labelled tree $\exptree$ and the degree constraints.

For that purpose, let us define an auxiliary time $\tilde{t}({\ibf})$ for each vertex ${\ibf} \in V(\exptree)$:
\begin{equation}\label{eq:ttilde-def}
\tilde{t}({\ibf}):=\left\{\,
\begin{tabular}{ll}
$\hat{t}$, & if ${\ibf} \in V(\mathcal{T}^{n,R}_{\hat{t}})$\\
$\edge{\jbf}$, & if ${\ibf} \notin V(\mathcal{T}^{n,R}_{\hat{t}})$ and vertex ${\jbf}$ is the parent of vertex ${\ibf}$,
\end{tabular}
\right.
\end{equation}
i.e., $\tilde{t}({\ibf})$ is the right-hand side of the inequality \eqref{eq:edge-label-bfs} or \eqref{eq:edge-label-md}.

We show that the $\tilde{t}({\ibf})$-truncated phantom saturation time $T^{n,\tilde{t}({\ibf})}_{{\ibf}}$ can be determined for any vertex ${\ibf} \in V(\exptree)$ by only looking at $\exptree$. Note that \refL{lem:R-neighborhood} will follow immediately, as if ${\ibf} \in V(\mathcal{T}^{n,R}_{\hat{t}})$, then we have $\tilde{t}({\ibf})=\hat{t}$ (see \eqref{eq:ttilde-def}).

We claim that $T^{n,\tilde{t}({\ibf})}_{{\ibf}}$ can be determined by recursion starting from the leaves of~$\exptree$. Let~${\ibf}$ be a leaf of $\exptree$. First, note that by \refD{def:phantom-sat-Gn} the incoming edge of ${\ibf}$ is ignored. On the other hand, as~${\ibf}$ is a leaf, each point of the point process~$\ppp{Y}{{\ibf}}$ is larger than~$\tilde{t}({\ibf})$ (see~\eqref{eq:bfs-children-points} and~\eqref{eq:md-children-points}). Hence the activation time of any edge~$e$ (that is defined as the first point of the PPP~$\ppp{X}{e}$, see \refD{def:edge-ppp}) connecting $\emb{{\ibf}}$ and a non-explored neighbor of~$\emb{{\ibf}}$ is larger than $\tilde{t}({\ibf})$ (see \eqref{eq:ppp-X-expl-def} and \eqref{eq:ppp-Y-expl-def}). Therefore, we have $T^{n,\tilde{t}({\ibf})}_{{\ibf}} = \tilde{t}({\ibf})$ (see~\eqref{eq:truncated-phantom}).

Now let ${\ibf} \in V(\exptree)$ be a vertex such that we know $T^{n,\tilde{t}({\ibf j})}_{{\ibf j}}$ for all of its children ${\ibf j}$. Let us define
\begin{equation}\label{eq:Ci-def}
C^n({\ibf}) := \left\{\, \edge{{\ibf j}} \, \Big| \, {\ibf j} \text{ is a child of } {\ibf} \text{ such that } \edge{{\ibf j}} < T^{n,\tilde{t}({\ibf j})}_{{\ibf j}} \, \right\}.
\end{equation}
Thus $C^n({\ibf})$ contains the labels $\edge{{\ibf j}}$ of those children ${\ibf j}$ such that ${\ibf j}$ has not saturated before its incoming edge appeared and $\edge{{\ibf j}} < \tilde{t}({\ibf j})$, which is inequality \eqref{eq:edge-label-bfs} in the BFS phase and \eqref{eq:edge-label-md} in the monotone decreasing phase.

Note that $|C^n({\ibf})|$ is the number of such children. Note that in the monotone decreasing phase in the definition of $C^n({\ibf})$ (see \eqref{eq:Ci-def}) we do not take into account children ${\ibf j}$ that satisfy $\edge{{\ibf j}} > \edge{{\ibf}}$ (see \eqref{eq:truncated-phantom}, \eqref{eq:ttilde-def}), since the value of~$T^{n,\tilde{t}({\ibf})}_{{\ibf}}$ does not depend on such children.

Let us denote by $\tau^n_*({\ibf},k)$ the $k^{\text{th}}$ smallest element of $C^n({\ibf})$. Then we have
\begin{equation}\label{eq:determine-truncated-phantom}
T^{n,\tilde{t}({\ibf})}_{{\ibf}}:=\left\{\,
\begin{tabular}{ll}
$\tilde{t}({\ibf})$, & if $|C^n({\ibf})| < d({\ibf})$,\\
$\tau^n_*({\ibf},d({\ibf}))$ & otherwise.
\end{tabular}
\right.
\end{equation}

The recursion \eqref{eq:determine-truncated-phantom} shows that the $\tilde{t}({\ibf})$-truncated phantom saturation time $T^{n,\tilde{t}({\ibf})}_{{\ibf}}$ can be determined for all vertices ${\ibf} \in V(\exptree)$ by only looking at $\exptree$ and the degree constraints.
\end{proof}

It remains to prove \refL{lem:alarm}. First we show some preliminary results about the~RBRW.

\subsection{RBRW properties}\label{subsec:properties_rbrw_tree}

In this section our goal is to study the properties of the RBRW introduced in Section~\ref{sss:rbrw}.

\begin{lemma}[RBRW is stationary]\label{lem:stat}
The RBRW described in Section~\ref{sss:rbrw} is a stationary tree-indexed Markov chain, i.e., if we condition on $\rbrwtree$, then for each ${\ibf} \in V(\rbrwtree)$ the distribution of $\rbrwemb{{\ibf}}$ is $\underline{\nu}^n$ (defined in~\eqref{eq:nu-def}).
\end{lemma}

\begin{proof}
Let us first generate the family tree $\rbrwtree$ of the RBRW. Recalling \refCl{cl:pwit-expl}, we note that $\rbrwtree$ can be viewed as the output of the two-phase exploration of the PWIT, thus it can be generated without looking at $G^n$. Recall that $\hat{v}^n$ is the embedding of $V(\rbrwtree)$ into~$V(G^n)$ (see \eqref{eq:emb-rbrw}). The embedded image $\rbrwemb{{\rt}}$ of the root ${\rt}$ has distribution~$\underline{\nu}^n$ (see Section~\ref{sss:rbrw}).

Recall that if ${\ibf} \in \rbrwtree$ then for any $v \in \rbrwnb{{\ibf}}$ we defined the PPP $\rbrwppp{X}{{\ibf},v}$ (see \eqref{eq:ppp-X-rbrw-def}). Given the part of the RBRW tree and the part of its embedding that we have generated so far, these PPPs are conditionally independent and have the same intensity. Moreover, $\rbrwppp{Y}{{\ibf}}$ denotes the union of these PPPs (see~\eqref{eq:ppp-Y-rbrw-def}) and the children of ${\ibf}$ are defined by the points of $\rbrwppp{Y}{{\ibf}}$ (see Section~\ref{sss:rbrw}). It follows from \eqref{eq:rbrwemb-def} and the coloring property of Poisson point processes that, conditionally on $\rbrwppp{Y}{{\ibf}}$ (i.e., the labels on the edges that connect $\ibf$ to its children), the children of vertex ${\ibf}$ are embedded into uniformly chosen neighbors of $\rbrwemb{{\ibf}}$, conditionally independently of each other.

Therefore, if $P$ denotes the transition matrix of the simple random walk on $G^n$ and the embedded image of a vertex ${\ibf} \in V(\rbrwtree)$ has distribution $\underline{\nu}^n$, then the embedded image of any child of ${\ibf}$ has distribution $\underline{\nu}^n \cdot P = \underline{\nu}^n$.
Therefore, by induction, conditional on $\rbrwtree$, for each ${\ibf} \in V(\rbrwtree)$, the distribution of $\rbrwemb{{\ibf}}$ is $\underline{\nu}^n$.
\end{proof}

\begin{notation}
Let $G^n$ be a high degree almost regular graph sequence and recall the sequences $r_n$, $a_n$, $b_n$ and $c_n$ from \refD{def:hdar}. Let us define
\begin{equation}\label{eq:K_n-def}
K_n := \min \left\{\, \log r_n, \, \log \left( \frac{1}{a_n+b_n+c_n} \right)\, \right\}.
\end{equation}
Note that $K_n \to \infty$ as $n \to \infty$ by Definition~\ref{def:hdar}.
\end{notation}

First we state a lemma that can be used to estimate the tail probabilities of the number of children of a vertex:
\begin{lemma}[Poisson tail probability]\label{lem:poi-tail}\cite[Theorem 5.4]{MU}
If $\xi \sim \text{POI}(\lambda)$ and $x > \lambda$, then
\begin{equation}\label{eq:poi-tail}
\prob(\xi \ge x) \le \mathrm{e}^{-\lambda} \cdot \left(\frac{\mathrm{e}\lambda}{x}\right)^x.
\end{equation}
\end{lemma}

\begin{definition}[Size of the RBRW family tree]
Let $G^n$ be a high degree almost regular graph sequence and we consider the RBRW on it. Let us introduce the following notation.
\begin{itemize}
\item $T_r$ denotes the total number of vertices of the $r^{\text{th}}$ generation for $r = 0,1,\dots, R$,
\item $T_{\text{BFS}}$ denotes the total number of vertices of the BFS phase (i.e., the union of the first $R$ generations),
\item $T_{\text{MD}}$ denotes the total number of vertices of the monotone decreasing phase.
\end{itemize}
\end{definition}

\begin{lemma}[Bounds on the size of the RBRW family tree]\label{lem:key:size}
Let us fix $\hat{t}$. Let $G^n$ be a high degree almost regular graph sequence. Let us consider the RBRW on $G^n$ and let us define $K_n$ as in~\eqref{eq:K_n-def}. Then we have
\begin{align}
&\prob \left(T_r > K_n^r \right) \stackrel{n \to \infty}{\longrightarrow} 0 \quad \text{for any } r \in \{0,\, 1, \dots,\, R\},\label{eq:rth-level}\\
&\prob \left(T_{\text{BFS}} > K_n^{R+1} \right) \stackrel{n \to \infty}{\longrightarrow} 0,\label{eq:bfs-size}\\
&\prob \left(T_{\text{MD}} > K_n^{3R+1} \right) \stackrel{n \to \infty}{\longrightarrow} 0.\label{eq:md-size}
\end{align}
\end{lemma}

\begin{proof}$ $
First, we prove \eqref{eq:rth-level} and \eqref{eq:bfs-size}:

Let ${\ibf}$ be a vertex in one of the first $R$ generations of the family tree $\rbrwtree$. Conditionally on the part of the tree that we have generated up to the exploration of $\ibf$, the distribution of the number of children of vertex ${\ibf}$ in the PWIT that have label less than $\hat{t}$ is a random variable $\xi$ with distribution $\text{POI}\left(\hat{t} \right)$ (we omit the dependence of $\xi$ on $\ibf$ for ease of notation).

By \refL{lem:poi-tail}, we can upper bound the tail probability of $\xi$:
\begin{equation}\label{eq:tail-upper-bound}
\prob(\xi \ge K_n) \stackrel{\eqref{eq:poi-tail}}{\le} \mathrm{e}^{-\hat{t}} \cdot \left(\frac{\mathrm{e} \cdot \hat{t}}{K_n}\right)^{K_n} =: \alpha = \alpha(n).
\end{equation}

Thus by \eqref{eq:tail-upper-bound}, with probability at least $1-\alpha$, the root has at most $K_n$ children. Similarly, each vertex in the first $R$ generations has at most $K_n$ children with high probability. If each vertex in the first $R$ generations has at most $K_n$ children, the tree has at most $K_n^r$ vertices on the $r^{\text{th}}$ level and less than $K_n^{R+1}$ vertices in the union of the first~$R$ generations. Using the union bound, we can get an upper bound for the probability that among the first $K_n^{R+1}$ vertices there exists a vertex with more than $K_n$ children. Denoting this event by $A_n(R)$, by $K_n \to \infty$, we obtain
\begin{align*}
\prob(A_n(R)) &\le K_n^{R+1} \cdot \alpha(n) \stackrel{\eqref{eq:tail-upper-bound}}{=} \mathrm{e}^{-\hat{t}} \cdot \left(\mathrm{e} \cdot \hat{t} \right)^{R+1} \cdot \left(\frac{\mathrm{e} \cdot \hat{t}}{K_n} \right)^{K_n-R-1} \stackrel{n \to \infty}{\longrightarrow} 0.
\end{align*}
As $A_n(R)$ contains the events $\{\, T_r > K_n^r \,\}$ for every $r \le R$ and the event ${\{\, T_{\text{BFS}} > K_n^{R+1} \,\}}$, it proves \eqref{eq:rth-level} and \eqref{eq:bfs-size}.

\smallskip

Finally, we prove property \eqref{eq:md-size}:

Let us denote by $S_i,\, i=1,\, 2,\, \dots,\, T_R$ the size of the subtrees rooted at vertices that belong to the $R^{\text{th}}$ generation. Let us also introduce the notation $\prob_R$ and $\ev_R$ for the conditional probability and expectation given the first $R$ generations of the RBRW family tree $\rbrwtree$. First, we show that for any $i \in \{1,\, 2,\, \dots,\, T_R\}$ we have
\begin{equation}\label{eq:Si-expectation}
\ev_R(S_i) = \mathrm{e}^{\hat{t}}.
\end{equation}

For this purpose, let us pick $i \in \{1, 2, \dots, T_R\}$ (i.e., a vertex in generation $R$) and let us denote by $S_i^r$ the number of vertices of $S_i$ that are in generation $r$ of~$S_i$ (now we start to count generations from the root of $S_i$). We prove that
\begin{equation}\label{eq:Sir-expectation}
\ev_R(S_i^r) = \frac{\hat{t}^r}{r!}.
\end{equation}
If we consider the offspring of the root of $S_i$ with label at most $\hat{t}$ and disregard the monotone decreasing label requirement, then the expected number of vertices in the~$r^{\text{th}}$ generation is equal to $\hat{t}^r$. Let us pick a vertex~${\jbf}$ in generation~$r$ without looking at the labels of edges. Now the edge labels on the path connecting the root of~$S_i$ and vertex~${\jbf}$ are i.i.d.\ with uniform distribution on the interval~$[0,\hat{t}]$. Vertex~${\jbf}$ is in~$S_i^r$ if and only if the edge labels are monotone decreasing on the path from the root of $S_i$ to ${\jbf}$. Since each possible permutation of these $r$ labels is equally likely, we have~\eqref{eq:Sir-expectation}.

Hence we obtain
\eqref{eq:Si-expectation}, because $\ev_R(S_i) = \sum_{r=0}^\infty \ev_R(S_i^r) \stackrel{\eqref{eq:Sir-expectation}}{=} \mathrm{e}^{\hat{t}}$.

By \eqref{eq:Si-expectation} and Markov's inequality we get
\begin{equation}\label{eq:Si-markov}
\prob_R \left(S_i \ge \mathrm{e}^{\hat{t}}\cdot K_n^{2R} \right) \le K_n^{-2R} \quad \forall i \in \{1,\, 2,\, \dots,\, T_R\}.
\end{equation}

Using the union bound we obtain
\begin{align}
\begin{split}\label{eq:T_MD-union-bound}
\prob\left(T_{\text{MD}} > \mathrm{e}^{\hat{t}}\cdot K_n^{3R}\right) &\le \prob \left(T_R > K_n^R \right) + \prob\left( \max \limits_{i \in \{\, 1,2,\dots, T_R \,\}} S_i \ge \mathrm{e}^{\hat{t}} \cdot K_n^{2R} \, \big| \, T_R \le K_n^R \right) \le\\
&\prob \left(T_R > K_n^R \right) + K_n^R \cdot \prob_R \left(S_1 \ge \mathrm{e}^{\hat{t}}\cdot K_n^{2R} \right) \stackrel{\eqref{eq:Si-markov}}{\le}\\
&\prob \left(T_R > K_n^R \right) + K_n^{-R} \stackrel{n \to \infty}{\longrightarrow} 0
\end{split}
\end{align}
by \eqref{eq:rth-level} and since $K_n \to \infty$. Note that $\mathrm{e}^{\hat{t}} \le K_n$ if $n$ is large enough, thus~\eqref{eq:T_MD-union-bound} implies~\eqref{eq:md-size}.
\end{proof}

\begin{corollary}[Size of the RBRW family tree]\label{cor:rbrw-bound}
Let $G^n$ be a high degree almost regular graph sequence and we consider the RBRW on it. Recall that $\rbrwtree$ is the family tree of the RBRW. Let us define $K_n$ as in~\eqref{eq:K_n-def}. Then
\begin{equation}\label{eq:rbrw-bound}
\prob \left(\big|V(\rbrwtree) \big| > K_n^{3R+2} \right)\stackrel{n \to \infty}{\longrightarrow} 0.
\end{equation}
\end{corollary}

Note that the event of \eqref{eq:rbrw-bound} depends only on $\rbrwtree$ and does not depend on the embedding $\hat{v}^n$.

\begin{proof}[Proof of \refCl{cl:rdcp-pwit-welldefined}]
\refL{lem:R-neighborhood} and \refC{cor:rbrw-bound} show that almost surely the saturation time of each vertex of $G_{\underline{p}}^\infty(\hat{t})$ can be determined by looking at an almost surely finite subgraph of the decorated PWIT.
\end{proof}

\begin{definition}[Defective vertices]
Let us call a vertex $v \in V(G^n)$ a \textbf{defective vertex} if its degree is not in the interval $(1 \pm b_n)r_n$. Let us denote the set of defective vertices by $\mathcal{D}$:
\begin{equation}\label{eq:defective-set}
\mathcal{D} := \{\, v \in V(G^n) \, | \, \deg(v) \notin [(1-b_n)r_n, \, (1+b_n)r_n] \,\}.
\end{equation}
\end{definition}

\begin{lemma}[Stationary measure of defective vertices]\label{lem:defective-measure}
Let $a_n$, $b_n$, $c_n$, $r_n$ as in \refD{def:hdar}. Then
\begin{equation}
\underline{\nu}^n(\mathcal{D}) = \frac{\sum \limits_{v \in \mathcal{D}} \deg(v)}{\sum \limits_{w \in V(G^n)} \deg(w)} \le 2(a_n+b_n+c_n) \quad \text{ if } n \text{ is large enough.}\label{eq:defective-rate}
\end{equation}
\end{lemma}

\begin{proof}
First note that if $n$ is large enough, then
\begin{equation}\label{eq:deg_sum}
\sum \limits_{w \in V(G^n)} \deg(w) \ge (1-c_n)r_n \cdot |V(G^n)|
\end{equation}
by \refD{def:hdar}~\ref{item:hdar_degree_sum}.

On the other hand, we can use \refD{def:hdar}~\ref{item:hdar_regular} and~\ref{item:hdar_degree_sum} to conclude

\begin{align}
\begin{split}\label{eq:defective_deg_sum}
\sum \limits_{w \in \mathcal{D}} \deg(w) &= \sum \limits_{w \in V(G^n)} \deg(w) - \sum \limits_{w \in (V(G^n) \setminus \mathcal{D})} \deg(w) \stackrel{\eqref{eq:defective-set}}{\le}\\
& (1+c_n)r_n \cdot |V(G^n)| - (1-a_n)|V(G^n)| \cdot (1-b_n)r_n \le\\
& (a_n+b_n+c_n)r_n \cdot |V(G^n)|.
\end{split}
\end{align}

Hence \eqref{eq:defective-rate} follows from \eqref{eq:deg_sum} and \eqref{eq:defective_deg_sum} using the fact that $1-c_n \ge \frac12$ if $n$ is large enough (see Definition~\ref{def:hdar}).
\end{proof}

\begin{lemma}[No defective vertices visited by the RBRW]\label{lem:rbrw-defective}
With high probability, the RBRW does not visit defective vertices, i.e.,
\begin{equation}\label{eq:rbrw-defective-bound}
\prob \left( \rbrwemb{V(\rbrwtree)} \cap \mathcal{D} \neq \emptyset \right) \stackrel{n \to \infty}{\longrightarrow} 0.
\end{equation}
\end{lemma}

\begin{proof}
By Lemmas~\ref{lem:stat} and~\ref{lem:defective-measure}, the probability that a vertex ${\ibf} \in V(\rbrwtree)$ is embedded into a defective vertex can be upper bounded by $2(a_n+b_n+c_n)$. Therefore, by the union bound we have
\begin{equation}\label{eq:no-defective-conditional}
\prob \left( \rbrwemb{V(\rbrwtree)} \cap \mathcal{D} \neq \emptyset \, \big| \, \rbrwtree \right) \le 2(a_n+b_n+c_n) \cdot |V(\rbrwtree)|.
\end{equation}

By \refC{cor:rbrw-bound}, we also know that $|V(\rbrwtree)|$ can be upper bounded by $K_n^{3R+2}$ with high probability. Hence \eqref{eq:no-defective-conditional} implies \eqref{eq:rbrw-defective-bound}, since
\begin{equation*}
K_n^{3R+2} \cdot 2(a_n+b_n+c_n) \stackrel{\eqref{eq:K_n-def}}{\longrightarrow} 0 \quad \text{ as } n \to \infty.
\end{equation*}
\end{proof}

\begin{definition}[RBRW condition]\label{def:rbrw-cond}
Recall that $\rbrwtree$ denotes the family tree of the RBRW on $G^n$ and $\mathcal{D}$ denotes the set of defective vertices (see \eqref{eq:defective-set}). In the following, we denote by $C_n$ the intersection of the following two events:
\begin{align}
&\text{the size of the RBRW family tree is not larger than } K_n^{3R+2}: && \big|V(\rbrwtree) \big| \le K_n^{3R+2},\label{eq:rbrw-cond1} \\
&\text{the embedded image of } \rbrwtree \text{ has no defective vertices:} && \rbrwemb{V(\rbrwtree)} \cap \mathcal{D} = \emptyset.\label{eq:rbrw-cond2}
\end{align}
We also use the notation $C_n^c$ for the complement of the event $C_n$. We call the condition that the event $C_n$ holds the \textbf{RBRW condition}.
\end{definition}

\begin{corollary}[RBRW condition holds]
It follows from \refC{cor:rbrw-bound} and \refL{lem:rbrw-defective} that with high probability the RBRW condition holds, i.e.,
\begin{equation}\label{eq:rbrw-cond-prob}
\prob(C_n) \stackrel{n \to \infty}{\longrightarrow} 1.
\end{equation}
\end{corollary}

\subsection{Estimates of alarm probabilities}\label{subsec:alarm_prob}

In this section our goal is to prove \refL{lem:alarm}, i.e., that with high probability, the alarm never rings during our coupling algorithm (see \refD{def:alarm}).
This also allows us to prove Theorem~\ref{thm:local-limit1} for $\hat{t}<\infty$. 

\begin{definition}[Uniform distribution on $G^n$]
Let us denote by $\underline{\mu}^n = \left( \mu^n_v \right)_{v \in V(G^n)}$ the uniform distribution on the vertices of $G^n$, i.e.,
\begin{equation}
\mu^n_v = \frac{1}{|V(G^n)|} \quad \forall \, v \in V(G^n).
\end{equation}
\end{definition}

Recall the notion of the measure $\underline{\nu}^n$ on $V(G^n)$ from \eqref{eq:nu-def}.

\begin{lemma}[No initialization alarm]\label{lem:alarm1} One can couple $\emb{{\rt}}$ and $\rbrwemb{{\rt}}$ (i.e., the embedded images of the roots of the two processes, cf.\ Sections~\ref{subsec:two_phase_algo} and~\ref{subsec:coupling_to_rbrw}) in a way that $\prob \left( \emb{{\rt}} = \rbrwemb{{\rt}} \right) \to 1$ as $n \to \infty$.
\end{lemma}
\begin{proof} First, observe that it is an easy consequence of \refD{def:hdar} (see \cite[Equation~(4)]{NP}) that the total variation distance $d_{\text{TV}}(\underline{\mu}^n, \underline{\nu}^n)$ of the uniform measure $\underline{\mu}^n$ and the stationary measure $\underline{\nu}^n$
goes to zero as $n \to \infty$. It is well-known (cf.\ \cite[Proposition~4.7]{LPW}) that there exists a coupling of $\emb{{\rt}}$ and $\rbrwemb{{\rt}}$ such that
$\prob \left( \emb{{\rt}} = \rbrwemb{{\rt}} \right) = 1- d_{\text{TV}}(\underline{\mu}^n, \underline{\nu}^n)$, thus the proof of \refL{lem:alarm1} is complete.
\end{proof}

\begin{lemma}[No injectivity alarm]\label{lem:alarm2}
With high probability, there is no injectivity alarm in the coupling.
\end{lemma}

\begin{proof}
Let $V(\rbrwtree) = \{ \, i_1,\, i_2, \, \dots,\, i_{|V(\rbrwtree)|} \, \}$, where the vertices are ordered in the same order as they are generated in $\rbrwtree$. In particular, the index of each vertex is larger than the index of its parent.

Observe that for any $l \in \left\{ \,1,2, \dots, |V(\rbrwtree)| \, \right\}$, we have
\begin{equation}\label{eq:inj-obsevarion}
\prob \left(\rbrwemb{i_l} \in \left( \cup_{j=1}^{l-1} \{ \rbrwemb{i_j} \} \right) \, \big|\, \rbrwtree,\, \left(\rbrwemb{i_j}\right)_{j=1}^{l-1} \right) \le \frac{l-1}{\deg \left(\rbrwemb{p_{\rbrwtree}(i_l)} \right)},
\end{equation}
where $p_{\rbrwtree}(\ibf)$ denotes the parent of vertex $\ibf$ in $\rbrwtree$. Therefore, we have
\begin{align}
\begin{split}\label{eq:v-not-inj-cond}
\prob \left( \hat{v}^n \text{ is not injective } | \, \rbrwtree \right) &
\le |V(\rbrwtree)| \cdot \sum \limits_{{\ibf} \in V(\rbrwtree) \setminus \{ \rt \} } \ev \left( \frac{1}{\deg \left( \rbrwemb{p_{\rbrwtree}({\ibf})} \right)} \, \Bigg| \, \rbrwtree \right) \le \\
& |V(\rbrwtree)|^2 \cdot \sum \limits_{{\ibf} \in V(\rbrwtree)} \ev \left( \frac{1}{\deg(\rbrwemb{{\ibf}})} \, \Big| \, \rbrwtree \right),
\end{split}
\end{align}
where the first inequality follows from the union bound, equation \eqref{eq:inj-obsevarion} and the tower property of conditional expectation, while the second
inequality follows from the fact that each vertex of $ \rbrwtree$ can be the parent of at most $|V(\rbrwtree)|$ other vertices of $ \rbrwtree$.
Note that for any ${\ibf} \in V(\rbrwtree)$ we have
\begin{equation}\label{eq:deg-recip-cond-exp}
\ev \left( \frac{1}{\deg(\rbrwemb{{\ibf}})} \, \Big| \, \rbrwtree \right) \stackrel{\text{\refL{lem:stat}}}{=} \sum \limits_{v \in V(G^n)} \underline{\nu}^n(v) \cdot \frac{1}{\deg(v)} \stackrel{\eqref{eq:nu-def}}{=} \frac{|V(G^n)|}{2|E(G^n)|}.
\end{equation}

Hence we obtain
\begin{equation}\label{eq:v-not-inj-cond2}
\prob \left( \hat{v}^n \text{ is not injective } | \, \rbrwtree \right) \stackrel{\eqref{eq:v-not-inj-cond},\, \eqref{eq:deg-recip-cond-exp}}{\le} |V(\rbrwtree)|^3 \cdot \frac{|V(G^n)|}{2|E(G^n)|} \stackrel{\text{\refD{def:hdar}}}{\le} \frac{2|V(\rbrwtree)|^3}{r_n},
\end{equation}
where the last inequality holds if $n$ is large enough. Therefore, we have
\begin{align*}
&\prob \left( \hat{v}^n \text{ is not injective} \right) \le\\
&\prob \left( \big|V(\rbrwtree) \big| > K_n^{3R+2} \right) + \prob \left(\, \hat{v}^n \text{ is not injective and } \big|V(\rbrwtree) \big| \le K_n^{3R+2} \, \right) \stackrel{\eqref{eq:rbrw-cond1},\, \eqref{eq:v-not-inj-cond2}}{\le} \\
& \prob( C_n^c ) + \frac{2\cdot (K_n^{3R+2})^3}{r_n} \stackrel{\eqref{eq:K_n-def},\, \eqref{eq:rbrw-cond-prob}}{\longrightarrow} 0 \quad \text{as } n \to \infty.
\end{align*}
\end{proof}

\begin{lemma}[No modification alarm]\label{lem:alarm3}
With high probability, there is no modification alarm in the coupling.
\end{lemma}

\begin{proof}
Let ${\ibf} \in V(\rbrwtree)$ and let us denote by $M_{{\ibf}}$ the event that there is a modification alarm when we explore vertex ${\ibf}$.

Recall the point processes $\ppp{U}{{\ibf},v}$ and $\ppp{V}{{\ibf},v}$ defined in \eqref{eq:ppp-U-def} and \eqref{eq:ppp-V-def}. By \refD{def:mod-alarm}, we know that
\begin{equation}\label{eq:Mi-with-ppp}
M_{{\ibf}} = \left\{\, \exists \, v \in \rbrwneunb{{\ibf}} \cup \rbrwactnb{{\ibf}} \text{ such that } \mathcal{U}_{{\ibf},v}^{n,1} \le \hat{t} \text{ or } \mathcal{V}_{{\ibf},v}^{n,1} \le \hat{t} \,\right\}.
\end{equation}

We also introduce a $\sigma$-algebra $\mathcal{F} := \sigma \left( \rbrwtree,\, \rbrwemb{{\ibf}},\, {\ibf} \in V(\rbrwtree)\, \right)$.

Then by \refD{def:coupling-rbrw}, for any $\ibf \in V(\rbrwtree)$ we have
\begin{equation}\label{eq:Mi-cond}
\prob(M_{{\ibf}} \, | \, \mathcal{F} ) \stackrel{\eqref{eq:Mi-with-ppp}}{\le} c_{\rbrwtree}({\ibf}) \cdot \left( 1-\frac{\deg(\rbrwemb{{\ibf}})}{r_n} \right)_{+} + \hat{t} \cdot \left( \frac{\deg(\rbrwemb{{\ibf}})}{r_n} -1 \right)_{+},
\end{equation}
where $c_{\rbrwtree}({\ibf})$ denotes the number of children of vertex ${\ibf}$ in $\rbrwtree$ and $x_{+}$ denotes the positive part of $x$.

Note that if the $\mathcal{F}$-measurable event $C_n$ holds (see \refD{def:rbrw-cond}), then by \eqref{eq:rbrw-cond2} we~have
\begin{equation}\label{eq:mod-alarm-aux}
\Bigg| \frac{\deg(\rbrwemb{{\ibf}})}{r_n} - 1 \Bigg| \le b_n.
\end{equation}

Therefore, for any $\ibf \in V(\rbrwtree)$ we have
\begin{equation}\label{eq:Mi-cond-with-Cn}
\prob \left(M_{{\ibf}} \, | \, \mathcal{F} \right) \cdot \indic{C_n} \stackrel{\eqref{eq:Mi-cond},\, \eqref{eq:mod-alarm-aux}}{\le} \indic{C_n} \cdot \left(|V(\rbrwtree)| \cdot b_n + \hat{t}\cdot b_n \right).
\end{equation}

Let us denote by $M$ the event that there is a modification alarm in the coupling. By the union bound we obtain
\begin{align}
\begin{split}\label{eq:Mi-union-with-Cn}
\prob(M \, | \, \mathcal{F})\cdot \indic{C_n} &= \prob \left(\bigcup \limits_{{\ibf} \in V(\rbrwtree)} M_{{\ibf}} \, \Big| \, \mathcal{F} \right) \cdot \indic{C_n} \stackrel{\eqref{eq:Mi-cond-with-Cn}}{\le}\\
& \indic{C_n} \cdot |V(\rbrwtree)| \cdot \left(|V(\rbrwtree)|+\hat{t} \right) \cdot b_n \stackrel{\eqref{eq:rbrw-cond1}}{\le} 2 \cdot \left(K_n^{3R+2}\right)^2 \cdot b_n.
\end{split}
\end{align}

Finally, we show that with high probability, there is no modification alarm in the coupling:
\begin{align*}
\prob (M) &\le \prob(C_n^c) + \prob \left( M \cap C_n \right) = \prob(C_n^c) + \ev \left( \prob \left( M \cap C_n \, | \, \mathcal{F} \right) \right) \stackrel{(*)}{=}\\
& \prob(C_n^c) + \ev \left( \prob \left( M \, | \, \mathcal{F} \right) \cdot \indic{C_n} \right) \stackrel{\eqref{eq:Mi-union-with-Cn}}{\le}\\
& \prob(C_n^c) + K_n^{6R+5} \cdot b_n \stackrel{\eqref{eq:K_n-def},\, \eqref{eq:rbrw-cond-prob}}{\longrightarrow} 0 \quad \text{as }n \to \infty,
\end{align*}
where in $(*)$ we used the fact that $C_n$ is $\mathcal{F}$-measurable.
\end{proof}

\begin{proof}[Proof of \refL{lem:alarm}]
\refL{lem:alarm} follows from Lemmas~\ref{lem:alarm1}, \ref{lem:alarm2} and~\ref{lem:alarm3}.
\end{proof}

\begin{proof}[Proof of Theorem~\ref{thm:local-limit1} for $\hat{t}<\infty$]
We first show that the relation \eqref{eq:local-weak-dist} holds between the graph sequence~$G^n_{\underline{p}}(\hat{t})$ and the random rooted graph~$G^\infty_{\underline{p}}(\hat{t})$, i.e., 
\begin{equation}\label{eq:local-limit-dist-goal}
\prob \left(\nbh{n}{\hat{t}}{u^n} \simeq (H,v) \right) \longrightarrow \prob \left(\nbh{\infty}{\hat{t}}{\rt} \simeq (H,v) \right) \quad \text{as } n \to \infty
\end{equation}
for every rooted graph $(H,v)$, where $u^n$ is a uniformly chosen vertex of $G^n$.

In \refL{lem:R-neighborhood} we have shown that $\exptree$ equipped with degree constraints determines $\nbh{n}{\hat{t}}{\emb{\rt}}$. Note that $\emb{\rt}$ is a uniformly chosen vertex of $G^n$ (see \refS{subsec:two_phase_algo}). On the other hand, $\rbrwtree$ equipped with degree constraints determines $\nbh{\infty}{\hat{t}}{\rt}$ (see \refCl{cl:pwit-expl}). Moreover, if $\exptree$ is isomorphic to $\rbrwtree$ and they have the same edge labels, then $\nbh{n}{\hat{t}}{\emb{\rt}}$ is isomorphic to $\nbh{\infty}{\hat{t}}{\rt}$. Hence, by \refL{lem:alarm} and \refC{cor:decorated-synch}, we infer that~\eqref{eq:local-limit-dist-goal} holds.

In order to prove local convergence in probability (i.e., that relation \eqref{eq:local-weak-in-prob} holds between $G^n_{\underline{p}}(\hat{t})$ and $G^\infty_{\underline{p}}(\hat{t})$), we use that it is equivalent to the convergence in distribution of the two-vertex measure by \cite[Remark~2.13]{H24b}.
Hence it suffices to show that, as~$n \to \infty$, we have
\begin{align}
\begin{split}\label{eq:two-vertex-measure-conv}
&\prob \left(\nbh{n}{\hat{t}}{u_1^n} \simeq (H_1,v_1), \, \nbh{n}{\hat{t}}{u_2^n} \simeq (H_2,v_2), \right) \longrightarrow\\
&\prob \left(\nbh{\infty}{\hat{t}}{\rt} \simeq (H_1,v_1) \right) \cdot
\prob \left(\nbh{\infty}{\hat{t}}{\rt} \simeq (H_2,v_2) \right)
\end{split}
\end{align}
for every pair of rooted graphs $(H_1,v_1)$ and $(H_2, v_2)$, where $u_1^n$, $u^n_2$ are independently and uniformly chosen vertices of $G^n$. In order to prove \eqref{eq:two-vertex-measure-conv}, it is enough to see that we can couple two i.i.d.\ copies of the RBRW with the joint exploration of the $R$-neighborhoods of two independent and uniformly distributed vertices of $G^n_*$ in a way that
(a) the alarm probability goes to $0$ and
(b) no alarm implies synchronization of both explorations with their corresponding RBRWs.

This coupling can be constructed analogously to the coupling described in \refS{subsec:coupling_to_rbrw}: 
we first perform the two-phase exploration of $G^n_*$ starting from $u_1^n$, then we perform it starting from $u^n_2$, naturally changing the definitions of various alarms: e.g.\ the cycle alarm now includes the event that the exploration starting from
$u^n_2$ reaches a vertex explored from $u_1^n$, and the injectivity alarm now includes
the event that the second RBRW visits a vertex of $G^n$ that is also visited by the first RBRW. The PPP modifications of Definition~\ref{def:coupling-rbrw} also need to be adjusted in a way that reflects the fact that
as we explore~$G^n_*$ from $u_1^n$, we find that the PPPs of some edges are empty on some intervals, and we need to remember this (negative) information as we explore $G^n_*$ from $u_2^n$ (essentially, we need to add the vertices explored from $u_1^n$ to the set of explored vertices when we start exploring from $u_2^n$). As a result, the probability of the modification alarm becomes slightly larger when we perform the exploration from $u^n_2$, but it still goes to zero. 

We omit the technical details, since the proofs and error bounds are routine and fairly minor modifications of the ones given in Sections~\ref{subsec:coupling_properties}, \ref{subsec:properties_rbrw_tree} and~\ref{subsec:alarm_prob} (the main point is that each exploration only reveals a relatively small part of the graph, so it is unlikely that the two explorations ever meet). 
\end{proof}

\subsection{Extensions and corollaries}\label{subsec:local_limit_extension}
In this subsection we complete the proof of Theorem~\ref{thm:local-limit1} by dealing with the final graph case $\hat{t}=\infty$. 
This also enables us to prove \refC{cor:local-limit-discrete}.

\begin{definition}[Neighborhood in the final graph]
Let us extend the notion of $\nbh{n}{\hat{t}}{v}$ and $\nbh{\infty}{\hat{t}}{\rt}$ (see \refD{def:B-nR}) to the cases when $\hat{t}$ is infinite as follows:
\begin{itemize}
\item $\nbh{n}{\infty}{v}$ denotes the $R$-radius neighborhood of vertex $v$ in $G^n_{\underline{p}}(\infty)$,
\item $\nbh{\infty}{\infty}{\rt}$ denotes the $R$-radius neighborhood of $\rt$ in $G^\infty_{\underline{p}}(\infty)$.
\end{itemize}
\end{definition}

\begin{proof}[Proof of Theorem~\ref{thm:local-limit1} for $\hat{t}=\infty$]
Similar to the $\hat{t}<\infty$ case in Section~\ref{subsec:alarm_prob}, 
we first show that the relation \eqref{eq:local-weak-dist} holds between $G^n_{\underline{p}}(\infty)$ and the random graph $G^\infty_{\underline{p}}(\infty)$ introduced in Definition~\ref{def:rdcp-on-pwit}, i.e.,
\begin{equation}\label{eq:final-graph-goal}
\prob \left(\nbh{n}{\infty}{\emb{{\rt}}} \simeq (H,v) \right) \longrightarrow \prob \left(\nbh{\infty}{\infty}{\rt} \simeq (H,v) \right) \quad \text{ as } n \to \infty
\end{equation}
for every rooted graph $(H,v)$.

We will show that for any $\varepsilon>0$, there exists $n_0 \in \mathbb{N}$ such that for any $n \ge n_0$ there is a coupling between $\nbh{n}{\infty}{\emb{{\rt}}}$ and $\nbh{\infty}{\infty}{\rt}$ such that
\begin{equation}\label{eq:isomorph-whp}
\prob \left(\nbh{n}{\infty}{\emb{{\rt}}} \simeq \nbh{\infty}{\infty}{\rt} \right) \ge 1-\varepsilon.
\end{equation}
This will indeed immediately imply \eqref{eq:final-graph-goal}.

First, note that
\begin{equation}\label{eq:Bt-equals-Binfty}
\begin{array}{cc}
\text{if each vertex of $\nbh{n}{t}{\emb{{\rt}}}$ and $\nbh{\infty}{t}{\rt}$ has saturated before time $t$, then} \\
 \text{we have } \nbh{n}{t}{\emb{{\rt}}} \simeq \nbh{n}{\infty}{\emb{{\rt}}} \text{ and } \nbh{\infty}{t}{\rt} \simeq \nbh{\infty}{\infty}{\rt}.
\end{array}
\end{equation}

Observe that we have
\begin{equation*}
\nbh{\infty}{t}{\rt} \stackrel{\text{a.s.}}{\longrightarrow} \nbh{\infty}{\infty}{\rt} \quad \text{ as } t \to \infty,
\end{equation*}
since $\nbh{\infty}{\infty}{\rt}$ is almost surely a finite tree and each phantom saturation time is almost surely finite (see \refC{cl:phantom-sat-finite} and \refD{def:mtbp}). Therefore, for any $R \in \mathbb{N}$ and $\varepsilon>0$, there exists $t^*=t^*(R,\varepsilon)<\infty$ such that
\begin{equation}\label{eq:each-vertex-sat-inf}
\prob \left( \, \forall \, \ibf \in V(\nbh{\infty}{t^*}{\rt}) \, : \, T_{\ibf} < t^* \, \right) > 1- \frac{\varepsilon}{2},
\end{equation}
where $ T_{\ibf}$ denotes the phantom saturation time of vertex $\ibf$ in the RDCP on the PWIT (see \refD{def:phantom-sat}), i.e., \eqref{eq:each-vertex-sat-inf} shows that with high probability, every vertex of $\nbh{\infty}{\infty}{\rt}$ has phantom saturated before time $t^*$.

The identity \eqref{eq:Bt-equals-Binfty} implies
\begin{equation}\label{eq:isomorph1}
\{ \, \forall \, \ibf \in V(\nbh{\infty}{t^*}{\rt}) \, : \, T_{\ibf} < t^* \} \subseteq \{ \nbh{\infty}{t^*}{\emb{{\rt}}} \simeq \nbh{\infty}{\infty}{\rt} \}.
\end{equation}

We couple $G^n_{\underline{p}}(t)$ and $G^\infty_{\underline{p}}(t)$ as in \refS{subsec:coupling_to_rbrw} until time $t^*$. From that point on, we continue building the two graphs (i.e., we try to add further edges when they activate, but only add them if the degree constraints at the endpoints are not violated) independently.

By \refL{lem:alarm}, we know that given $R$, $t^*$ and $\varepsilon$, there exists $n_0$ such that the probability that the alarm rings during the coupling is less than ${\varepsilon/2}$ for any $n \ge n_0$. Moreover, by the second point of \refC{cor:decorated-synch}, \begin{equation}\label{eq:isomorph2}
 \text{if the alarm does not ring until time $t^*$, then } \nbh{n}{t^*}{\emb{{\rt}}} \simeq \nbh{\infty}{t^*}{\rt}.
\end{equation}

Moreover, if the alarm does not ring until $t^*$, then $\modtree \equiv \rbrwtree$ and we also know from Lemma~\ref{lem:R-neighborhood} that the $t^*$-truncated phantom saturation times (cf.\ Definition~\ref{def_tr_st}) of the vertices in the first $R$ generations of $\modtree$ and $\rbrwtree$ coincide, and hence by \eqref{eq:isomorph1} and \eqref{eq:isomorph2}, we also have
\begin{equation}\label{eq:isomorph3}
 \{ \text{no alarm by time } t^* \} \cap \{ \, \forall \, \ibf \in V(\nbh{\infty}{t^*}{\rt}) \, : \, T_{\ibf} < t^* \}
 \subseteq \left\{
 \nbh{n}{\infty}{\emb{{\rt}}} \simeq \nbh{\infty}{\infty}{\rt} \right\}.
\end{equation}

Since the probability of the event on the l.h.s.\ of \eqref{eq:isomorph3} is at least $1-\varepsilon$ if $n \geq n_0$, we obtain \eqref{eq:isomorph-whp}.

Finally, the stronger notion of local convergence in probability \eqref{eq:local-weak-in-prob} follows by fairly routine modifications of the above argument, 
which we omit (they are similar to the ones outlined for the $\hat{t}<\infty$ case in Section~\ref{subsec:alarm_prob}). 
\end{proof}

\begin{proof}[Proof of \refC{cor:local-limit-discrete}]
The expected degree of the root in $G^\infty_{\underline{p}}(F_{\underline{p}}^{-1}(2s))$ is equal to~$2s$ by \refD{def:F(t)}.
Thus the empirical average of the degrees of the vertices of $G^n_{\underline{p}}(F_{\underline{p}}^{-1}(2s))$ converges to $2s$ as $n \to \infty$ by \refT{thm:local-limit1}. In other words, we have
\begin{equation*}
\frac{|E(G^n_{\underline{p}}(F_{\underline{p}}^{-1}(2s)))|}{|V(G^n)|} \stackrel{\mathbb{P}}{\longrightarrow} s, \qquad n \to \infty.
\end{equation*}
Thus, by Claim~\ref{cl:final_graphs} (i.e., using the natural coupling of the discrete- and continuous-time versions of the RDCP on $G^n$), 
for any $\varepsilon>0$, we have
\begin{equation*}
 \lim_{n \to \infty} \mathbb{P}\left( G^n_{\underline{p}}(F_{\underline{p}}^{-1}(2s)-\varepsilon) \subseteq
 G_{\underline{p}}^{n, k_n(s)} \subseteq G^n_{\underline{p}}(F_{\underline{p}}^{-1}(2s)+\varepsilon) \right)
 =1.
\end{equation*}
Now the claim of \refC{cor:local-limit-discrete} follows from \refT{thm:local-limit1} if we let $\varepsilon \to 0$ and use the fact that the distribution of $G_{\underline{p}}^\infty(\hat{t})$ is a continuous function in $\hat{t}$ with respect to the topology induced by local weak convergence. 

Finally, the statement pertaining to the local weak convergence of the discrete-time final graphs follows from the analogous statement about the continuous-time final graphs, since they have the same distribution, cf.~Claim~\ref{cl:final_graphs}.
\end{proof}

\section{RTP and MTBP representations of RDCP on PWIT}\label{subsec:rtp_mtbp}
In this section our main goal is to prove the MTBP characterization of $G^\infty_{\underline{p}}(\hat{t})$ stated in Theorem~\ref{thm:mtbp-reproduces}. 
We begin by rigorously deducing the ODE \eqref{lambda_ode_first_intro} that characterizes~$\lambda(\cdot)$, then we prove Proposition~\ref{prop:rtp_satu} (i.e., the forward recursion for saturation times).
Then we deduce some properties of phantom saturation times and prove \refT{thm:mtbp-reproduces}.

Recall the notation of $A_{t_0} := \{\,(t,s) \in \mathbb{R}_+^2 \, | \, 0 \le t \le t_0,\, t \le s \,\}$ from \eqref{eq:A-def}.

Recall the definition of the RDCP on the PWIT (see \refD{def:rdcp-on-pwit}) and the notion of phantom saturation times from \refD{def:phantom-sat}. Recall the definition of the thinned point process $\underline{\tau}_{\bullet}$ from Definition~\ref{def:rde}.

\begin{definition}[Counting process of children of the root]\label{def:counting_proc_of_chr}
Let $\tau^j$ denote the label (i.e., the activation time) of the edge that connects the root vertex $\rt$ of the PWIT to its $j^{\text{th}}$ child. Let $T_j$ denote the phantom saturation time of the $j^{\text{th}}$ child of the root.
Let us define the counting process $X_t$ of the point process $\underline{\tau}_{\bullet}$ by
\begin{equation}\label{eq:Xt-def}
X_t := \Big| \lbrace(\tau^j, T_j)_{j=1}^\infty\rbrace \cap A_t \Big|
=\sum_{j=1}^{\infty} \indic{\tau^j \leq t} \cdot \indic{\tau^j < T_j}= \sum_{\ell=1}^{\infty} \indic{\tau^\ell_{\bullet} \leq t}, \quad t \in \mathbb{R}_+.
\end{equation}
\end{definition}

Note that $X_t$ counts the number of children of $\rt$ at time $t$ in the RDCP on the PWIT if we ignore the degree constraint of $\rt$ (noting that we can only add the edge with activation time $\tau^j$ if $ \tau^j< T_j$).
Let us define
\begin{equation}\label{def_lambda_first}
\Gamma(t):= \mathbb{P}( T \leq t ), \qquad \lambda(t) := \int \limits_0^t (1-\Gamma(s)) \, \mathrm{d}s.
\end{equation}

\begin{claim}[PPP of the children]\label{cl:X-is-counting-process}
The points $(\tau^j, T_j) \in \mathbb{R}_+^2,$ $j=1, \,2, \, \dots$ form a PPP on $\mathbb{R}_+^2$ with intensity measure $\zeta$ defined as follows:
\begin{equation}\label{eq:intensity-meas}
\zeta\left(\, (t_1,t_2] \times (s_1, s_2]\,\right) := (t_2 - t_1)(\Gamma(s_2)-\Gamma(s_1)).
\end{equation}

Hence $\underline{\tau}_{\bullet}$ is an inhomogeneous Poisson point process on $\mathbb{R}_+$ with intensity $1-\Gamma(t)$ at time $t$. In particular, $X_t$ has Poisson distribution with parameter $\lambda(t)$.
\end{claim}

\begin{proof}
It follows from the coloring property of Poisson point processes and the fact that $T_1, \, T_2, \, \dots$ have the same distribution as $T_\rt$ and are independent of each other and the PPP~$(\tau^j)_{j=1}^\infty$.
\end{proof}

\begin{notation}\label{not:D-Xt} Let
$\underline{p}$ satisfy the mild Assumption~\ref{assump:degree_constraint_mild}.
Let us denote by $D$ a random variable with distribution $\underline{p}$ that is independent of~$X_t$ (cf.~Definition~\ref{def:counting_proc_of_chr}). We have
\begin{equation}\label{eq:def-pk-qk}
p_k = \prob(D=k) \quad \text{ and } \quad q_k \stackrel{\eqref{eq:def-qk}}{=} \prob(D\ge k) \quad \text{for } k =1,2, \dots.
\end{equation}
\end{notation}

Our next lemma makes the sketch outlined in Remark~\ref{remark:rde_ode_sketch} rigorous.

\begin{lemma}[Unique well-behaved solution]\label{lem:unique_rde}
Let
$\underline{p}$ satisfy the mild Assumption~\ref{assump:degree_constraint_mild}. Then
\begin{enumerate}[label=(\alph*)]
\item\label{item:rde_unique}
the RDE introduced in \refD{def:rde} has a unique solution,
\item\label{item:T_finite_ET_infinite} if $T$ is a random variable whose distribution solves the RDE, then
\begin{equation} \mathbb{P}(T<+\infty)=1 \qquad \text{and} \qquad \mathbb{E}(T)=+\infty,\end{equation} 
\item\label{item:density}
$T$ is absolutely continuous with p.d.f.\ $f$ that satisfies \eqref{eq:f-with-lambda-rdcp}, where $\lambda(\cdot)$ solves the ODE \eqref{eq:lambda-ivp} with initial condition \eqref{eq:lambda-ivp-null},
\item\label{item:lambda_prime} $\lambda'(t)$ has the following probabilistic interpretation:
\begin{equation}
\lambda'(t) = \prob(D>X_t) = \prob(T > t), \label{eq:lambda-diff-prob}
\end{equation}
\item\label{item:lambda_zero_infty} $\lambda(t)$ satisfies
\begin{equation}\label{eq:lambda-lim}
\lim \limits_{t \to \infty} \lambda'(t) = 0
\quad \text{ and } \quad
\lim \limits_{t \to \infty} \lambda(t) = \infty.
\end{equation}%
\end{enumerate}
\end{lemma}

\begin{proof}
Let $T$ be an $\mathbb{R}_+ \cup \{ \infty \}$-valued random variable with a distribution that solves the RDE. Using the distribution of $T$, we can repeat Definition~\ref{def:counting_proc_of_chr}, \eqref{def_lambda_first} and Claim~\ref{cl:X-is-counting-process}. Since the distribution of $T$ solves the RDE, we have
\begin{equation*}
\mathbb{P}\left( \chi[\underline{\tau},D](T_1, T_2, \dots) > s \right)=1-\Gamma(s) \quad \text{for any } s \in \mathbb{R}_+,
\end{equation*}
moreover, by the definition \eqref{eq:rde-recursion} of $\chi$, we have $\{\chi[\underline{\tau},D](T_1, T_2, \dots) >s\}=\{X_s < D\}$, thus
the functions $\Gamma$ and $\lambda$ satisfy
\begin{equation}\label{eq:lambda-integral-equation}
1-\Gamma(s) = \sum \limits_{k=1}^{\infty} p_k \cdot \sum \limits_{l=0}^{k-1} \mathrm{e}^{-\lambda(s)} \cdot \frac{\lambda(s)^l}{l!}, \quad s \geq 0.
\end{equation}
Thus if we define the function $\Psi\colon \mathbb{R}_+ \to [0,1]$ by $\Psi(\lambda):=\sum_{k=1}^{\infty} p_k \cdot \sum_{l=0}^{k-1} \mathrm{e}^{-\lambda} \cdot \frac{\lambda^l}{l!}$ and integrate both sides of \eqref{eq:lambda-integral-equation} from $0$ to $t$, then by \eqref{def_lambda_first}, we obtain that the function $\lambda$ that arises from a solution of the RDE satisfies the integral equation $\lambda(t)=\int_0^t \Psi(\lambda(s))\, \mathrm{d}s, \, t \geq 0$. Observing that $\Psi$ is Lipschitz continuous, we conclude that there is a unique differentiable function $\lambda(\cdot)$ that solves 
\eqref{eq:lambda-integral-equation}.
We have $\lambda'(t)= \Psi(\lambda(t))$, thus $\lambda(\cdot)$ satisfies the ODE~\eqref{eq:lambda-ivp} (since $\Psi(\lambda)=\mathrm{e}^{-\lambda} \cdot \sum_{k=0}^{\infty} \frac{\lambda^k}{k!} q_{k+1}$) and the initial condition~\eqref{eq:lambda-ivp-null}.
The first equality of~\eqref{eq:lambda-diff-prob} follows from \eqref{eq:lambda-ivp} and the second follows from \eqref{def_lambda_first}.
We have proved~\ref{item:rde_unique} and~\ref{item:lambda_prime}.

Next we prove \ref{item:lambda_zero_infty}.
Let us indirectly assume that $\lim_{t \to \infty} \lambda(t) \neq \infty$. Since $\lambda(t)$ is increasing (see \eqref{eq:lambda-diff-prob}), our indirect assumption implies that $\lim_{t \to \infty} \lambda(t) = c < \infty$ for some positive constant~$c$. But then $\lambda'(t) = \prob(D>X_t) \geq \prob(X_t=0) \geq \mathrm{e}^{-c}$ since $\prob(D \geq 1 )=1$ follows from the mild Assumption~\ref{assump:degree_constraint_mild}. Now $\lambda'(\cdot) \geq \mathrm{e}^{-c}$ implies ${\lim_{t \to \infty}\lambda(t) = \infty}$, and this contradicts our indirect assumption. Therefore, ${\lambda(t) \to \infty}$ and ${\lambda'(t)=\prob(D>X_t) \to 0}$ as~${t \to \infty}$, since ${\prob(D <+\infty )=1}$ and ${X_t \Rightarrow \infty}$. The proof of~\ref{item:lambda_zero_infty} is complete.

The proof of~\ref{item:T_finite_ET_infinite}
now readily follows from~\ref{item:lambda_zero_infty}, since \begin{align}
&\mathbb{P}(T=+\infty)=
\lim_{t \to \infty} \prob(T > t) \stackrel{\eqref{eq:lambda-diff-prob}}{=}
\lim_{t \to \infty}\lambda'(t)\stackrel{\eqref{eq:lambda-lim}}{=}0, \\
&\mathbb{E}(T)=\int \limits_0^{\infty} \prob(T > t)\, \mathrm{d}t \stackrel{\eqref{eq:lambda-diff-prob}}{=}
\int \limits_0^{\infty} \lambda'(t)\, \mathrm{d}t\stackrel{\eqref{eq:lambda-ivp-null}}{=}
\lim_{t \to \infty} \lambda(t) \stackrel{\eqref{eq:lambda-lim}}{=}+\infty.
\end{align}

Finally, we prove \ref{item:density}.
The function~$\Psi$ is analytic, thus the solution $\lambda(\cdot)$ of the ODE~${\lambda'(t)= \Psi(\lambda(t))}$ is twice differentiable, thus the distribution of $T$ is absolutely continuous with p.d.f.\ $f$ that satisfies~\eqref{eq:f-with-lambda-rdcp}. The proof of Lemma~\ref{lem:unique_rde} is complete.
\end{proof}
The proof of \refCl{cl:phantom-sat-finite} follows from Claim~\ref{cl:phantom-sat-time-solves-rde} and Lemma~\ref{lem:unique_rde}.

Recall the notion of $T^{\lozenge}_{\ibf}$ and $\eta^{\lozenge}_{\ibf}$ from Definition~\ref{def:sat_times_pwit}.

\begin{proof}[Proof of Proposition~\ref{prop:rtp_satu}]
First, we will show that $T^{\lozenge}_{\ibf} =\tau_{\lozenge}^{\ibf d(\ibf)}$ holds by treating the cases $\eta^{\lozenge}_{\ibf} = 0$ and $\eta^{\lozenge}_{\ibf} = 1$ separately.

On the one hand, if $\eta^{\lozenge}_{\ibf} = 0$ then
all of the neighbors of $\ibf$ that we add are its children, i.e., vertex $\ibf$ saturates as soon as it is connected to $d(\ibf)$ children. Therefore, ${\tau^{\ibf d(\ibf)}_{\bullet} =\tau_{\lozenge}^{\ibf d(\ibf)}=T^{\lozenge}_{\ibf}}$ (cf.\ Definition~\ref{def:thinned_pp_on_pwit} and the first case of \eqref{eq:tau_lozenge_def}).

On the other hand, if $\eta^{\lozenge}_{\ibf} = 1$ then $\ibf$ saturates as soon as it is connected to $d(\ibf)-1$ children and also to its parent. Note that $\eta^{\lozenge}_{\ibf} = 1$ implies that $\tau^{\ibf} < \tau^{\ibf d(\ibf)}_{\bullet}$ (cf.\ Definition~\ref{def:sat_times_pwit}). Thus, by the second case of \eqref{eq:tau_lozenge_def}, we obtain that $T^{\lozenge}_{\ibf} =\tau_{\lozenge}^{\ibf d(\ibf)}$ holds in this case, too.

It remains to prove the formula \eqref{eq:sat_time_explicit} for $\eta^{\lozenge}_{\ibf j}$. By Definition~\ref{def:sat_times_pwit}, $\eta^{\lozenge}_{\ibf j}=1$ if and only if
\begin{itemize}
\item $\tau^{\ibf j}$ is a point of $\underline{\tau}^{\ibf}_\bullet$, i.e., $\tau^{\ibf j} < T_{\ibf j}$ (cf. Definition~\ref{def:thinned_pp_on_pwit}) and
\item $\tau^{\ibf j}$ is one of the first $d(\ibf)$ points of $\underline{\tau}^{\ibf}_{\lozenge}$ (cf.\ \eqref{eq:tau_lozenge_def}), i.e., $\tau^{\ibf j} \leq \tau_{\lozenge}^{\ibf d(\ibf)}$ or equivalently (by the first formula of \eqref{eq:sat_time_explicit}) $\tau^{\ibf j} \leq T^{\lozenge}_{\ibf}$.
\end{itemize}
Thus \eqref{eq:sat_time_explicit} holds. It finishes the proof of Proposition~\ref{prop:rtp_satu}.
\end{proof}

Recall from \refN{not:f(t)-A} and Lemma~\ref{lem:unique_rde} that $f(t)$ denotes the probability density function of $T_\rt$.
By \eqref{def_lambda_first} and $1-\Gamma(s)=\int_s^\infty f(u)\, \mathrm{d}u$ it follows that we
\begin{equation}
\label{eq_lambda_f_lambda_diff}
\lambda(t) = \int \limits_0^t \int \limits_s^\infty f(u) \, \mathrm{d}u \, \mathrm{d}s, \qquad 
\lambda'(t) = \int \limits_t^\infty f(s) \, \mathrm{d}s.
\end{equation}

Recall the notation of $z_k^t$ from \eqref{eq:def-zk} and Notation~\ref{not:D-Xt}.

\begin{claim}[Alternative formulas for $f(t)$ and $z_k^t$]
We have
\begin{equation}\label{eq:f-z-probabilistic-meaning}
\begin{split}
f(t) & = \lim \limits_{h \to 0} \frac{1}{h} \prob(X_t < D \le X_{t+h}) = \prob(D>X_t) \cdot \prob(D=X_t +1), \\
z_k^t & = \prob(X_t = k, D=k+1).
\end{split}
\end{equation}
\end{claim}

\begin{proof}
Follows from \eqref{eq:f-with-lambda-rdcp}, \eqref{eq:def-zk} and \eqref{eq:lambda-diff-prob}.
\end{proof}

Let us recall the edge-labelled MTBP that we introduced in \refD{def:mtbp}.

\begin{remark}
We note that the conditional joint distribution of the types of the offspring of the root given its type is different from that of other vertices. For the root, there are $d(\rt)-1$ children as in \eqref{eq:mtbp-joint-dens} and there is one last child with activation time equal to $T_\rt$ and phantom saturation time as in \eqref{eq:dth-dens}, and the types of all $d(\rt)$ children are conditionally independent given $T_\rt$ and $d(\rt)$. For the other vertices we have only $d(\ibf)-1$ children, which can be generated like the first $d(\rt)-1$ children of the root.
\end{remark}

\begin{claim}[Probabilistic interpretation of $p^t_k$]
Recall the notation $p_k^t$ from \eqref{eq:d-mass}. We have the following probabilistic interpretation for $p_k^t$
\begin{equation}\label{eq:p_k^t_prob_interpret}
p_k^t = \prob(D=k \,|\, X_t = D-1) = \prob(X_t = k-1 \,|\, X_t = D-1).
\end{equation}
\end{claim}

\begin{proof}[Proof of \refT{thm:mtbp-reproduces}]
First we prove the statement of the theorem in the $\hat{t}=\infty$ case, i.e., we show that the edge-labelled MTBP tree (defined in \refD{def:mtbp}) has the same law as the connected component of the root in the final graph of the RDCP on the PWIT, where the types in the MTBP correspond to the phantom saturation times in the RDCP on the PWIT. For that purpose, we need to show that if we condition on the event that the phantom saturation time of vertex $\ibf$ in the RDCP on the PWIT is $T_{\ibf} = t_0$, then
\begin{enumerate}[label=(\roman*)]
\item\label{item:MTBPi} the conditional probability mass function of $d(\ibf)$ is given by \eqref{eq:d-mass},
\item\label{item:MTBPii} the root has $d(\rt)$ children, while any other vertex $\ibf$ has $d(\ibf)-1$ children,
\item\label{item:MTBPiii} the probability density function of the phantom saturation time of child $d(\rt)$ of the root is \eqref{eq:dth-dens},
\item\label{item:MTBPiv} for the first $d(\rt)-1$ children of $\rt$, the set $\{(\tau^1, T_1), \dots, (\tau^{d(\rt)-1}, T_{d(\rt)-1}) \}$ of pairs have the same distribution as the set of $d(\rt)-1$ i.i.d.\ points of $\mathbb{R}_+^2$ with joint density function \eqref{eq:mtbp-joint-dens} and this random set is also independent of the phantom saturation time of the $d(\rt)^{\text{th}}$ child,
\item\label{item:MTBPv} for the $d(\ibf)-1$ children of $\ibf \neq \rt$, the set $\{(\tau^{\ibf j}, T_{\ibf j}), \, j=1,\dots, d(\ibf)-1 \}$ of pairs have the same distribution as the set of $d(\ibf)-1$ i.i.d.\ points of $\mathbb{R}_+^2$ with joint density function \eqref{eq:mtbp-joint-dens}.
\end{enumerate}

We start by proving~\ref{item:MTBPi}.
\begin{align*}
p_k^{t_0} &= \prob \left(d(\rt)=k \, | \, T_\rt \in [t_0, t_0+\mathrm{d}t] \right) = \frac{\prob(T_\rt \in [t_0, t_0+\mathrm{d}t] \, | \, d(\rt)=k) \cdot p_k}{f(t_0)\, \mathrm{d}t } = \\
& \frac{\lambda'(t_0)\mathrm{e}^{-\lambda(t_0)}\cdot \frac{\lambda(t_0)^{k-1}}{(k-1)!}\cdot p_k}{f(t_0)} \stackrel{\eqref{eq:f-with-lambda-rdcp}}{=} \frac{\frac{\lambda(t_0)^{k-1}}{(k-1)!}\cdot p_k}{\sum \limits_{l=0}^\infty \frac{\lambda(t_0)^{l}}{l!}\cdot p_{l+1}} = \frac{z_{k-1}^{t_0}}{\sum \limits_{l=0}^\infty z_l^{t_0}}.
\end{align*}

For the other vertices, the conditional probability mass function of $d(\ibf)$ is obtained analogously. Thus we obtained~\ref{item:MTBPi}.

Property~\ref{item:MTBPii} is trivial, i.e., the root has $d(\rt)$ children, while any other vertex $\ibf$ has $d(\ibf)-1$ children, since one of its neighbors is its ancestor.

Recall \refCl{cl:X-is-counting-process}. Statements~\ref{item:MTBPiii}, \ref{item:MTBPiv} and~\ref{item:MTBPv} follow from equation~\eqref{eq:intensity-meas} and the properties of Poisson point processes, as we now explain.

Equation~\eqref{eq:intensity-meas} implies that the intensity of the point process $(\tau^j, \, T_j)_{j=1}^{\infty}$ at $(t,s)$ is $\Gamma'(s) = f(s)$. Therefore, the conditional probability density function of the phantom saturation time $T_{d(\rt)}$ of child $d(\rt)$ of the root given $\tau^{d(\rt)}=t_0$ is \eqref{eq:dth-dens}. This proves~\ref{item:MTBPiii}. A fundamental property of PPPs states that under the condition that the PPP $(\tau^j, \, T_j)_{j=1}^{\infty}$ of the children of the root has exactly $d(\rt)-1$ points in the region $A_{t_0}$, these points are independent and identically distributed with joint density function given in \eqref{eq:mtbp-joint-dens}. The independence of the phantom saturation time of the $d(\rt)^{\text{th}}$ child of $\rt$ follows from that the points of a PPP in disjoint regions are independent. Thus we proved~\ref{item:MTBPiv}. Property~\ref{item:MTBPv} follows analogously.

We have thus proved \refT{thm:mtbp-reproduces} in the $\hat{t}=\infty$ case.

Now let us fix $0 \le \hat{t} < \infty$. Since the edge labels also agree on the two processes, we obtain that $\mathcal{M}_{\underline{p}}(\hat{t})$ (introduced in \refD{def:mtbp-at-t}), decorated with vertex types and edge labels, has the same distribution as $G^\infty_{\underline{p}}(\hat{t})$ (introduced in \refD{def:rdcp-on-pwit}) decorated with phantom saturation times and edge labels.
\end{proof}

\section{Further properties of the MTBP and its branching operator}\label{sec:branching_operator}

In Section~\ref{sec:branching_operator}
we study the branching operator $B_{\hat{t}}$ of the MTBP at time $\hat{t}$ (introduced in \refD{def:branching-operator-at-t}) and the principal eigenvalue of this operator. In Section~\ref{subsec:bopint} we only assume that $\underline{p}$ satisfies the mild Assumption~\ref{assump:degree_constraint_mild}, but then we switch to the strict Assumption~\ref{assump:degree_constraint_strict} for the rest of this paper.

In Section~\ref{subsec:bopint} we first derive an integral formula for $B_{\hat{t}}$. We then introduce a measure on~$\mathbb{R}_+$ with Radon--Nikodym derivative denoted by $\rho(\cdot)$ and we introduce the p.d.f.\ $H(\cdot)$ of the first time when the counting process $X_t$ reaches $D-1$ (cf.\ Notation~\ref{not:D-Xt}), noting that the functions $\rho$ and $H$ will become important in Sections~\ref{subsec:PF_theo} and~\ref{subsec:sturm_liouville}, respectively.
We then derive some useful properties of the function $H$, prove $F_{\underline{p}}(t) = \int_0^t (\lambda'(s))^2 \, \mathrm{d}s$ (cf.~Definition~\ref{def:F(t)}) and deduce Claim~\ref{cl:F(t)-inverse} from it. 

In Section~\ref{subsec:PF_theo} we prove that $B_{\hat{t}}$ is a self-adjoint, irreducible Perron--Frobenius operator on $L^2 \left(\mathbb{R}_+, \rho \right)$. We defer the proof of the Hilbert--Schmidt property of $B_{\infty}$ to Appendix~\ref{appendix_main}. These properties imply that $\norm{B_{\hat{t}}}$ is the principal eigenvalue of $B_{\hat{t}}$.

In Section~\ref{subsec:sturm_liouville} we characterize $\norm{B_{\hat{t}}}$ in terms of a second order linear ODE with appropriate boundary conditions. Note that the function $H(\cdot)$ appears at the r.h.s.\ of this ODE. 
We defer some of the details of the proof of Theorem~\ref{thm:PF-eigenvalue} (i.e., the spectral characterization of the critical time $\hat{t}_c(\underline{p})$) to Appendix~\ref{appendix_main}. 

In Section~\ref{subsec:tc_defs_coincide} we prove \refP{prop:tc-equivalence}, i.e., that $t_c(\underline{p}) = F_{\underline{p}}\bigpar{\hat{t}_c(\underline{p})}/2$.

\subsection{The branching operator as an integral operator}\label{subsec:bopint}

In our MTBP the type of a vertex $\ibf$ corresponds to the phantom saturation time $T_{\ibf}$ in the RDCP on the PWIT (see \refT{thm:mtbp-reproduces}), so the type space is $\mathbb{R}_+$. For a MTBP we can find out whether it is subcritical, supercritical or critical by looking at the principal eigenvalue of the branching operator. Let us recall that~$B_{\hat{t}}$ denotes the branching operator of $\mathcal{M}_{\underline{p}}(t)$ (\refD{def:branching-operator-at-t}). In this subsection, we prove an integral formula for~$B_{\hat{t}}$ that we already mentioned, see \eqref{eq:bo_integral_intro}. We also study the function $F_{\underline{p}}(t)$ and prove Claim~\ref{cl:F(t)-inverse}. In this subsection we assume that $\underline{p}$ satisfies the mild Assumption~\ref{assump:degree_constraint_mild}.

\begin{remark}
Note that in our case the joint offspring distribution of the root is different from that of any other vertex in later generations. In the definition of $B_{\hat{t}}$ we use the offspring distribution of later generations.
\end{remark}

\begin{notation}\label{not:E}
Recall the definition of our MTBP from Definition~\ref{def:mtbp}. Let $\underline{p}$ satisfy the mild Assumption~\ref{assump:degree_constraint_mild} and let 
\begin{equation}\label{eq:E-def}
E(t) := \ev(d(\ibf)-1 \, | \, T_{\ibf} = t), \qquad t \in \mathbb{R}_+.
\end{equation}
\end{notation}
$E(t)$ is the expected number of children of a non-root vertex $\ibf$ with type $t$.

\begin{claim}[Formula for $E(t)$]
Recall the notation $z_k^t$ and $p_k^t$ from \eqref{eq:def-zk} and \eqref{eq:d-mass}. We have
\begin{equation}\label{eq:E-formula}
E(t) = \left( \sum \limits_{k=1}^\infty k \cdot p_k^t \right) - 1 = \frac{\sum \limits_{k=1}^\infty k \cdot z_{k-1}^t}{\sum \limits_{k=0}^{\infty} z_k^t}-1 = \frac{\sum \limits_{k=0}^{\infty} k \cdot z_k^t}{\sum \limits_{k=0}^{\infty} z_k^t}.
\end{equation}

Moreover, $E(t)<\infty$ for any $t \in \mathbb{R}_+$.
\end{claim}

\begin{proof}
The formulas in \eqref{eq:E-formula} follow from \eqref{eq:E-def} and \eqref{eq:d-mass}. Furthermore, \eqref{eq:p_k^t_prob_interpret} implies that for any $t \in \mathbb{R}_+$, we have
\begin{equation*}
E(t) \stackrel{\eqref{eq:E-formula},\, \eqref{eq:p_k^t_prob_interpret}}{=} \frac{\sum \limits_{k=0}^\infty k\cdot \prob(X_t = D-1 = k)}{\prob(X_t = D-1)} \le \frac{\sum \limits_{k=0}^\infty k\cdot \prob(X_t = k)}{\prob(X_t = D-1)} = \frac{\lambda(t)}{\prob(X_t = D-1)} < \infty.
\end{equation*}
\end{proof}

Recall from \refN{not:f(t)-A} and Lemma~\ref{lem:unique_rde} that $f(t)$ denotes the probability density function of $T_\rt$.

\begin{claim}[Integral formula for the branching operator]
For a function $\varphi\colon \mathbb{R}_+ \to \mathbb{R}$, we~have
\begin{equation}\label{eq:branch-op-rdcp}
B_{\hat{t}} \varphi(t_0) = \frac{E(t_0)}{\lambda(t_0)} \int \limits_0^\infty f(s) \cdot \left(\hat{t} \wedge t_0 \wedge s \right) \cdot \varphi(s) \, \mathrm{d}s.
\end{equation}
\end{claim}

\begin{proof}
Recall the joint density function $g_{t_0}(t,s)$ introduced in \eqref{eq:mtbp-joint-dens}. We have
\begin{align*}
B_{\hat{t}} \varphi(t_0) &\stackrel{\eqref{eq:branch-op-def}}{=} E(t_0) \cdot \iint \limits_{A_{t_0 \wedge \hat{t}}} g_{t_0}(t,s) \cdot \varphi(s) \, \mathrm{d}s \, \mathrm{d}t \stackrel{\eqref{eq:mtbp-joint-dens}}{=} E(t_0) \cdot \int \limits_0^{\infty} \int \limits_0^{\hat{t} \wedge t_0 \wedge s} \frac{f(s) \varphi(s)}{\lambda(t_0)} \, \mathrm{d}t \, \mathrm{d}s=\\
& \frac{E(t_0)}{\lambda(t_0)} \int \limits_0^\infty f(s) \cdot \varphi(s) \cdot \left(\hat{t} \wedge t_0 \wedge s \right) \, \mathrm{d}s.
\end{align*}
\end{proof}

Before deriving further properties of the branching operator $B_{\hat{t}}$, we introduce some notation.

\begin{definition}[Hitting time of $D-1$ children]
Let $D$ denote a random variable with distribution $\underline{p}$ that satisfies the mild Assumption~\ref{assump:degree_constraint_mild}, which is independent of the process~$X_t$ (see \refN{not:D-Xt}). Let
\begin{equation}\label{eq:tau-def}
\tau := \min \{\, t \,:\, X_t = D-1 \,\}.
\end{equation}
\end{definition}

\begin{remark}\label{rem:p1_zero}
Under the strict Assumption~\ref{assump:degree_constraint_strict} we have $p_1=0$. The reason for this is that we have $\prob(\tau = 0)=p_1$, so if $p_1>0$ then the distribution of $\tau$ has an atom at zero. On the other hand, in Lemma~\ref{lem:H_prop} we will show that if $p_1=0$ then the distribution of $\tau$ is absolutely continuous, and its p.d.f.\ (that we will denote by $H(\cdot)$) will play a key role in the ODE characterization of the principal eigenvalue of $B_{\hat{t}}$ (cf.\ Section~\ref{subsec:sturm_liouville}).
\end{remark}

\begin{definition}[Auxiliary functions]\label{def:auxiliary_rho}
Let us introduce the notation
\begin{equation}\label{eq:H-rho-def}
H(t):= \frac{E(t)f(t)}{\lambda(t)}, \qquad \rho(t):= \frac{\lambda(t)f(t)}{E(t)} ,
\end{equation}
where $\lambda(t)$ is defined by ~\eqref{eq:lambda-ivp} and~\eqref{eq:lambda-ivp-null}, $f(t)$ is the probability density function of $T_{\rt}$ and $E(t)$ is defined in \eqref{eq:E-def}.
\end{definition}

The motivation behind the definition of function $H(t)$ is that if $p_1=0$ then it is the probability density function of $\tau$ defined in \eqref{eq:tau-def} (see \refL{lem:H_prop}). On the other hand, the function $\rho(t)$ has a key role in the analysis of the branching operator: we will show that $B_{\hat{t}}$ is self-adjoint on the space $L^2(\mathbb{R}_+, \, \rho)$ (see \refL{lem:properties-of-branching-operator}).

\begin{lemma}[Properties of function $H(t)$]\label{lem:H_prop}
Let us consider the function $H(t)$ defined in~\eqref{eq:H-rho-def}.%
{\vspace{-0.25em}\begin{enumerate}[label=(\alph*)]
\itemsep 0.125em \partopsep=0pt \parsep 0em 
\item Then
\begin{equation}\label{eq:H-interpret}
H(t) = \frac{\mathrm{d}}{\mathrm{d}t}\prob(D-1 \le X_t) \stackrel{\eqref{eq:tau-def}}{=} \frac{\mathrm{d}}{\mathrm{d}t}\prob(\tau \le t).
\end{equation}
Therefore,
\begin{equation}\label{eq:H-int}
\int \limits_0^\infty H(s) \, \mathrm{d}s = 1 - p_1.
\end{equation}
\item The cumulative distribution function of $\tau$:
\begin{equation}\label{eq:H-primitive}
\prob(\tau \le t) = p_1 + \int \limits_0^t H(s) \, \mathrm{d}s = 1-\prob(X_t \le D-2) = 1- \mathrm{e}^{-\lambda(t)} \cdot \sum \limits_{k=0}^{\infty} \frac{\lambda(t)^k}{k!} \cdot q_{k+2}.
\end{equation}
\item $H(t)$ is bounded:
\begin{equation}\label{eq:H-bound}
0 \le H(t) \le 1.
\end{equation}%
\vspace{-0.25em}\end{enumerate}}%
\end{lemma}

\begin{proof}$ $
First, note that by the definition of $H(t)$ we have
\begin{align}
\begin{split}\label{eq:H-formula}
H(t) &\stackrel{\eqref{eq:H-rho-def}}{=} \frac{E(t)f(t)}{\lambda(t)} \stackrel{\eqref{eq:f-with-lambda-rdcp},\, \eqref{eq:E-formula}}{=} \frac{1}{\lambda(t)} \cdot \frac{\sum \limits_{k=0}^{\infty} k \cdot z_k^t}{\sum \limits_{k=0}^{\infty} z_k^t} \cdot \lambda'(t) \cdot \sum \limits_{k=0}^{\infty} z_k^t \stackrel{\eqref{eq:def-zk}}{=}\\
& \lambda'(t)\mathrm{e}^{-\lambda(t)} \cdot \sum \limits_{k=0}^{\infty} \frac{\lambda(t)^k}{k!} \cdot p_{k+2} \stackrel{\eqref{eq:def-qk}}{=} \frac{\mathrm{d}}{\mathrm{d} t} \left( -\mathrm{e}^{-\lambda(t)} \cdot \sum \limits_{k=0}^{\infty} \frac{\lambda(t)^k}{k!} \cdot q_{k+2} \right).
\end{split}
\end{align}

Now the statements of the lemma follow from \eqref{eq:H-formula}.

\begin{enumerate}[label=(\alph*)]
\item Note that 
\begin{align*}
H(t) &\stackrel{\eqref{eq:H-formula}}{=} \frac{\mathrm{d}}{\mathrm{d}t} \left( - \mathrm{e}^{-\lambda(t)} \cdot \sum \limits_{k=0}^{\infty} \frac{\lambda(t)^k}{k!} \cdot q_{k+2} \right) = \frac{\mathrm{d}}{\mathrm{d}t} \left( -\prob(D-2 \ge X_t) \right) =\\
&\frac{\mathrm{d}}{\mathrm{d}t} \left( \prob(D-2 < X_t) \right) = \frac{\mathrm{d}}{\mathrm{d}t} \left( \prob(D-1 \le X_t) \right).
\end{align*}
\item Note that $\prob(\tau = 0) = p_1$ and
\begin{equation*}
\prob(\tau \le t) \stackrel{\eqref{eq:tau-def}}{=} \prob(X_t \ge D-1) = 1-\prob(X_t < D-1) = 1-\prob(X_t +2 \le D) .
\end{equation*}

\item Non-negativity of $H(t)$ is trivial. The upper bound follows from \eqref{eq:H-formula}:
\begin{equation*}
H(t) \stackrel{\eqref{eq:H-formula}}{=} \lambda'(t)\mathrm{e}^{-\lambda(t)} \cdot \sum \limits_{k=0}^{\infty} \frac{\lambda(t)^k}{k!} \cdot p_{k+2} \stackrel{p_k\le 1}{\le} \lambda'(t) \stackrel{\eqref{eq:lambda-diff-prob}}{\le} 1.
\end{equation*}
\end{enumerate}
\end{proof}

Recall from \refD{def:F(t)} that $F_{\underline{p}}(t)$ denotes the expected number of neighbors of the root in $G^\infty_{\underline{p}}(t)$. Note that, by \refT{thm:mtbp-reproduces}, $F_{\underline{p}}(t)$ is equal to the expected number of neighbors of the root in $\mathcal{M}_{\underline{p}}(t)$. We prove a useful formula for $F_{\underline{p}}(t)$ in \refL{lem:F(t)-alternative} that will help us to prove Claim~\ref{cl:F(t)-inverse}.

\begin{claim}[Expected number of neighbors]
By the definition of $F_{\underline{p}}(t)$, it is clear that
\begin{equation}\label{eq:F(t)-def}
F_{\underline{p}}(t) = \int \limits_0^\infty (B_t(\bbone))(s) \cdot f(s) \, \mathrm{d}s + \int \limits_0^t f(s) \, \mathrm{d}s.
\end{equation}
\end{claim}

\begin{lemma}[Alternative formula for $F_{\underline{p}}(t)$]\label{lem:F(t)-alternative}
Let $\underline{p}$ satisfy the mild Assumption~\ref{assump:degree_constraint_mild}. We have
\begin{equation}\label{eq:F(t)-formula}
F_{\underline{p}}(t) = \int_0^t (\lambda'(s))^2 \, \mathrm{d}s,
\end{equation}
where $\lambda(t)$ is defined by~\eqref{eq:lambda-ivp} and~\eqref{eq:lambda-ivp-null}.
\end{lemma}

\begin{proof}
Since both sides of \eqref{eq:F(t)-formula} is $0$ at $t=0$, it is enough to show that
\begin{equation}\label{eq:F'}
F_{\underline{p}}'(t) = (\lambda'(t))^2.
\end{equation}

By $\eqref{eq:branch-op-rdcp}$ and using integration by parts we have
\begin{equation}\label{eq:B-indicator}
(B_t(\bbone))(s) = \frac{E(s)\lambda(t \wedge s)}{\lambda(s)}.
\end{equation}
Hence
\begin{align*}
F_{\underline{p}}'(t) &\stackrel{\eqref{eq:F(t)-def},\, \eqref{eq:B-indicator}}{=} E(t)f(t) + \lambda'(t) \cdot \int \limits_t^\infty \frac{E(s)f(s)}{\lambda(s)} \, \mathrm{d}s - E(t)f(t) + f(t) \stackrel{\eqref{eq:f-with-lambda-rdcp},\, \eqref{eq:H-rho-def}}{=}\\
& \lambda'(t) \cdot \int \limits_t^\infty H(s) \, \mathrm{d}s + \lambda'(t)\mathrm{e}^{-\lambda(t)} \cdot \sum \limits_{k=0}^\infty \frac{\lambda(t)^k}{k!} \cdot p_{k+1} \stackrel{\eqref{eq:H-primitive}}{=}\\
& \lambda'(t) \cdot \mathrm{e}^{-\lambda(t)} \cdot \sum \limits_{k=0}^\infty \frac{\lambda(t)^k}{k!} \cdot (q_{k+2}+p_{k+1}) \stackrel{\eqref{eq:lambda-ivp}}{=} (\lambda'(t))^2,
\end{align*}
thus \eqref{eq:F'} holds.
\end{proof}

\begin{remark}
Recall from \eqref{eq:lambda-diff-prob} that $\lambda'(t) = \prob(T_{\rt} > t)$, where $T_{\rt}$ is the phantom saturation time of $\rt$ in the RDCP on the PWIT and we also know that $T_\rt$ is equal to the saturation time of $\rt$. Hence, the probability that two randomly chosen vertices are unsaturated in the RDCP on $K_n$ is approximately $(\lambda'(t))^2$. It implies that the expected number of edges added in the RDCP on $K_n$ until time $t$ is approximately $\frac n2 \int_0^t (\lambda'(s))^2 \, \mathrm{d}s$. This argument provides an alternative heuristic proof of~\eqref{eq:F(t)-formula}.
\end{remark}

Using the above formulas for $F_{\underline{p}}(t)$, we can conclude the following claim that immediately implies Claim~\ref{cl:F(t)-inverse}.

\begin{claim}[Properties of $F_{\underline{p}}(t)$]\label{cl:F(t)-prop}
Let $\underline{p}$ satisfy the mild Assumption~\ref{assump:degree_constraint_mild}. The function~${F_{\underline{p}}:\, \mathbb{R}_+ \to \mathbb{R}_+}$ is continuous, strictly monotone increasing, $F_{\underline{p}}(0)=0$ and
\begin{equation}\label{eq:F(t)-limit}
\lim \limits_{t \to \infty} F_{\underline{p}}(t) = \ev(D).
\end{equation}
\end{claim}

\begin{proof}
Equation \eqref{eq:F(t)-formula} implies continuity, monotonicity and $F_{\underline{p}}(0)=0$. Equation \eqref{eq:F(t)-limit} follows from \eqref{eq:F(t)-def}, by noting~that
\begin{align*}
\lim \limits_{t \to \infty} F_{\underline{p}}(t) \stackrel{\eqref{eq:F(t)-def}}{=} \int \limits_0^\infty (B_\infty(\bbone))(s) \cdot f(s) \, \mathrm{d}s + \int \limits_0^\infty f(s) \, \mathrm{d}s \stackrel{\eqref{eq:branch-op-def}}{=} \int \limits_0^\infty \ev(D-1) \cdot f(s) \, \mathrm{d}s + 1 = \ev(D).%
\end{align*}
\end{proof}

\subsection{Perron--Frobenius theory}\label{subsec:PF_theo}

From this section, we assume that $\underline{p}$ satisfies the strict Assumption~\ref{assump:degree_constraint_strict}. Recall the function $\rho$ from \eqref{eq:H-rho-def}.

\begin{notation}\label{not:eltwo}
On the space $L^2 \left(\mathbb{R}_+, \rho \right)$ we denote the scalar product by $\langle \cdot, \cdot \rangle_\rho$ and the norm by $\norm{\cdot}_2$. I.e., for any $g, h \in L^2 \left(\mathbb{R}_+, \rho \right)$:
\begin{equation*}
\langle g, h \rangle_\rho := \int \limits_0^\infty g(t) h(t) \rho(t) \, \mathrm{d}t, \quad \norm{g}_2 := \sqrt{\langle g,g \rangle_\rho} = \left( \int \limits_0^\infty g(t)^2 \rho(t) \, \mathrm{d}t \right)^{1/2}.
\end{equation*}
\end{notation}

\begin{lemma}[Properties of $B_{\hat{t}}$]\label{lem:properties-of-branching-operator} Let $\underline{p}$ satisfy the strict Assumption~\ref{assump:degree_constraint_strict}.
For any ${\hat{t} \in \mathbb{R}_+ \cup \{+\infty\}}$
the operator $B_{\hat{t}}$ is 
\begin{enumerate}[label=(\roman*)]
\item\label{item:bo_sa} self-adjoint on the space $L^2(\mathbb{R}_+, \, \rho)$,
\item\label{item:bo_hs} Hilbert--Schmidt on the space $L^2(\mathbb{R}_+, \, \rho)$ (and therefore $\norm{B_{\hat{t}}} <+\infty$).
\end{enumerate}
\end{lemma}

\begin{proof}[Proof of Lemma~\ref{lem:properties-of-branching-operator}~\ref{item:bo_sa}]
Note that by \eqref{eq:branch-op-rdcp}, we know that
\begin{equation}\label{eq:kernel}
(B_{\hat{t}} \varphi) (u) = \int \limits_0^\infty \mathcal{K}_{\hat{t}}(u,s) \varphi(s) \rho(s) \, \mathrm{d}s,
\end{equation}
where $\mathcal{K}_{\hat{t}}(u,s) = \frac{E(u)\cdot E(s)}{\lambda(u) \cdot \lambda(s)} \cdot (\hat{t} \wedge u \wedge s)$ is the kernel of the branching operator. The symmetry of $\mathcal{K}_{\hat{t}}(u,s)$ implies that $B_{\hat{t}}$ is self-adjoint on the space $L^2 \left(\mathbb{R}_+, \rho \right)$ for any $\hat{t} \in \mathbb{R}_+ \cup \{+\infty\}$.
\end{proof}

We defer the proof of Lemma~\ref{lem:properties-of-branching-operator}~\ref{item:bo_hs} to Appendix~\ref{appendix_main} (noting that this is one of our proofs, where we use that $\underline{p}$ satisfies $p_k=0$ for all $k > \Delta$, as stated in Assumption~\ref{assump:degree_constraint_strict}). We discuss the possible advantages of proving the Hilbert--Schmidt property in the $\hat{t}=\infty$ case as well in Remark~\ref{rem:pf_infty}.

Let us now collect some facts about $B_{\hat{t}}$ that follow from Perron--Frobenius theory applied to Hilbert--Schmidt integral operators.

\begin{lemma}[Eigenvalues, eigenfunctions of $B_{\hat{t}}$]\label{lem:eigen-properties}
Let $\hat{t} \in \mathbb{R}_+ \cup \{+\infty\}$ and let us consider the branching operator $B_{\hat{t}}$. It satisfies the following properties.
{\vspace{-0.25em}\begin{enumerate}
\itemsep 0.125em \partopsep=0pt \parsep 0em 
\item $B_{\hat{t}}$ is compact and has an orthonormal eigenbasis.
\item $\norm{B_{\hat{t}}}$ is the principal eigenvalue of $B_{\hat{t}}$ with multiplicity 1 and with a positive eigenfunction.
\item The absolute values of all other eigenvalues of $B_{\hat{t}}$ are strictly smaller than $\norm{B_{\hat{t}}}$.
\item The only eigenvalue with a non-negative eigenfunction is $\norm{B_{\hat{t}}}$.
\vspace{-0.25em}\end{enumerate}}%
\end{lemma}

\begin{proof}
We have seen in \refL{lem:properties-of-branching-operator} that $B_{\hat{t}}$ is Hilbert--Schmidt and self-adjoint. Therefore, it is compact (see \cite[Theorem VI.22]{RS}) and has an orthonormal eigenbasis (see \cite[Theorem VI.16]{RS}).

The fact that the principal eigenvalue is $\norm{B_{\hat{t}}}$ follows from the self-adjointness (see \cite[Theorem VI.6]{RS}).

Observe that $\mathcal{K}_{\hat{t}}(u,s)>0$ holds if $u,s,\hat{t} \in (0,+\infty)$ (cf.\ \eqref{eq:kernel}). This implies the rest of the statement (cf.\ \cite[Section 6.5 Theorem 5.2]{M}).
\end{proof}

\subsection{Characterization of principal eigenvalue using second order ODE}\label{subsec:sturm_liouville}

The goal of Section~\ref{subsec:sturm_liouville} is to give an alternative characterization of the principal eigenvalue and the corresponding eigenfunction of the branching operator in terms of a second order linear differential equation with boundary conditions, thus turning it into a problem of Sturm--Liouville type.

By \refL{lem:eigen-properties}, we can introduce the following notation.

\begin{definition}[Principal eigenfunction]\label{def:w-def}
Let us fix $\hat{t} \in \mathbb{R}_+ \cup \{+\infty\}$ and consider the branching operator $B_{\hat{t}}$. Let $\mu_{\hat{t}}$ denote its principal eigenvalue (i.e., $\mu_{\hat{t}}=\norm{B_{\hat{t}}}$, cf.\ Lemma~\ref{lem:eigen-properties}) and let $v_{\hat{t}}(t)$ denote the corresponding eigenfunction, normalized in a way that $w_{\hat{t}}'(0) = 1$, where
\begin{equation}\label{eq:w-def}
w_{\hat{t}}(t) := \frac{\lambda(t)v_{\hat{t}}(t)}{E(t)}.
\end{equation}
\end{definition}

In order for the above definition to make mathematical sense, we will prove in the next lemma that the limit that defines $w_{\hat{t}}'(0)$ exists and $w_{\hat{t}}'(0) \in (0,\infty) $ holds. 

\begin{lemma}[Properties of function $w_{\hat{t}}(t)$]\label{lem:w-properties}
Let $\underline{p}$ satisfy the strict Assumption~\ref{assump:degree_constraint_strict}. Let us fix $\hat{t} \in \mathbb{R}_+ \cup \{+\infty\}$ and let us consider the function $w_{\hat{t}}(t)$ introduced in \refD{def:w-def}.%
{\vspace{-0.25em}\begin{enumerate}[label=(\alph*)]
\itemsep 0.125em \partopsep=0pt \parsep 0em 

\item The function $w_{\hat{t}}(t)$ is positive, monotone increasing and concave on $\mathbb{R}_+$.

\item We have
\begin{equation}\label{eq:Hw-integrable}
\int \limits_0^\infty H(s)w_{\hat{t}}(s) \, \mathrm{d}s < \infty.
\end{equation}

\item On the interval $(0,\hat{t})$ the function $w_{\hat{t}}(t)$ solves the following \textbf{second order linear} initial value problem
\begin{align}
\mu_{\hat{t}} w_{\hat{t}}''(t) &= -H(t)w_{\hat{t}}(t) \quad \text{for } 0 < t < \hat{t} \label{eq:w_that-2nd-order-diffeq}\\
w_{\hat{t}}(0) &= 0 \label{eq:w_that-iv-null1}\\
w_{\hat{t}}'(0) &= 1. \label{eq:w_that-iv-null2}
\end{align}

Moreover, if $\hat{t}< \infty$ then $w_{\hat{t}}(t)$ also satisfies the boundary condition
\begin{equation}
\mu_{\hat{t}} \cdot \left(\lim \limits_{t \nearrow \hat{t}}w_{\underline{\hat{t}}}'(t)\right) - w_{\hat{t}}(\hat{t}) \cdot \left( 1- \int \limits_0^{\hat{t}} H(s) \, \mathrm{d}s \right)=0, \label{eq:2nd-order-extra-mu}
\end{equation}
moreover, $w_{\hat{t}}(t)$ is constant for $t \ge \hat{t}$ (i.e., $w_{\hat{t}}(t)=w_{\hat{t}}(\hat{t})$ for all $t \ge \hat{t}$).

On the other hand, if $\hat{t}=\infty$ then for $w_\infty(t)$ we have
\begin{equation}\label{eq:2nd-order-extra-infty}
\lim \limits_{t \to \infty} w_{\infty}'(t) = 0.
\end{equation}%
\vspace{-0.25em}\end{enumerate}}%
\end{lemma}

\begin{proof}
First, note that the non-negativity of $w_{\hat{t}}$ easily follows from its definition \eqref{eq:w-def}, since $\lambda(t)$, $E(t)$ and $v_{\hat{t}}(t)$ are all non-negative on $\mathbb{R}_+$ (note that non-negativity of $v_{\hat{t}}$ follows from \refL{lem:eigen-properties}).

The function $v_{\hat{t}}$ is an eigenfunction of $B_{\hat{t}}$ corresponding to the eigenvalue $\mu_{\hat{t}}$, thus we have $\mu_{\hat{t}} v_{\hat{t}} = B_{\hat{t}}v_{\hat{t}}$. Therefore, by \eqref{eq:branch-op-rdcp} if $t_0 \le \hat{t}$ then
\begin{equation}\label{eq:t0-less-rdcp}
\mu_{\hat{t}} \lambda(t_0) v_{\hat{t}}(t_0) = E(t_0)\left[ \int \limits_0^{t_0} f(s)v_{\hat{t}}(s)s \, \mathrm{d}s + t_0 \cdot \int \limits_{t_0}^\infty f(s) v_{\hat{t}}(s) \, \mathrm{d}s \right],
\end{equation}

and if $\hat{t}< \infty$ and $t_0 \ge \hat{t}$, then we have
\begin{equation}\label{eq:large-t0-rdcp}
\mu_{\hat{t}} \frac{\lambda(t_0) v_{\hat{t}}(t_0)}{E(t_0)} = \left[ \int \limits_0^{\hat{t}} f(s)v_{\hat{t}}(s)s \, \mathrm{d}s + \hat{t} \cdot \int \limits_{\hat{t}}^\infty f(s) v_{\hat{t}}(s) \, \mathrm{d}s \right].
\end{equation}
Observe that the right hand side of equation \eqref{eq:large-t0-rdcp} does not depend on $t_0$, thus the function ${w_{\hat{t}}(t_0) = \lambda(t_0) v_{\hat{t}}(t_0)/E(t_0)}$ is constant for each $t_0 \ge \hat{t}$, i.e., $w_{\hat{t}}(t) \equiv w_{\hat{t}}\left( \hat{t} \right)$.

If $t_0 \le \hat{t}$ then by \eqref{eq:t0-less-rdcp} we have
\begin{equation}\label{eq:w-eq-rdcp}
\mu_{\hat{t}} w_{\hat{t}}(t_0) = \left( \int \limits_0^{t_0} H(s)w_{\hat{t}}(s)s \, \mathrm{d}s + t_0 \cdot \int \limits_{t_0}^\infty H(s) w_{\hat{t}}(s) \, \mathrm{d}s \right).
\end{equation}

By differentiating \eqref{eq:w-eq-rdcp} with respect to $t_0$, we have the following equation
\begin{equation}\label{eq:w-diff-rdcp}
\mu_{\hat{t}} w_{\hat{t}}'(t_0) = \int \limits_{t_0}^\infty H(s)w_{\hat{t}}(s) \, \mathrm{d}s
\quad \text{for } 0 < t_0 < \hat{t}.
\end{equation}

In case $t_0 \nearrow \hat{t} < \infty$, using that $w_{\hat{t}}(t)$ is constant for $t \ge \hat{t}$ if $\hat{t}<\infty$ and $p_1 = 0$ by the strict Assumption~\ref{assump:degree_constraint_strict}, equation~\eqref{eq:w-diff-rdcp} gives~\eqref{eq:2nd-order-extra-mu}:
\begin{equation*}
\mu_{\hat{t}} \cdot \left(\lim \limits_{t_0 \nearrow \hat{t}}w_{\underline{\hat{t}}}'(t_0)\right) = w_{\hat{t}}\left( \hat{t}\right) \int \limits_{\hat{t}}^\infty H(s) \, \mathrm{d}s \stackrel{\eqref{eq:H-int}}{=} w_{\hat{t}}(\hat{t}) \cdot \left( 1- \int \limits_0^{\hat{t}} H(s) \, \mathrm{d}s \right).
\end{equation*}

If we differentiate \eqref{eq:w-diff-rdcp} with respect to $t_0$ we obtain \eqref{eq:w_that-2nd-order-diffeq}.

Now we can prove the monotonicity and the concavity of $w_{\hat{t}}(t)$. Since $w_{\hat{t}}(t)$ is positive, by \eqref{eq:w-diff-rdcp} and \eqref{eq:w_that-2nd-order-diffeq} we know that $w_{\hat{t}}(t)$ is monotone increasing and concave on the interval $[0, \hat{t}]$. In particular, if $\hat{t}=\infty$ then this statement holds on $\mathbb{R}_+$. If $\hat{t}<\infty$ then by \eqref{eq:2nd-order-extra-mu} we also know that $w'_{\hat{t}}(\hat{t})>0$ and $w_{\hat{t}}(t)$ is constant for $t \ge \hat{t}$, therefore, the statement holds also on $\mathbb{R}_+$.

By the boundedness of $B_{\hat{t}}$ (see \refL{lem:properties-of-branching-operator}), we know that $\norm{v_{\hat{t}}}_2 < \infty$, i.e.,
\begin{equation}\label{eq:Hw^2-integrable}
\int \limits_0^\infty H(t) \cdot w_{\hat{t}}(t)^2\, \mathrm{d}t \stackrel{\eqref{eq:H-rho-def},\, \eqref{eq:w-def}}{=}\int \limits_0^\infty v_{\hat{t}}(t)^2 \cdot \rho(t) \, \mathrm{d}t < \infty.
\end{equation}

Note that by the monotonicity and concavity of function $w_{\hat{t}}(t)$, \eqref{eq:Hw^2-integrable} implies also \eqref{eq:Hw-integrable}. Therefore, equation \eqref{eq:w-diff-rdcp} shows that $w_{\hat{t}}'(0) \in (0, \infty)$. Since any constant multiple of an eigenfunction is also an eigenfunction, we can choose $w_{\hat{t}}'(0)=1$, i.e., we obtain~\eqref{eq:w_that-iv-null2}.

It remains to show \eqref{eq:w_that-iv-null1} and \eqref{eq:2nd-order-extra-infty}.

For $t_0=0$ equation~\eqref{eq:w-eq-rdcp} using \eqref{eq:Hw-integrable} gives \eqref{eq:w_that-iv-null1}.

On the other hand, if $\hat{t}=\infty$ and we take $t_0 \to \infty$ in \eqref{eq:w-diff-rdcp}, then using \eqref{eq:Hw-integrable} we obtain~\eqref{eq:2nd-order-extra-infty}.
\end{proof}

We already have all of the ingredients required for the proof of \refT{thm:PF-eigenvalue}.
Let us note that the characterization of sub/supercriticality of a multi-type branching process via the principal eigenvalue of its branching operator (cf.\ statements~\ref{item:subcrit_eigen_statement}.\ and~\ref{item:supcrit_eigen_statement}.\ of \refT{thm:PF-eigenvalue}) is a classical result (see e.g.\ \cite[Chapter~V.3 Theorem~2]{AN} or \cite[Theorem~6.1]{BJR}). However, we have not found this result in the literature in the generality we need it. Therefore, for the sake of completeness, we include the proof that suits our setting, but we defer it to Appendix~\ref{appendix_main} (noting that this is one of our proofs, where we use that $\underline{p}$ satisfies $p_k=0$ for all $k > \Delta$, as stated in the strict Assumption~\ref{assump:degree_constraint_strict}).

\subsection{Equivalence of critical time definitions}
\label{subsec:tc_defs_coincide}

The goal of Section~\ref{subsec:tc_defs_coincide} is to prove \refP{prop:tc-equivalence}, i.e., that $t_c(\underline{p}) = F_{\underline{p}}\bigpar{\hat{t}_c(\underline{p})}/2$. Let us stress that this is the only proof of our paper, where we rely on the results of \cite{WW}.

\begin{proposition}[Sharpness of phase transition of limit object]\label{prop:emergence-of-giant}
Let $\underline{p}$ satisfy the strict Assumption~\ref{assump:degree_constraint_strict}. The critical time $\hat{t}_c(\underline{p})$ defined in \eqref{eq:t_crit-def} (i.e., the time when the susceptibility of the tree $G^\infty_{\underline{p}}(\hat{t})$ explodes) is equal to the infimum of times when $G^\infty_{\underline{p}}(\hat{t})$ has infinitely many vertices with positive probability:
\begin{equation}\label{eq:giant-emerge}
\hat{t}_c(\underline{p}) = \inf \left\{ \, \hat{t} \ge 0 \, \Big| \, \prob \left( |V(G^\infty_{\underline{p}}(\hat{t}))| = \infty \right)>0 \, \right\}.
\end{equation}
\end{proposition}

\begin{proof}
This follows from $\norm{B_{\hat{t}_c(\underline{p})}} = \mu_{\hat{t}_c(\underline{p})} = 1$ and the other statements of \refT{thm:PF-eigenvalue}.
\end{proof}

\begin{proof}[Proof of \refP{prop:tc-equivalence}]
In this proof we denote by $t_c^W(\underline{p})$ the critical time determined in \cite[Theorem 1.6]{WW} and by $t_c(\underline{p})$ the critical time defined in \eqref{eq:discrete-tc-def}. Our goal is to prove that $t_c^W(\underline{p}) = t_c(\underline{p})$.

Let us denote by $Q^{n,k}(t)$ the ratio of the number of vertices whose component has size at most $k$ in~$G^n_{\underline{p}}(t)$. By \refC{cor:local-limit-discrete}, it follows that for any $t \in \mathbb{R}_+$ and $k \in \mathbb{N}$, the random variable $Q^{n,k}(t)$ converges in probability to $Q^{\infty,k}(t)$ as $n \to \infty$, where $Q^{\infty,k}(t) := \prob \left( \, | V(G^\infty_{\underline{p}}(t)) | \le k \right)$. 
Observe that if $t < t_c(\underline{p})$ then
\begin{equation}\label{eq:subcrit-prop}
\forall \, \varepsilon> 0, \, \exists \, k \in \mathbb{N}: \quad Q^{\infty,k}(t) \ge 1- \varepsilon,
\end{equation}
and using \eqref{eq:giant-emerge}, if $t > t_c(\underline{p})$ then
\begin{equation}\label{eq:supercrit-prop}
\exists \, \varepsilon> 0, \, \forall \, k \in \mathbb{N}: \quad Q^{\infty,k}(t) \le 1- \varepsilon.
\end{equation}

On the other hand, the third point of \cite[Theorem 1.6]{WW} implies that if $t<t_c^W(\underline{p})$, the susceptibility of the RDCP on $K_n$ at time $t$ converges to a finite number as $n \to \infty$ and it implies \eqref{eq:subcrit-prop}. Moreover, the second point of \cite[Theorem 1.6]{WW} shows that if $t > t_c^W(\underline{p})$, there exists a giant component in the RDCP on $K_n$ at time $t$ and it implies \eqref{eq:supercrit-prop}.

This argument shows that $t_c^W(\underline{p}) = t_c(\underline{p})$ holds.
\end{proof}

\section{Asymptotic expansion of the critical time}\label{sec:t_crit_asymp}

In this section our main goal is to prove \refT{thm:tc-asymptotics}, i.e., to give an asymptotic expansion of the critical time. 
Recalling from Theorem~\ref{thm:PF-eigenvalue} that $\hat{t}=\hat{t}_c$ is the unique value of $\hat{t}$ for which the principal eigenvalue of $B_{\hat{t}}$ is equal to $1$, we will use the ODE characterization of the corresponding eigenfunction (derived in Lemma~\ref{lem:w-properties}) to give asymptotically matching upper and lower bounds on the value of $\hat{t}_c$. The key idea is simple: we know that if $\hat{t}=\hat{t}_c$, then the value of $\mu_{\hat{t}}$ from \eqref{eq:w_that-2nd-order-diffeq} must be equal to $1$, thus we obtain the eigenfunction $w_{\hat{t}}(\cdot)$ by solving the initial value problem \eqref{eq:w_that-2nd-order-diffeq}--\eqref{eq:w_that-iv-null2}, therefore, we can identify $\hat{t}_c$ as the (leftmost) value of $\hat{t}$ for which the boundary condition \eqref{eq:2nd-order-extra-mu} is also satisfied.

In order to understand the ODE \eqref{eq:w_that-2nd-order-diffeq} from Lemma~\ref{lem:w-properties}, we first need to understand the function $H$ that appears in the r.h.s.\ of the ODE: this is what we will do in Lemma~\ref{lem:error-asymptotics}. Then we can estimate the solution of the ODE and the value of $\hat{t}_c$ in Lemma~\ref{lem:t-crit-asymp}. Finally, in Lemma~\ref{lem:F(tc)-asymp} we prove some bounds that allow us to convert the result \eqref{eq:tc} about $\hat{t}_c$ into the result \eqref{eq:discrete-tc} about $t_c$.

\medskip

 Let $\underline{p}^m$ and $D^m$ be defined as in \refT{thm:tc-asymptotics}, i.e., $\underline{p}^m, \, m=1,2,\dots$ is a sequence of distributions on $\{\,2, 3, \dots\,\}$, $D^m$ is a random variable with distribution~$\underline{p}^m$ and $D^m \Rightarrow \infty$ as $m \to \infty$. Note that $\Delta_m$ (see \eqref{eq:upper-bound-assumption}) can depend on~$m$, and that $\Delta_m \to \infty$ as~$m \to \infty$.

\begin{remark}
$D^m \Rightarrow \infty$ is equivalent to $p^m_k \to 0$ as $m \to \infty$ for any $k \in \mathbb{N}$, or, alternatively, $q^m_k \to 1$ as $m \to \infty$ for any $k$, where $p^m_k = \prob(D^m=k)$ and $q^m_k = \prob(D^m \ge k)$.
\end{remark}

First, we introduce some notation and state Lemmas~\ref{lem:error-asymptotics}, \ref{lem:t-crit-asymp} and~\ref{lem:F(tc)-asymp}. These lemmas will immediately imply \refT{thm:tc-asymptotics}.

\begin{notation}\label{not:crit-asymptotic-notations}
If the degree constraint distribution is $\underline{p}^m$, we use the notation $\lambda_m(t)$, $X^m_t$ and $H_m(t)$ instead of $\lambda(t)$ (see~\eqref{eq:lambda-ivp} and \eqref{eq:lambda-ivp-null}), $X_t$ (see \eqref{eq:Xt-def}) and $H(t)$ (see \eqref{eq:H-rho-def}), respectively. Let us denote by $\mathcal{X}^m$ the PPP corresponding to the counting process $X^m_t$.

Let us introduce the sequences
\begin{equation}\label{eq:delta-I-J-def}
\delta_m:= \int \limits_0^1 H_m(s) \cdot (1-s^2) \, \mathrm{d}s, \quad I_m := \int \limits_0^{1+2\delta_m} H_m(s) \, \mathrm{d}s, \quad J_m := \int \limits_0^1 \prob(D^m \le X^m_s) \, \mathrm{d}s, \quad m \in \mathbb{N}.
\end{equation}

Note that
\begin{equation}\label{eq:delta-smaller-than-I}
\delta_m \le I_m.
\end{equation}

We also use the notation $\mathcal{O}(h_m)$ and $h_m \sim g_m$ for sequences $h_m\colon \mathbb{N} \to \mathbb{R}$, $g_m\colon \mathbb{N} \to \mathbb{R}$:
\begin{align*}
&\mathcal{O}(h_m) := \left\{\, f_m\colon \mathbb{N}\to \mathbb{R} \, | \, \exists \, c \in \mathbb{R}_+ \, : \, \forall \, m \in \mathbb{N} \, : \, |f_m| \le c \cdot h_m \, \right\},\\
&h_m \sim g_m \quad \text{if and only if} \quad \lim \limits_{m \to \infty} \frac{h_m}{g_m} = 1.
\end{align*}
\end{notation}

\begin{lemma}[Order of error terms]\label{lem:error-asymptotics}
As $m \to \infty$ we have
\begin{equation}\label{eq:error-asymptotics}
\delta_m \sim \frac{2}{\mathrm{e}} \ev \left(\frac{1}{D^m!}\right), \qquad \frac{I_m^2}{\delta_m} \longrightarrow 0 \quad \text{ and } \quad \frac{J_m}{\delta_m} \longrightarrow 0.
\end{equation}
\end{lemma}

We prove \refL{lem:error-asymptotics} in \refS{subsec:error_terms}.

\begin{lemma}[Asymptotic formula for $\hat{t}_c$]\label{lem:t-crit-asymp}
Recall \refD{def:t-crit} and the notation of $\delta_m$ and $I_m$ (see~\eqref{eq:delta-I-J-def}). We have
\begin{equation}\label{eq:tc-formula}
\hat{t}_c(\underline{p}^m) = 1+ \delta_m + \mathcal{O}(I_m^2).
\end{equation}
\end{lemma}

We prove \refL{lem:t-crit-asymp} in \refS{subsec:t_crit_asymp_formula}.

Recall that $F_{\underline{p}}(t)$ denotes the expected number of neighbors of the root in $G^\infty_{\underline{p}}(t)$ (see \refD{def:F(t)}). We show that the difference between $\hat{t}_c(\underline{p}^m)$ and $F_{\underline{p}^m}(\hat{t}_c(\underline{p}^m))$ is small.

\begin{lemma}[Difference of $\hat{t}_c$ and $F_{\underline{p}}(\hat{t}_c)$]\label{lem:F(tc)-asymp}
\begin{equation}\label{eq:F(tc)-formula}
F_{\underline{p}^m}(\hat{t}_c(\underline{p}^m)) = \hat{t}_c(\underline{p}^m) - 2 J_m + \mathcal{O}(I_m^2),
\end{equation}
where $J_m$ and $I_m$ are defined in \eqref{eq:delta-I-J-def}.
\end{lemma}

We prove \refL{lem:F(tc)-asymp} in \refS{subsec:F_tc_asymp}.

\begin{proof}[Proof of \refT{thm:tc-asymptotics}]
Lemmas~\ref{lem:t-crit-asymp} and~\ref{lem:error-asymptotics} imply \eqref{eq:tc}. Together with \refL{lem:F(tc)-asymp}, they also imply that
\begin{equation}\label{eq:F(tc)}
\frac{F_{\underline{p}^m}(\hat{t}_c(\underline{p}^m))-1}{\frac{2}{\mathrm{e}} \cdot \ev \left( \frac{1}{D^m!} \right)} \stackrel{m \to \infty}{\longrightarrow} 1.
\end{equation}
By \eqref{eq:discrete-tc-def} equation \eqref{eq:F(tc)} is equivalent to \eqref{eq:discrete-tc}.
\end{proof}

\subsection{Order of error terms}\label{subsec:error_terms}

In this section our goal is to prove \refL{lem:error-asymptotics}. First, we introduce two homogeneous Poisson point processes that approximate $\mathcal{X}^m$ (defined in \refN{not:crit-asymptotic-notations}). After calculating some facts about the asymptotic behavior of these homogeneous PPPs (see \refL{lem:eps-moments}) and proving the bound on $J_m$ stated in \refL{lem:Jm-bound}, we can easily deduce \refL{lem:error-asymptotics}.

\begin{notation}
Recall from \refN{not:crit-asymptotic-notations} that $X^m_t$ is the counting process of the PPP $\mathcal{X}^m$, which has intensity function $\lambda_m'(t)$. Now we define $\tau^m$ similarly to $\tau$ (see \eqref{eq:tau-def}), but for degree distribution $\underline{p}^m$:
\begin{equation}\label{eq:tau_m-def}
\tau_m := \min \{\, t \, : \, X^m_t =D^m-1 \,\}.
\end{equation}
\end{notation}

\begin{definition}[Auxiliary homogeneous PPPs]
Let $\overline{\mathcal{X}}$ be a homogeneous Poisson point process with intensity $1$, independent of $D^m$. Let $\overline{X}_t$ be the counting process of $\overline{\mathcal{X}}$. Similarly, we introduce a homogeneous Poisson point process
$\underline{\mathcal{X}}^m$ with intensity function~$\lambda'_m(1)$, also independent of $D^m$. Let $\underline{X}^m_t$ denote the counting process of $\underline{\mathcal{X}}^m$. Let us introduce the notation $\overline{\tau}_m$ and $\underline{\tau}_m$:
\begin{equation}\label{eq:tau_const-def}
\overline{\tau}_m:= \min \{\, t \, : \, \overline{X}_t = D^m-1 \,\}, \qquad \underline{\tau}_m:= \min \{\, t \, : \, \underline{X}^m_t = D^m-1 \,\}.
\end{equation}

Let
\begin{equation}\label{eq:eps-definitions}
\varepsilon_m := 1- \tau_m, \quad \overline{\varepsilon}_m := 1-\overline{\tau}_m, \quad \underline{\varepsilon}_m := 1-\underline{\tau}_m.
\end{equation}
\end{definition}

\begin{notation}
Let $\mathcal{X}$ and $\mathcal{Y}$ be point processes on $\mathbb{R}_+$. We denote by $\mathcal{X} \le \mathcal{Y}$ if $\mathcal{Y}$ can be obtained from $\mathcal{X}$ by inserting points.

We also introduce the notation $\xi^+$ for the positive part of a random variable $\xi$, i.e., we define $\xi^+ := \xi \cdot \indic{\xi \ge 0}$.
\end{notation}

\begin{claim}[Coupling between Poisson point processes]
For any $m \in \mathbb{N}$, the Poisson point processes $\mathcal{X}^m$, $\overline{\mathcal{X}}$ and $\underline{\mathcal{X}}^m$ can be defined on the same probability space in a way that
\begin{equation}\label{eq:PPPs-connection}
\underline{\mathcal{X}}^m \le \mathcal{X}^m \quad \text{holds on the interval } [0,1], \quad \text{and} \quad \mathcal{X}^m \le \overline{\mathcal{X}} \quad \text{on } \mathbb{R}_+.
\end{equation}

Moreover, we have
\begin{equation}\label{eq:eps-connection}
\underline{\varepsilon}_m^+ \le \varepsilon_m^+ \le \overline{\varepsilon}_m^+.
\end{equation}
\end{claim}

\begin{proof}
The intensity function of the PPPs $\underline{\mathcal{X}}^m$, $\mathcal{X}^m$ and $\overline{\mathcal{X}}$ is $\lambda'_m(1)$, $\lambda_m'(t)$ and $1$, respectively. For any $m \in \mathbb{N}$, the function $\lambda_m'(t)$ is monotone decreasing (see \eqref{eq:lambda-diff-prob}). Therefore, for any $t \in [0,1]$, we have $\lambda'_m(1) \le \lambda_m'(t)$ and for any $t \in \mathbb{R}_+$, we have $\lambda_m'(t) \le 1$. Hence on the interval~$[0,1]$, the PPP $\underline{\mathcal{X}}^m$ can be defined as a thinning of the PPP $\mathcal{X}^m$ and on $\mathbb{R}_+$, the PPP $\mathcal{X}^m$ can be defined as a thinning of the PPP $\overline{\mathcal{X}}$. The proof of \eqref{eq:PPPs-connection} is complete.

By \eqref{eq:tau_m-def} and \eqref{eq:tau_const-def}, the inequalities of \eqref{eq:PPPs-connection} imply that
\begin{equation}\label{eq:tau-connection}
(1-\underline{\tau}_m) \cdot \indic{\underline{\tau}_m \le 1} \le (1-\tau_m) \cdot \indic{\tau_m \le 1} \le (1-\overline{\tau})_m \cdot \indic{\overline{\tau}_m \le 1}.
\end{equation}

By \eqref{eq:eps-definitions}, the inequalities of \eqref{eq:tau-connection} are equivalent to the inequalities of \eqref{eq:eps-connection}.
\end{proof}

\begin{lemma}[Asymptotic behavior in case of constant intensity]\label{lem:eps-moments}
Recall $\overline{\varepsilon}_m$ and $\underline{\varepsilon}_m$, defined in \eqref{eq:eps-definitions}. We have
\begin{align}
\ev \left(\overline{\varepsilon}_m^+ \right) &\sim \frac{1}{\mathrm{e}} \ev \left( \frac{1}{D^m!} \right), \label{eq:eps-upper}\\
\ev \left(\underline{\varepsilon}_m^+ \right) &\sim \frac{1}{\mathrm{e}} \ev \left( \frac{1}{D^m!} \right), \label{eq:eps-lower}\\
\ev \left(\overline{\varepsilon}_m^2 \cdot \indic{\overline{\varepsilon}_m \ge 0} \right) & = \mathcal{O} \left( \ev \left( \frac{1}{(D^m+1)!} \right) \right). \label{eq:eps-second-mom}
\end{align}
\end{lemma}

\begin{proof}
We use Campbell's theorem (see for example \cite[Theorem 5.14]{K}), i.e., that given that a homogeneous PPP on $\mathbb{R}$ has exactly $k$ points in the interval $[0,1]$, then the joint distribution of these $k$ points is the same as the joint distribution of the order statistics of $k$ i.i.d.\ random variables with distribution UNI$[0,1]$.

It is also known (see \cite[Example 3.1.1 and equation (3.1.6)]{DN}) that if $Z_{k,l}$ denotes the~$l^{\text{th}}$ point of this order statistics, then
\begin{equation}\label{eq:lth-point-moments}
\ev(Z_{k,l}) = \frac{l}{k+1} \quad \text{and} \quad \ev \left(Z_{k,l}^2 \right) = \frac{l (l+1)}{(k+1) (k+2)}.
\end{equation}

Observe that given the conditions $D^m=d$ and $\overline{X}_1=k$, where $k\ge d-1$ (hence $\overline{\varepsilon}_m \ge 0$, see \eqref{eq:eps-definitions}), the conditional distribution of $\overline{\varepsilon}_m$ is the same as the distribution of $Z_{k,k-d+2}$. Similarly, given $D^m=d$ and $\underline{X}^m_1=k$, where $k\ge d-1$ (hence $\underline{\varepsilon}_m \ge 0$, see \eqref{eq:eps-definitions}), the conditional distribution of $\underline{\varepsilon}_m$ is the same as the distribution of $Z_{k,k-d+2}$. In particular, by \eqref{eq:lth-point-moments}, we have
\begin{align}
&\ev \left(\overline{\varepsilon}_m \, | \, D^m=d, \, \overline{X}_1=k \right) = \ev \left(\underline{\varepsilon}_m \, | \, D^m=d, \, \underline{X}^m_1=k \right) = \frac{k-d+2}{k+1}, \label{eq:eps-1st-moment} \\
&\ev \left(\overline{\varepsilon}_m^2 \, | \, D^m=d, \, \overline{X}_1=k \right) = \frac{(k-d+2)(k-d+3)}{(k+1)(k+2)}. \label{eq:eps-2nd-moment}
\end{align}

First, we prove \eqref{eq:eps-upper}:
\begin{align}
\begin{split}\label{eq:constant-intensity-asymptotic}
\ev \left( \overline{\varepsilon}_m^+ \right) &\stackrel{\eqref{eq:eps-1st-moment}}{=} \sum \limits_{d=2}^{\Delta_m} p_d^m \sum \limits_{k=d-1}^\infty \mathrm{e}^{-1} \cdot \frac{1}{k!} \cdot \frac{k-d+2}{k+1} =\\
&\sum \limits_{d=2}^{\Delta_m} \left[ p_d^m \mathrm{e}^{-1} \cdot \left( \frac{1}{d!} + \sum \limits_{k=d+1}^\infty \frac{1}{k!} \cdot (k-d+1) \right) \right] \stackrel{(*)}{=}\\
& \frac{1}{\mathrm{e}} \cdot \ev \left( \frac{1}{D^m !}\right) + \sum \limits_{d=2}^{\Delta_m} \left[ p_d^m \mathrm{e}^{-1} \cdot \left( \sum \limits_{k=d+1}^\infty \frac{1}{k!} \cdot (k-d+1) \right) \right].
\end{split}
\end{align}

Using the fact that $\sum \limits_{k=d+1}^\infty \frac{1}{k!} \cdot (k-d+1) \le \frac{2}{(d+1)!} + \sum \limits_{k=d+2}^\infty \frac{1}{(k-1)!} \le \frac{4}{(d+1)!}$, equation \eqref{eq:constant-intensity-asymptotic} implies that
\begin{equation}\label{eq:constant-intensity-upper-1st-mom}
\frac{1}{\mathrm{e}} \cdot \ev \left( \frac{1}{D^m !} \right) \le \ev \left( \overline{\varepsilon}_m^+ \right) \le \frac{1}{\mathrm{e}} \cdot \ev \left( \frac{1}{D^m !}\right) + \frac{4}{\mathrm{e}} \cdot \ev \left( \frac{1}{(D^m+1) !}\right).
\end{equation}

Observing that $\frac{\ev \left( \frac{1}{(D^m+1) !}\right)}{\ev \left( \frac{1}{D^m !}\right)} \longrightarrow 0$ as $m \to \infty,$ \eqref{eq:constant-intensity-upper-1st-mom} implies \eqref{eq:eps-upper}.

Now we prove \eqref{eq:eps-lower}. A calculation similar to \eqref{eq:constant-intensity-asymptotic} shows that
\begin{equation}\label{eq:eps-lower-expectation}
\ev \left(\underline{\varepsilon}_m^+ \right) \ge \mathrm{e}^{-\lambda'_m(1)} \cdot \ev \left( \frac{(\lambda'_m(1))^{D^m}}{D^m !} \right).
\end{equation}

Note that $\lambda'_m(1)$ converges to $1$ as $m\to \infty$. Moreover,
\begin{equation}\label{eq:lambda'-lower-bound}
1-\lambda'_m(1) \stackrel{\eqref{eq:lambda-diff-prob}}{=} \prob (D^m \le X^m_1)=\sum \limits_{d=2}^{\Delta_m} \left( p^m_d \cdot \sum \limits_{k=d}^\infty \mathrm{e}^{-1} \cdot \frac{1}{k!} \right) \le \sum \limits_{d=2}^{\Delta_m} p^m_d \cdot \mathrm{e}^{-1} \cdot \frac{2}{d!} = \frac{2}{\mathrm{e}} \cdot \ev \left( \frac{1}{D^m !} \right).
\end{equation}

Therefore,
\begin{align*}
\ev \left(\underline{\varepsilon}_m^+ \right) &\stackrel{\eqref{eq:eps-lower-expectation},\, \eqref{eq:lambda'-lower-bound}}{\ge} \mathrm{e}^{-1} \cdot \ev \left( \frac{ \left(1-\frac{2}{\mathrm{e}} \cdot \ev \left( \frac{1}{D^m !} \right) \right)^{D^m}}{D^m !} \right) \stackrel{(*)}{\ge} \mathrm{e}^{-1} \cdot \ev \left( \frac{ 1-\frac{2D^m}{\mathrm{e}} \cdot \ev \left( \frac{1}{D^m !} \right)}{D^m !} \right)=\\
& \mathrm{e}^{-1} \cdot \ev\left( \frac{1}{D^m!} \right) \cdot \left(1-\frac{2}{\mathrm{e}} \cdot \ev \left( \frac{1}{(D^m-1)!} \right) \right) \sim \mathrm{e}^{-1} \cdot \ev\left( \frac{1}{D^m!} \right), \quad \text{since } D^m \Rightarrow \infty,
\end{align*}
where at $(*)$ we used Bernoulli's inequality. Together with \eqref{eq:eps-connection}, \eqref{eq:eps-upper}, we can conclude~\eqref{eq:eps-lower}.

Finally, we prove \eqref{eq:eps-second-mom}:
\begin{align*}
\ev \left( \overline{\varepsilon}^2_m \cdot \indic{\overline{\varepsilon}_m \ge 0} \right) &\stackrel{\eqref{eq:eps-2nd-moment}}{=} \sum \limits_{d=2}^{\Delta_m} p_d^m \sum \limits_{k=d-1}^\infty \mathrm{e}^{-1} \cdot \frac{1}{k!} \cdot \frac{(k-d+2)(k-d+3)}{(k+1)(k+2)} \le\\
&\sum \limits_{d=2}^{\Delta_m} p_d^m \mathrm{e}^{-1} \left( \frac{2}{(d+1)!} + \frac{6}{(d+2)!} + \sum \limits_{k=d+1}^\infty \frac{1}{k!} \right) \le\\
&\sum \limits_{d=2}^{\Delta_m} p_d^m \mathrm{e}^{-1} \cdot \frac{10}{(d+1)!} = \mathcal{O} \left( \ev \left( \frac{1}{(D^m+1)!} \right) \right).
\end{align*}
\end{proof}

\begin{lemma}[Upper bound on $J_m$]\label{lem:Jm-bound}
Recall $J_m$ from \eqref{eq:delta-I-J-def}. We have
\begin{equation}\label{eq:Jm-bound}
J_m = \mathcal{O} \left( \ev \left( \frac{1}{(D^m+1)!} \right) \right).
\end{equation}
\end{lemma}

\begin{proof}
Let us define $\kappa_m := \min \{\, t \,:\, X^m_t=D^m \,\}$, $\beta_m := 1-\kappa_m$. Observe that by the definition of $\kappa_m$ and \eqref{eq:tau_m-def}, we have
\begin{equation}\label{eq:kappa-larger}
\tau_m \le \kappa_m, \quad \text{hence} \quad \beta_m \le \varepsilon_m.
\end{equation}

First, note that $\beta_m \ge 0$ if and only if $\kappa_m - \tau_m \le \varepsilon_m$. On the other hand, by the definition of $\tau_m$ (see \eqref{eq:tau_m-def}) and $\kappa_m$, the difference $\kappa_m-\tau_m$ is the time that we need to wait from $\tau_m$ until the next point of $\mathcal{X}^m$. Hence, we can define a random variable $\phi \sim \text{EXP}(1)$ that is independent to $\tau_m$ and $\kappa_m-\tau_m > \phi$ (see \eqref{eq:PPPs-connection}). Now we can deduce~that
\begin{equation}\label{eq:beta-positive-upper}
\prob(\beta_m \ge 0 \, | \, \tau_m ) = \prob( \kappa_m -\tau_m \le \varepsilon_m \, | \, \tau_m ) \le \prob( \phi \le \varepsilon_m \, | \, \tau_m ) \le \varepsilon_m.
\end{equation}

Therefore,
\begin{align}
\begin{split}\label{eq:kappa-asymptotic}
\ev((1-\kappa_m)\cdot \indic{\kappa_m \le 1})
&= \ev(\beta_m \cdot \indic{\kappa_m \le 1})
= \ev(\beta_m^+ \cdot \indic{\kappa_m \le 1}) \stackrel{\eqref{eq:kappa-larger}}{\le}\\
& \ev(\beta_m^+ \cdot \indic{\tau_m \le 1})
\le \ev(\varepsilon_m \cdot \indic{\beta_m \ge 0} \cdot \indic{\tau_m \le 1}) \stackrel{\text{tower rule}}{\le}\\
& \ev \left(\varepsilon_m \cdot \indic{\tau_m \le 1} \cdot \prob(\beta_m \ge 0 \, | \, \tau_m ) \right) \stackrel{\eqref{eq:beta-positive-upper}}{\le} \ev(\varepsilon_m^2 \cdot \indic{\varepsilon_m \ge 0}).
\end{split}
\end{align}

Hence
\begin{align*}
J_m &\stackrel{\eqref{eq:delta-I-J-def}}{=}\int \limits_0^1 \prob(D^m \le X^m_s) \, \mathrm{d}s
 = \int \limits_0^1 \prob(\kappa_m \le s) \, \mathrm{d}s
 = \int \limits_0^1 \ev(\indic{\kappa_m \le s}) \, \mathrm{d}s = \\
& \ev \left( \, \int \limits_{\kappa_m}^1 \, \mathrm{d}s \right) = \ev \left((1-\kappa_m) \cdot \indic{\kappa_m \le 1} \right) \stackrel{\eqref{eq:kappa-asymptotic}}{\le} \ev \left(\varepsilon_m^2 \cdot \indic{\varepsilon_m \ge 0} \right) \stackrel{\eqref{eq:eps-connection}}{\le}\\
&\ev \left(\overline{\varepsilon}_m^2 \cdot \indic{\overline{\varepsilon}_m \ge 0} \right) \stackrel{\eqref{eq:eps-second-mom}}{=} \mathcal{O} \left( \ev \left( \frac{1}{(D^m+1)!} \right) \right).
\end{align*}
\end{proof}

Now we can conclude \refL{lem:error-asymptotics}.

\begin{proof}[Proof of \refL{lem:error-asymptotics}]
First we prove that $\delta_m \sim \frac{2}{\mathrm{e}} \ev \left(\frac{1}{D^m!}\right)$. Note that
\begin{equation}\label{eq:delta-with-tau}
\delta_m \stackrel{\eqref{eq:delta-I-J-def}}{=} \int \limits_0^1 H_m(s) \cdot (1-s^2) \, \mathrm{d}s \stackrel{\eqref{eq:H-interpret}}{=} \ev \left( (1-\tau_m^2) \cdot \indic{\tau_m \le 1} \right).
\end{equation}

Using that $1-s^2 = 2(1-s) - (1-s)^2$, we obtain
\begin{align}
\begin{split}\label{eq:delta-with-eps}
\delta_m &\stackrel{\eqref{eq:delta-with-tau}}{=} 2\ev \left( (1-\tau_m) \cdot \indic{\tau_m \le 1 }\right) - \ev \left( (1-\tau_m)^2 \cdot \indic{\tau_m \le 1} \right) \stackrel{\eqref{eq:eps-definitions}}{=}\\
& 2\ev \left( \varepsilon_m^+ \right) - \ev \left( \varepsilon_m^2 \cdot \indic{\varepsilon_m \ge 0} \right).
\end{split}
\end{align}

By \eqref{eq:eps-connection}, it follows from \eqref{eq:delta-with-eps} that
\begin{equation}\label{eq:delta-bounds}
 2\ev \left( \underline{\varepsilon}_m^+ \right) - \ev \left( \overline{\varepsilon}^2_m \cdot \indic{\overline{\varepsilon}_m \ge 0} \right) \le \delta_m \le 2\ev \left( \overline{\varepsilon}_m^+ \right).
\end{equation}

Equation \eqref{eq:delta-bounds} together with \refL{lem:eps-moments} proves the first identity of \eqref{eq:error-asymptotics}. Moreover, by \eqref{eq:Jm-bound}, it also implies the third limit stated in \eqref{eq:error-asymptotics}. 
Finally, we show the second statement of \eqref{eq:error-asymptotics}. Note that $p_1 = 0$ by the strict Assumption~\ref{assump:degree_constraint_strict}. Therefore,
\begin{align*}
I_m &\stackrel{\eqref{eq:delta-I-J-def}, \, \eqref{eq:H-primitive}}{=} \prob(X^m_{1+2\delta_m} \ge D^m-1) \stackrel{\eqref{eq:PPPs-connection}}{\le} \prob(\overline{X}_{1+2\delta_m} \ge D^m-1) \le\\
& \prob(\overline{X}_1 \ge D^m-1) + \prob( \overline{X}_1 \neq \overline{X}_{1+2\delta_m}) \le \mathrm{e}^{-1} \sum \limits_{d=2}^{\Delta_m} p_d^m \sum \limits_{k=d-1}^\infty \frac{1}{k!} + 2\delta_m \le\\
& 2\mathrm{e}^{-1} \ev \left( \frac{1}{(D^m-1)!} \right) + 2\delta_m \stackrel{(*)}{\sim} \mathcal{O} \left( \ev \left( \frac{1}{(D^m-1)!} \right) \right),
\end{align*}
where at $(*)$ we used that $\delta_m \sim \frac{2}{\mathrm{e}} \ev \left(1/{(D^m!)}\right)$. Therefore, $I_m^2 = \mathcal{O} \left( \ev \left[ {((D^m-1)!)^{-2}} \right] \right)$ and $I_m/\delta_m^2 \to 0$ as $m \to \infty$.
\end{proof}

\subsection{Asymptotic formula for the critical time}\label{subsec:t_crit_asymp_formula}

\paragraph{}In this section our goal is to prove \refL{lem:t-crit-asymp}, i.e., to prove the asymptotic formula~\eqref{eq:tc-formula} for $\hat{t}_c(\underline{p}^m)$. The key ingredient of the proof of \refL{lem:t-crit-asymp} is \refL{lem:w-properties} in which we derived a second order linear initial value problem for the principal eigenfunction of the branching operator. We approximate the solution of this initial value problem in \refL{lem:Wm-gamma-prop} and the critical time in \refL{lem:tc-approx}.

\begin{definition}[Extended second order linear differential equation]\label{def:Wm-gamma-theta-def}
Let us consider a sequence of distributions $\underline{p}^m$ as in \refT{thm:tc-asymptotics} and recall from \refN{not:crit-asymptotic-notations} the function $H_m(t)$. Let $W_m(t)$ be the solution of the initial value problem
\begin{align}
W_m''(t) &= -H_m(t)W_m(t) \quad \text{for } 0 < t < \infty \label{eq:Wm-diffeq}\\
W_m(0) &= 0 \label{eq:Wm-null}\\
W_m'(0) &= 1. \label{eq:Wm-diff-null}
\end{align}

We also introduce the notation (cf.\ \eqref{eq:2nd-order-extra-mu})
\begin{equation}\label{eq:gamma-def}
\gamma_m(t) := W_m'(t) -W_m(t) \cdot \left( 1-\int \limits_0^t H_m(s) \, \mathrm{d}s \right)
\end{equation}
and
\begin{equation}\label{eq:theta-def}
\theta_m := \inf \{ \, t > 0 \, | \, W'_m(t)=0 \, \}.
\end{equation}
\end{definition}

\begin{lemma}[Properties of $W_m$ and $\gamma_m$]\label{lem:Wm-gamma-prop}
Let us consider the functions $W_m(t)$ and $\gamma_m(t)$ introduced in \refD{def:Wm-gamma-theta-def}. For any $m \in \mathbb{N}$ the following holds.%
{\vspace{-0.25em}\begin{enumerate}[label=(\alph*)]
\itemsep 0.125em \partopsep=0pt \parsep 0em 
\item The function $W_m(t)$ is non-negative, strictly increasing and concave on the interval~$[0, \theta_m]$. Moreover, if we consider the function $w_{\hat{t}_c(\underline{p}^m)}(t)$ (see \refD{def:w-def}), then
\begin{align}
W_m(t) &= w_{\hat{t}_c(\underline{p}^m)}(t) \quad \forall \, t \in [0,\hat{t}_c(\underline{p}^m)], \label{eq:Wm-equals-wtc}\\
W_m(t) &> w_{\hat{t}_c(\underline{p}^m)}(t) \quad \forall \,t \in (\hat{t}_c(\underline{p}^m), \theta_m]. \label{eq:Wm-larger-wtc}
\end{align}

\item For any $t \in \mathbb{R}_+$ we have
\begin{align}
W_m'(t) &= 1-\int \limits_0^t H_m(s)W_m(s) \, \mathrm{d}s, \label{eq:Wm-diff-formula}\\
W_m(t) &= t-\int \limits_0^t (t-s) \cdot H_m(s)W_m(s) \, \mathrm{d}s. \label{eq:Wm-formula}
\end{align}
\item The function $\gamma_m(t)$ is strictly decreasing on the interval $[0, \theta_m]$.
\item\label{item:unique_root} We have $\hat{t}_c(\underline{p}^m) < \theta_m$ and $\hat{t}_c(\underline{p}^m)$ is the unique root of $\gamma_m(t)$ on $[0,\theta_m]$.
\item For any $t \in [0, \theta_m]$:
\begin{equation}\label{eq:gamma-diffdiff}
0\le \gamma''_m(t) \le t+1.
\end{equation}%
\vspace{-0.25em}\end{enumerate}}%
\end{lemma}

\begin{proof}$ $
\begin{enumerate}[label=(\alph*)]
\item Observe that $w_{\hat{t}_c(\underline{p}^m)}(t)$ solves the initial value problem \{\eqref{eq:Wm-diffeq}, \eqref{eq:Wm-null}, \eqref{eq:Wm-diff-null}\} on the interval $(0,\hat{t}_c(\underline{p}^m))$ (cf.\ the initial value problem \{\eqref{eq:w_that-2nd-order-diffeq}, \eqref{eq:w_that-iv-null1}, \eqref{eq:w_that-iv-null2}\} of \refL{lem:w-properties} for $\hat{t}=\hat{t}_c(\underline{p}^m)$, where $\mu_{\hat{t}_c(\underline{p}^m)}=1$), hence we obtain \eqref{eq:Wm-equals-wtc}. However, $w_{\hat{t}_c(\underline{p}^m)}(t)$ is constant on the interval $[\hat{t}_c(\underline{p}^m), \theta_m]$, while $W_m(t)$ is strictly increasing on~$[0, \theta_m]$ (see~\eqref{eq:theta-def}). It implies \eqref{eq:Wm-larger-wtc} and also the non-negativity of $W_m(t)$ on the interval $[0, \theta_m]$. Concavity of $W_m(t)$ follows from its non-negativity and \eqref{eq:Wm-diffeq}.

\item Follows from \eqref{eq:Wm-diffeq} and the initial values using the fundamental theorem of calculus.

\item It follows from the negativity of $\gamma_m'(t)$:
\begin{align}
\begin{split}\label{eq:gamma-diff}
\gamma_m'(t) &\stackrel{\eqref{eq:gamma-def}}{=} W_m''(t) -W_m'(t) +W_m'(t) \cdot \int \limits_0^t H_m(s) \, \mathrm{d}s +W_m(t)H_m(t) \stackrel{\eqref{eq:Wm-diffeq}}{=}\\
& -W_m'(t) \cdot \left(1- \int \limits_0^t H_m(s) \, \mathrm{d}s \right).
\end{split}
\end{align}
The right hand-side of \eqref{eq:gamma-diff} is negative by \eqref{eq:H-int} and the fact that $W_m'(t)>0$ on $[0,\theta_m)$ by the definition of $\theta_m$ \eqref{eq:theta-def}. Therefore, $\gamma_m'(t) < 0$ on $[0,\theta_m)$.

\item The inequality $\hat{t}_c(\underline{p}^m) < \theta_m$ follows from the fact that $W_m(t)$ is monotone increasing on $[0, \hat{t}_c(\underline{p}^m)]$ (see \eqref{eq:Wm-equals-wtc} and \refL{lem:w-properties}) and the definition of $\theta_m$ (see \eqref{eq:theta-def}).

By \eqref{eq:2nd-order-extra-mu} and the fact that $W_m(\hat{t}_c(\underline{p}^m)) = w_{\hat{t}_c(\underline{p}^m)}(\hat{t}_c(\underline{p}^m))$ (see \eqref{eq:Wm-equals-wtc}), we know that $\gamma_m(\hat{t}_c(\underline{p}^m)) =0$. Since $\gamma_m$ is monotone, it has no other root on the interval $[0,\theta_m]$.

\item Differentiating \eqref{eq:gamma-diff} and using \eqref{eq:Wm-diffeq} we obtain
\begin{align*}
\gamma_m''(t) &\stackrel{\eqref{eq:gamma-diff},\, \eqref{eq:Wm-diffeq}}{=} H(t) \cdot \left(W_m(t) \cdot \int \limits_t^\infty H(s) \, \mathrm{d}s + W_m'(t)\right) \stackrel{\eqref{eq:H-int},\, \eqref{eq:H-bound}}{\le}\\
& W_m(t)+W_m'(t) \stackrel{\eqref{eq:Wm-formula},\,\eqref{eq:Wm-diff-formula}}{\le}t+1
\end{align*}
and $\gamma_m''(t) \ge 0$ since all of the terms in the above formula are non-negative if~${t \in [0,\theta_m]}$.
\end{enumerate}
\end{proof}

Recall $\delta_m$ and $I_m$ from \eqref{eq:delta-I-J-def} and note that \refL{lem:error-asymptotics} implies that $\delta_m \to 0$ and $I_m \to 0$ as $m \to \infty$.

\begin{lemma}[Bounds on $\hat{t}_c$]\label{lem:tc-approx} We have $1+2\delta_m < \theta_m$.
Furthermore, let us consider the function $\gamma_m(t)$ defined in \eqref{eq:gamma-def} and let~${\tilde{t}_m:= 1+\delta_m}$.
There exists $c>0$ such that if $m$ is large enough, then we have
\begin{equation}\label{eq:gamma-bounds}
\gamma_m(\tilde{t}_m-c \cdot I_m^2) > 0, \quad \gamma_m(\tilde{t}_m+c \cdot I_m^2) < 0.
\end{equation}
\end{lemma}

\begin{proof}
First note that if $t \le 1+2\delta_m$ then $W'_m(t) = 1 + \mathcal{O}(I_m)$ (by equation \eqref{eq:Wm-diff-formula}), thus $1+2\delta_m < \theta_m$ if $m$ is large enough (see the definition of $\theta_m$ \eqref{eq:theta-def}).

Now we prove \eqref{eq:gamma-bounds}. If $t \le 1+2\delta_m$ then using that $t < 2$ if $m$ is large enough, by~\eqref{eq:H-bound} and \eqref{eq:Wm-formula}, we have
\begin{equation}\label{eq:Wm-asymptotic1}
W_m(t) = t + \mathcal{O}(I_m),
\end{equation}
and hence from \eqref{eq:Wm-formula} and \eqref{eq:Wm-diff-formula} we obtain
\begin{align}
\begin{split}\label{eq:Wm-asymptotic2}
W_m(t) &\stackrel{\eqref{eq:Wm-formula},\, \eqref{eq:Wm-asymptotic1}}{=} t-\int \limits_0^t H_m(s)\cdot (t-s)\cdot s \, \mathrm{d}s + \mathcal{O}(I_m^2),\\ W'_m(t) &\stackrel{\eqref{eq:Wm-diff-formula},\, \eqref{eq:Wm-asymptotic1}}{=} 1-\int \limits_0^t H_m(s)\cdot s \, \mathrm{d}s + \mathcal{O}(I_m^2).
\end{split}
\end{align}

Therefore, for any $t \le 1+2\delta_m$
\begin{multline}
\label{eq:gamma-asymptotic}
\gamma_m(t) \stackrel{\eqref{eq:gamma-def}}{=} W_m'(t) +W_m(t) \cdot \left(-1+ \cdot \int \limits_0^t H_m(s) \, \mathrm{d}s \right) \stackrel{\eqref{eq:Wm-asymptotic2}}{=}\\
1-\int \limits_0^t H_m(s)\cdot s \, \mathrm{d}s + \left( t-\int \limits_0^t H_m(s)\cdot (t-s)\cdot s \, \mathrm{d}s \right) \cdot \left(-1+\int \limits_0^t H_m(s) \, \mathrm{d}s \right) + \mathcal{O}(I_m^2)\stackrel{(*)}{=}\\
1-t + \int \limits_0^t H_m(s) \cdot (t+ts-s-s^2) \, \mathrm{d}s+ \mathcal{O}(I_m^2),
\end{multline}
where at $(*)$ we used that for any $t \le 1+2\delta_m$, we have
\begin{equation*}
\left( \int \limits_0^t H_m(s) \cdot (t-s) \cdot s \, \mathrm{d}s \right) \cdot \left( \int \limits_0^t H_m(s) \, \mathrm{d}s\right) \le t^2 \cdot \left( \int \limits_0^t H_m(s) \, \mathrm{d}s \right)^2 = \mathcal{O}(I_m^2).
\end{equation*}

In particular, for $t=\tilde{t}_m=1+\delta_m$, we obtain that
\begin{align}
\begin{split}\label{eq:gamma(ttilde)}
\gamma_m(\tilde{t}_m) &\stackrel{\eqref{eq:gamma-asymptotic}}{=}-\delta_m + \int \limits_0^{1+\delta_m} H_m(s) \cdot (1-s^2+\delta_m\cdot(1+s)) \, \mathrm{d}s+ \mathcal{O}(I_m^2) \stackrel{\eqref{eq:delta-smaller-than-I}}{=}\\
&-\delta_m+\int \limits_0^1 H_m(s)\cdot (1-s^2) \, \mathrm{d}s + \mathcal{O}(I_m^2) \stackrel{\eqref{eq:delta-I-J-def}}{=} \mathcal{O}(I_m^2).
\end{split}
\end{align}

By \eqref{eq:gamma(ttilde)}, there exists a constant $C>0$ such that if $m$ is large enough, then
\begin{equation}\label{eq:gamma-t-tilde-upper}
|\gamma_m(\tilde{t}_m)| \le C \cdot I_m^2.
\end{equation}

Now we show that if $c = 2C$ then \eqref{eq:gamma-bounds} holds. We use first order Taylor approximation of $\gamma_m$ at $\tilde{t}_m$, noting that $|\gamma_m''(t)| \le 3$ on the interval $[0,1+2\delta_m]$ if $m$ is large enough by~\eqref{eq:gamma-diffdiff}. We begin by showing the first inequality of \eqref{eq:gamma-bounds}:
\begin{align*}
\gamma_m(\tilde{t}_m-2C \cdot I_m^2) &\ge \gamma_m(\tilde{t}_m) - 2C \cdot I_m^2 \cdot \gamma_m'(\tilde{t}_m) -\frac 32 \cdot \left(2C \cdot I_m^2 \right)^2 \stackrel{(*)}{\ge}\\
& -C \cdot I_m^2 + 2C \cdot I_m^2 + \mathcal{O}\left( I_m^3 \right) > 0 \quad \text{if } m \text{ is large enough,}
\end{align*}
where at $(*)$ we used \eqref{eq:gamma-t-tilde-upper} and that
\begin{equation}\label{eq:w-diff-tilde-t}
\gamma_m'(\tilde{t}_m)\stackrel{\eqref{eq:gamma-diff}}{=} -W_m'(\tilde{t}_m) \cdot \left( 1- \int \limits_0^{\tilde{t}_m} H_m(s) \, \mathrm{d}s \right) \stackrel{\eqref{eq:Wm-asymptotic2}}{=} -1+ \mathcal{O}\left( I_m \right) .
\end{equation}

We can prove the second inequality of \eqref{eq:gamma-bounds} similarly:
\begin{align*}
\gamma_m(\tilde{t}_m + 2C \cdot I_m^2) &\le \gamma(\tilde{t}_m) + 2C \cdot I_m^2 \cdot \gamma_m'(\tilde{t}_m) +\frac 32 \cdot \left( 2C \cdot I_m^2 \right)^2 \stackrel{\eqref{eq:gamma-t-tilde-upper},\, \eqref{eq:w-diff-tilde-t}}{\le}\\
& -C \cdot I_m^2 + \mathcal{O}\left( I_m^3 \right) < 0
\end{align*}
if $m$ is large enough.
\end{proof}

\begin{proof}[Proof of \refL{lem:t-crit-asymp}]
By \refL{lem:Wm-gamma-prop}~\ref{item:unique_root} we know that the unique root of $\gamma_m(t)$ on the interval $[0,\theta_m]$ is $\hat{t}_c(\underline{p}^m)$. On the other hand, by \refL{lem:tc-approx}, we also know that if $m$ is large enough, then there is a root of $\gamma_m(t)$ in the interval $(\tilde{t}_m -c\cdot I_m^2,\tilde{t}_m +c\cdot I_m^2)$, where $\tilde{t}_m := 1+\delta_m$.

By \refL{lem:tc-approx} we also know that $1+2\delta_m < \theta_m$ if $m$ is large enough. Therefore, the root of $\gamma_m(t)$ that is close to $\tilde{t}_m$ must be $\hat{t}_c(\underline{p}^m)$, i.e., $|\hat{t}_c(\underline{p}^m)-\tilde{t}_m| = \mathcal{O}(I_m^2)$.
\end{proof}

\subsection{Asymptotic formula for expected number of neighbors}\label{subsec:F_tc_asymp}
In this final section we prove \refL{lem:F(tc)-asymp}.

\begin{proof}[Proof of \refL{lem:F(tc)-asymp}]
First, note that if $s \le 1+2\delta_m$ then
\begin{equation}\label{eq:F(tc)-asymp-aux}
\prob(D^m \le X^m_s) \le \prob(D^m-1 \le X^m_s) \stackrel{\eqref{eq:H-interpret}}{=} \int \limits_0^s H_m(u) \, \mathrm{d}u \stackrel{\eqref{eq:delta-I-J-def}}{\le} I_m.
\end{equation}

Hence
\begin{equation}\label{eq:square-term-is-small}
\int \limits_0^{\hat{t}_c(\underline{p}^m)} \prob(D^m \le X^m_s)^2 \, \mathrm{d}s \stackrel{\eqref{eq:tc-formula},\, \eqref{eq:F(tc)-asymp-aux}}{=} \mathcal{O}(I_m^2).
\end{equation}

Therefore, we obtain \eqref{eq:F(tc)-formula} by noting~that
\begin{align*}
F_{\underline{p}^m}(\hat{t}_c(\underline{p}^m)) &\stackrel{\eqref{eq:F(t)-formula}}{=} \int \limits_0^{\hat{t}_c(\underline{p}^m)} (\lambda_m'(s))^2 \, \mathrm{d}s \stackrel{\eqref{eq:lambda-diff-prob}}{=} \int \limits_0^{\hat{t}_c(\underline{p}^m)} (1-\prob(D^m \le X^m_s))^2 \, \mathrm{d}s \stackrel{\eqref{eq:square-term-is-small}}{=}\\
& \hat{t}_c(\underline{p}^m) - 2 \int \limits_0^{\hat{t}_c(\underline{p}^m)} \prob(D^m \le X^m_s) \, \mathrm{d}s + \mathcal{O}(I_m^2) \stackrel{\eqref{eq:delta-I-J-def},\, \eqref{eq:tc-formula},\, \eqref{eq:F(tc)-asymp-aux}}{=}\\
& \hat{t}_c(\underline{p}^m) - 2 \int \limits_0^1 \prob(D^m \le X^m_s) \, \mathrm{d}s + \mathcal{O}(I_m^2) \stackrel{\eqref{eq:delta-I-J-def}}{=} \hat{t}_c(\underline{p}^m) - 2 J_m + \mathcal{O}(I_m^2),
\end{align*}
completing the proof of \refL{lem:F(tc)-asymp}.
\end{proof}

\appendix
\section{Appendix}\label{appendix_main}

In Section~\ref{subsec:app_HS} we prove Lemma~\ref{lem:properties-of-branching-operator}~\ref{item:bo_hs}. In Section~\ref{subsec:app_spectral_characterization} we prove \refT{thm:PF-eigenvalue}.

\subsection{Hilbert--Schmidt property of the branching operator}\label{subsec:app_HS}

\noindent After proving Lemma~\ref{lem:properties-of-branching-operator}~\ref{item:bo_hs} we will explain in Remark~\ref{rem:hs_barely} why the proof is a close call.

\begin{proof}[Proof of Lemma~\ref{lem:properties-of-branching-operator}~\ref{item:bo_hs}]
The goal is to prove that $B_{\hat{t}}$ is a Hilbert--Schmidt operator for any $\hat{t} \in \mathbb{R}_+ \cup \{+\infty\}$. In order to prove it, it is enough to see that the Hilbert--Schmidt norm of $B_{\infty}$ is finite. Let $\tau_1$ and $\tau_2$ be i.i.d.\ random variables with the same distribution as $\tau$ (defined in \eqref{eq:tau-def}). We have
\begin{multline}\label{eq:HS_bounds_ineqs}
\norm{B_{\infty}}_{\text{HS}}^2 \stackrel{\eqref{eq:kernel}}{=} \int \limits_0^\infty \int \limits_0^\infty \mathcal{K}_{\infty}(u,s)^2 \, \rho(u) \rho(s)\,\mathrm{d}u\, \mathrm{d}s \stackrel{\eqref{eq:H-rho-def},\, \eqref{eq:H-interpret}}{=} \ev \left( (\tau_1 \wedge \tau_2)^2 \right) =\\
\int \limits_0^\infty \prob \left(\tau > \sqrt{t} \right)^2 \, \mathrm{d}t \stackrel{\eqref{eq:tau-def}}{=} \int \limits_0^\infty \prob(X_{\sqrt{t}} \le D-2)^2 \, \mathrm{d}t = \int \limits_0^\infty \prob(X_s \le D-2)^2 \cdot 2s \, \mathrm{d}s.
\end{multline}

Observe that the above integral is finite on bounded intervals, therefore, it is enough to prove that
\begin{equation}\label{eq:HS-sufficient-goal}
\int \limits_{\lambda^{-1}(\Delta)}^\infty \prob(X_s \le D-2)^2 \cdot 2s \, \mathrm{d}s < \infty,
\end{equation}
where $\Delta$ is defined in the strict Assumption~\ref{assump:degree_constraint_strict}.
Let us define the functions $\Psi(y)$, $\psi(y)$ such that
\begin{equation}\label{eq:Psi-psi-def}
\Psi(\lambda) := \mathrm{e}^{-\lambda} \cdot \sum \limits_{k=0}^{\Delta-1} \frac{\lambda^k}{k!} \cdot q_{k+1}, \quad \text{and} \quad \psi(t) := \int \limits_{\Delta}^t \frac{1}{\Psi(u)} \,\mathrm{d}u,
\end{equation}
thus we have $\lambda'(t) = \Psi(\lambda(t))$ (cf.\ \eqref{eq:lambda-ivp}).
Without loss of generality, we can assume that $p_\Delta > 0$. Then, by the definition of $\Psi(\lambda)$, for any $\lambda > \Delta$, we have
\begin{equation}\label{eq:Psi-bounds}
\mathrm{e}^{-\lambda} \cdot \frac{\lambda^{\Delta-1}}{(\Delta-1)!} \cdot p_\Delta \le \Psi(\lambda) \le \mathrm{e}^{-\lambda} \cdot \frac{\lambda^{\Delta-1}}{(\Delta-1)!} \cdot \Delta .
\end{equation}

Observe that
\begin{equation}\label{eq:Psi-psi-formula}
\psi(\lambda(t)) = t-\lambda^{-1}(\Delta), \quad \text{i.e.,} \quad \lambda^{-1}(t) = \lambda^{-1}(\Delta)+\psi(t) \quad \text{and} \quad (\lambda^{-1})'(t) = \frac{1}{\Psi(t)}.
\end{equation}

Now we obtain
\begin{multline}
\label{eq:HS-integral-finite}
\int \limits_{\lambda^{-1}(\Delta)}^\infty \prob(X_s \le D-2)^2 \cdot 2s \, \mathrm{d}s \stackrel{\eqref{eq:Xt-def},\, \eqref{eq:def-pk-qk}}{=} \int \limits_{\lambda^{-1}(\Delta)}^\infty \left( \mathrm{e}^{-\lambda(s)} \cdot \sum \limits_{k=0}^{\Delta-2} \frac{\lambda(s)^k}{k!} \cdot q_{k+2} \right)^2 \cdot 2s \, \mathrm{d}s \stackrel{\eqref{eq:lambda-ivp}}{\le}\\
\int \limits_{\lambda^{-1}(\Delta)}^\infty \left( \Delta \cdot \frac{\lambda'(s)}{\lambda(s)} \right)^2 \cdot 2s \, \mathrm{d}s \stackrel{y=\lambda(s), \, \eqref{eq:lambda-lim}}{=} 2\Delta \cdot \int \limits_{\Delta}^\infty \frac{\lambda^{-1}(y)}{(\lambda^{-1})'(y)} \cdot \frac{1}{y^2} \, \mathrm{d}y \stackrel{\eqref{eq:Psi-psi-formula}}{=}\\
 2\Delta \int \limits_{\Delta}^\infty \Psi(y) \cdot \lambda^{-1}(\Delta) \cdot \frac{1}{y^2} \, \mathrm{d}y +
 2\Delta \int \limits_{\Delta}^\infty \Psi(y) \cdot \psi(y) \cdot \frac{1}{y^2} \, \mathrm{d}y.
 \end{multline}
 The first term on the r.h.s.\ of \eqref{eq:HS-integral-finite} is finite since $ \Psi(y) \leq 1$ (cf.\ \eqref{eq:Psi-psi-def}). Next we show that the second term on the r.h.s.\ of \eqref{eq:HS-integral-finite} is finite:
\begin{multline}\label{eq:hs_calc_2}
 \int \limits_{\Delta}^\infty \Psi(y) \cdot \psi(y) \cdot \frac{1}{y^2} \, \mathrm{d}y 
 \stackrel{\eqref{eq:Psi-psi-def}}{=} \int \limits_{\Delta}^\infty \left( \int \limits_{\Delta}^y \frac{\Psi(y)}{\Psi(t)} \, \mathrm{d}t \right) \cdot \frac{1}{y^2} \, \mathrm{d}y \stackrel{\eqref{eq:Psi-bounds}}{\le}\\
 \int \limits_{\Delta}^\infty \left( \int \limits_{\Delta}^y \frac{\mathrm{e}^{-y} \cdot y^{\Delta-1} \cdot \Delta}{\mathrm{e}^{-t} \cdot t^{\Delta-1} \cdot p_\Delta} \, \mathrm{d}t \right) \cdot \frac{1}{y^2} \, \mathrm{d}y \stackrel{(*)}{<} \infty.
\end{multline}

For the last step, marked with $(*)$, observe that the inner integral can be upper bounded as follows:
\begin{align*}
\int \limits_{\Delta}^y \mathrm{e}^{-(y-t)} \cdot \left( \frac{y}{t} \right)^{\Delta-1} \, \mathrm{d}t & \le \int \limits_{\Delta}^{y/2} \mathrm{e}^{-y/2} \cdot y^{\Delta-1} \, \mathrm{d}t + \int \limits_{y/2}^y \mathrm{e}^{-(y-t)} \cdot 2^{\Delta-1} \, \mathrm{d}t \le \mathrm{e}^{-y/2} \cdot y^\Delta + 2^\Delta.
\end{align*}

Inequality \eqref{eq:HS-integral-finite} implies \eqref{eq:HS-sufficient-goal}, which implies that $B_{\hat{t}}$ is Hilbert--Schmidt for any $\hat{t} \in \mathbb{R}_+ \cup \{+\infty\}$.

Finally, $B_{\hat{t}}$ is bounded, since the operator norm of a Hilbert--Schmidt operator is bounded by its Hilbert--Schmidt norm (see \cite[Theorem 6.22]{RS}): $\norm{B_{\hat{t}}} \le \norm{B_{\hat{t}}}_{\text{HS}}$.
\end{proof}

\begin{remark}[The proof of \eqref{eq:HS-sufficient-goal} was a close call]\label{rem:hs_barely}
It is easy to check using \eqref{eq:Psi-psi-def} that~$\Psi(\lambda)$ decays exponentially as $\lambda \to \infty$ and thus $\psi(t)$ grows exponentially as $t \to \infty$. 
Putting this together with \eqref{eq:Psi-psi-formula},
we see that $\lambda(t)$ grows logarithmically as $t \to \infty$. This implies that $\prob(X_s \le D-2)$ is `roughly' $\mathrm{e}^{-\lambda(s)}$, i.e., it is `roughly' $1/s$ when $s \gg 1$, and this (admittedly naive) calculation gives that the integral in \eqref{eq:HS-sufficient-goal} is `roughly' equal to~${\int_{\lambda^{-1}(\Delta)}^\infty \frac{1}{s^2} \cdot 2s \, \mathrm{d}s=+\infty}$.
\end{remark}

\begin{remark}[The principal eigenvalue of $B_\infty$]\label{rem:pf_infty}
We proved Lemma~\ref{lem:properties-of-branching-operator}~\ref{item:bo_hs} in the ${\hat{t}=\infty}$ case as well, because we wanted Perron--Frobenius theory (i.e., the results of Lemma~\ref{lem:eigen-properties}) to be applicable in the $\hat{t}=\infty$ case, too. Although in the current paper we only use the~${\hat{t}<\infty}$ case, let us note that the $\hat{t}=\infty$ case might find some future applications: for example, the principal eigenvalue of $B_\infty$ captures the Malthusian growth rate of the multi-type branching process that arises as the local limit of the final graphs~$G^n_{\underline{p}}(\infty)$.
\end{remark}

\subsection{Spectral characterization of sub/supercriticality of \texorpdfstring{$\mathcal{M}_{\underline{p}}(\hat{t})$}{M}}\label{subsec:app_spectral_characterization}

The goal of Section~\ref{subsec:app_spectral_characterization} is to prove \refT{thm:PF-eigenvalue}.

Let us recall the definition of the MTBP from Definition~\ref{def:mtbp} as well as the definitions of~$\lambda(\cdot)$, $E(\cdot)$, $\rho(\cdot)$ and the normalized principal eigenfunction $v_{\hat{t}}$ of $B_{\hat{t}}$ from Notation~\ref{not:lambda_q_z}, Notation~\ref{not:E}, Definition~\ref{def:auxiliary_rho} and Definition~\ref{def:w-def}, respectively.
We also recall the definitions of $\langle \cdot, \cdot \rangle_\rho$ and $\norm{\cdot}_2$ from Notation~\ref{not:eltwo}.

\begin{lemma}[Some elements of $L^2(\mathbb{R}_+, \rho)$]\label{lem:L2-elements} Let $\underline{p}$ satisfy the strict Assumption~\ref{assump:degree_constraint_strict}.
The functions $\bbone$, $\frac{E}{\lambda}$, $\frac{E \cdot \tilde{f}}{\lambda \cdot f}$, $v_{\hat{t}}$, $v_{\hat{t}}^2$ for any $\hat{t} \in \mathbb{R}_+$ are in $L^2(\mathbb{R}_+, \rho)$, where $\tilde{f}$ denotes the probability density function of the type of the $d(\rt)^{\text{th}}$ child of the root if the type of the root is chosen with density function $f(t)$.
\end{lemma}

\begin{proof}
Note that $d(\ibf) \ge 2$ for any vertex $\ibf$, therefore, by \eqref{eq:E-def}, $E(t) \ge 1$. Hence
\begin{equation}\label{eq:1-norm}
\left( \norm{\bbone}_2 \right)^2 \stackrel{\eqref{eq:H-rho-def}}{=} \int \limits_0^\infty \frac{\lambda(t)f(t)}{E(t)} \, \mathrm{d}t \stackrel{\eqref{eq:f-with-lambda-rdcp}, \, \eqref{eq:E-def}}{\le} -\int \limits_0^\infty \lambda(t) \cdot \lambda''(t) \, \mathrm{d}t \leq \int \limits_0^\infty \lambda'(t)^2 \, \mathrm{d}t \stackrel{\eqref{eq:F(t)-formula}}{=} F_{\underline{p}}(\infty) \le \Delta.
\end{equation}

By \eqref{eq:H-rho-def} and \eqref{eq:H-int} (noting that $p_1=0$ by the strict Assumption~\ref{assump:degree_constraint_strict}), we obtain that $\left(\norm{\frac{E}{\lambda}}_2\right)^2 = \int_0^\infty H(s) \, \mathrm{d}s =1$.

To prove that $\frac{E \cdot \tilde{f}}{\lambda \cdot f} \in L^2(\mathbb{R}_+, \rho)$, we estimate $\tilde{f}$:
\begin{equation}\label{eq:ftilde-estimation}
\tilde{f}(t)\stackrel{\eqref{eq:dth-dens}}{=} \int \limits_0^t f(s) \cdot \frac{f(t)}{\int \limits_s^\infty f(u) \, \mathrm{d}u} \, \mathrm{d}s \stackrel{\eqref{eq:lambda-diff-prob}, \, \eqref{eq:f-z-probabilistic-meaning}}{=} f(t) \cdot \int \limits_0^t \prob(D=X_s + 1) \, \mathrm{d}s \stackrel{\eqref{eq:lambda-diff-prob}}{\le} f(t) \cdot \lambda(t).
\end{equation}

Therefore,
\begin{align*}
\left( \norm{\frac{E \cdot \tilde{f}}{\lambda \cdot f}}_2 \right)^2 \stackrel{\eqref{eq:ftilde-estimation}}{\le} \int \limits_0^\infty E(t)^2 \rho(t) \, \mathrm{d}t \stackrel{\eqref{eq:E-def}}{\le} \Delta^2 \cdot \norm{\bbone}_2 \stackrel{\eqref{eq:1-norm}}{<} \infty.
\end{align*}

The fact that $\norm{v_{\hat{t}}}_2 < \infty$ follows from the boundedness of $B_{\hat{t}}$ (see \refL{lem:properties-of-branching-operator}).

Finally, we prove that $v_{\hat{t}}^2 \in L^2(\mathbb{R}_+, \rho)$ for any $\hat{t} \in \mathbb{R}_+$. Observe that $\lambda(t)$ and $w_{\hat{t}}(t)$ are monotone increasing, concave functions with $\lambda(0) = w_{\hat{t}}(0) = 0$, ${\lambda'(0) = w_{\hat{t}}'(0) = 1}$ and~$w_{\hat{t}}(t)$ is constant for $t \ge \hat{t}$ (see \eqref{eq_lambda_f_lambda_diff} and \refL{lem:w-properties}). Therefore, if we define ${t^* := (\lambda')^{-1}(1/2) \in \mathbb{R}_+}$, we have
\begin{equation}
\begin{alignedat}{4}
\label{eq:lambda-w-bounds}
&\lambda(t) \ge \frac{t}{2},\quad && w_{\hat{t}}(t) \le t \quad && \forall \, t \le t^*,&&\\
&\lambda(t) \ge \lambda(t^*), \quad && w_{\hat{t}}(t) \le w_{\hat{t}}(\hat{t}) \quad && \forall \, t \ge t^*.&&
\end{alignedat}
\end{equation}

Therefore,
\begin{align*}
\left( \norm{v_{\hat{t}}^2}_2 \right)^2 &\stackrel{\eqref{eq:H-rho-def}, \, \eqref{eq:w-def}}{=} \int \limits_0^\infty \frac{E(t)^2}{\lambda(t)^2} \cdot w_{\hat{t}}(t)^4 \cdot H(t) \, \mathrm{d}t \stackrel{\eqref{eq:E-def},\, \eqref{eq:H-bound},\, \eqref{eq:lambda-w-bounds}}{\le}\\
&\int \limits_0^{t^*} \frac{4\Delta^2}{t^2} \cdot t^4 \cdot 1 \, \mathrm{d}t + \int \limits_{t^*}^\infty \frac{\Delta^2}{(\lambda(t^*))^2} \cdot w_{\hat{t}}(\hat{t})^4 \cdot H(t) \, \mathrm{d}t \stackrel{\eqref{eq:H-int}}{<} \infty.
\end{align*}
\end{proof}

In the next claim we assume that the offspring distribution of the root of the MTBP is the same as that of other vertices.

\begin{claim}[Powers of the branching operator]\label{claim_power_of_br}
It is known that if $B$ denotes the branching operator of a MTBP and $B^r$ is the $r^{\text{th}}$ power of $B$, then we have
\begin{equation}\label{eq:power-of-branching-op}
\left(B^r\varphi \right)(t_0) = \ev \left( \sum \limits_{i=1}^{M^r} \varphi(t_i) \right),
\end{equation}
where $M^r$ is the (random) number of descendants in the $r^{\text{th}}$ generation of a vertex with type $t_0$ in the MTBP and $t_1, \, t_2, \, \dots, \, t_{M^r}$ are the (random) types of these descendants in the $r^{\text{th}}$ generation.
\end{claim}

\begin{proof}[Proof of \refT{thm:PF-eigenvalue}]
We have seen the first point of the theorem in \refL{lem:properties-of-branching-operator}.

Recall that $\mu_{\hat{t}} = \norm{B_{\hat{t}}}$ is the principal eigenvalue of the branching operator $B_{\hat{t}}$ and $v_{\hat{t}}$ denotes a corresponding eigenfunction (cf.\ \refL{lem:eigen-properties} and \refD{def:w-def}).

In this proof, we use the notation $g_r \asymp h_r$ for sequences $g_r$, $h_r\colon \mathbb{N} \to \mathbb{R}$ if for any~$r \in \mathbb{N}$, we have $c g_r \le h_r \le C g_r$ for some positive constants $c$, $C$ (which can depend on $\hat{t}$).

First we prove that if $\norm{B_{\hat{t}}} < 1$ then $\mathcal{M}_{\underline{p}}(\hat{t})$ is subcritical, i.e., $\ev( | \mathcal{M}_{\underline{p}}(\hat{t}) | ) < \infty$. Let us denote by $\mathcal{M}_{\underline{p}}^r(\hat{t})$ the set of vertices of the $r^\text{th}$ generation of $\mathcal{M}_{\underline{p}}(\hat{t})$. We prove that $\ev( |\mathcal{M}_{\underline{p}}^r(\hat{t})|) \asymp \mu_{\hat{t}}^r$. Note that Claim~\ref{claim_power_of_br} cannot be directly applied to our setting because of the $d(\rt)^{\text{th}}$ child of $\rt$
(cf.\ Definition~\ref{def:mtbp}), therefore, we will deal with the contribution of the descendants of the $d(\rt)^{\text{th}}$ child of $\rt$ to $\ev( | \mathcal{M}_{\underline{p}}(\hat{t}) | )$ later.

Equation \eqref{eq:power-of-branching-op} implies that the expected number of descendants of the first $d(\varnothing)-1$ children of the root in the $r^\text{th}$ generation of $\mathcal{M}_{\underline{p}}(\hat{t})$ is equal to $\int_0^\infty (B_{\hat{t}}^r \bbone)(t) f(t) \mathrm{d}t= \left\langle \frac{E}{\lambda}, \, B_{\hat{t}}^r \bbone \right\rangle_{\rho}$, which can be bounded by $ \norm{\frac{E}{\lambda}}_2 \norm{B_{\hat{t}}^r} \norm{\bbone}_2 \leq \norm{\frac{E}{\lambda}}_2 \norm{B_{\hat{t}}}^r \norm{\bbone}_2 \asymp \mu_{\hat{t}}^r $ with the help of Lemma~\ref{lem:L2-elements}.

Similarly, if $\tilde{f}(t)$ denotes the probability density function of the type of the $d(\rt)^{\text{th}}$ child of $\rt$, then the expected number of descendants of this vertex in the $r^\text{th}$ generation is $\left\langle \frac{E \cdot \tilde{f}}{\lambda \cdot f}, \, B_{\hat{t}}^{r-1} \bbone \right\rangle_{\rho}$. Similarly to the above calculation (and using Lemma~\ref{lem:L2-elements} again) we obtain $\left\langle \frac{E \cdot \tilde{f}}{\lambda \cdot f}, \, B_{\hat{t}}^{r-1} \bbone \right\rangle_{\rho} \asymp \mu_{\hat{t}}^{r-1}$.
Therefore,
\begin{equation}\label{eq:Mr-approximation}
\ev( |\mathcal{M}_{\underline{p}}^r(\hat{t})|) = \left\langle \frac{E}{\lambda}, \, B_{\hat{t}}^r \bbone \right\rangle_{\rho} + \left\langle \frac{E \cdot \tilde{f}}{\lambda \cdot f}, \, B_{\hat{t}}^{r-1} \bbone \right\rangle_{\rho} \asymp \mu_{\hat{t}}^r.
\end{equation}

This implies the subcriticality of $\mathcal{M}_{\underline{p}}(\hat{t})$, since
\begin{equation}
 \ev( | \mathcal{M}_{\underline{p}}(\hat{t}) | ) = \sum \limits_{r=0}^\infty \ev( |\mathcal{M}_{\underline{p}}^r(\hat{t})|) \stackrel{\eqref{eq:Mr-approximation}}{\asymp} \frac{1}{1-\mu_{\hat{t}}^r} < \infty.
\end{equation}

Now we prove that if $\norm{B_{\hat{t}}} > 1$ then $\mathcal{M}_{\underline{p}}(\hat{t})$ is supercritical, i.e., $\prob( | \mathcal{M}_{\underline{p}}(\hat{t}) | = \infty ) > 0$. We show that the process is supercritical even if we ignore the $d(\rt)^\text{th}$ child of the root (and its descendants). Note that this way the offspring distribution is the same for each vertex. Let us denote this smaller graph and its $r^\text{th}$ generation by $\widetilde{\mathcal{M}}_{\underline{p}}(\hat{t})$ and $\widetilde{\mathcal{M}}^r_{\underline{p}}(\hat{t})$, respectively. We introduce the notation
\begin{equation}\label{eq:Yr-def}
Y^r(\hat{t}) := \sum \limits_{\ibf \in \widetilde{\mathcal{M}}_{\underline{p}}^r(\hat{t})} v_{\hat{t}}(T_{\ibf}),
\end{equation}
where $T_{\ibf}$ is the type of vertex $\ibf$. We will show that there exists a constant $\hat{c}>0$ such that
\begin{equation}\label{eq:Yr-moments}
\ev(Y^r(\hat{t})) \asymp \mu_{\hat{t}}^r, \quad \text{Var} \left(Y^r(\hat{t}) \right) \le \hat{c} \mu_{\hat{t}}^{2r}.
\end{equation}

As soon as we prove \eqref{eq:Yr-moments}, supercriticality follows by the Paley--Zygmund inequality:
\begin{equation}\label{eq:Paley-Zygmund}
\prob \left(Y^r(\hat{t}) \ge \frac 12 \ev(Y^r(\hat{t})) \right) \ge \frac 14 \frac{ \left( \ev(Y^r(\hat{t})) \right)^2}{\ev \left((Y^r(\hat{t}))^2 \right)} \stackrel{\eqref{eq:Yr-moments}}{\ge} C, \quad r=1, 2, 3, \dots
\end{equation}
for some positive constant $C$, and therefore,
\begin{equation*}
\prob \left( \, |\widetilde{\mathcal{M}}_{\underline{p}}(\hat{t})| = \infty \right) = \lim \limits_{r \to \infty}\prob \left( \widetilde{\mathcal{M}}_{\underline{p}}^r(\hat{t}) \neq \emptyset \right) \stackrel{\eqref{eq:Paley-Zygmund}}{\ge} C> 0,
\end{equation*}
i.e., $\widetilde{\mathcal{M}}_{\underline{p}}(\hat{t})$ is supercritical.

It remains to show \eqref{eq:Yr-moments}. Let $\mathcal{F}_r$ be the $\sigma$-algebra generated by the first $r$ generations of $\widetilde{\mathcal{M}}_{\underline{p}}(\hat{t})$, i.e., by the variables $\big\{ \, T_\ibf \, | \, \ibf \in \bigcup \limits_{k=0}^r \widetilde{\mathcal{M}}^k_{\underline{p}}(\hat{t}) \, \big\}$. Let us first observe that $Y^r(\hat{t})/\mu_{\hat{t}}^r$ is a martingale:
\begin{align}
\begin{split}\label{eq:Yr-expectation}
\ev(Y^{r+1}(\hat{t}) \, | \, \mathcal{F}_r ) &\stackrel{\eqref{eq:Yr-def}}{=} \sum \limits_{\ibf \in \widetilde{\mathcal{M}}^r_{\underline{p}}(\hat{t})} \ev \left( \sum \limits_j v_{\hat{t}}(T_{\ibf j}) \, \Big| \, \mathcal{F}_r \right) \stackrel{\eqref{eq:branch-op-def}}{=} \sum \limits_{\ibf \in \widetilde{\mathcal{M}}^r_{\underline{p}}(\hat{t})} B_{\hat{t}}v_{\hat{t}} (T_\ibf) =\\
& \sum \limits_{\ibf \in \widetilde{\mathcal{M}}^r_{\underline{p}}(\hat{t})} \mu_{\hat{t}}v_{\hat{t}} (T_\ibf) \stackrel{\eqref{eq:Yr-def}}{=} \mu_{\hat{t}} \cdot Y^r(\hat{t}).
\end{split}
\end{align}
It implies the first identity of \eqref{eq:Yr-moments}.

Now we estimate the variance of $Y^r(\hat{t})$. First, observe that for any $\ibf \in \widetilde{\mathcal{M}}^r_{\underline{p}}(\hat{t})$, given~$T_\ibf$ and $d(\ibf)$, the phantom saturation of times of its children are conditionally independent (see \refD{def:mtbp}). Therefore,
\begin{align}
\begin{split}\label{eq:Yr+1-condiitonal-expectation}
\ev \left( \sum \limits_j v_{\hat{t}}(T_{\ibf j}) \, \Big| \, T_\ibf, \, d(\ibf) \right) &\stackrel{\eqref{eq:mtbp-joint-dens}}{=} (d(\ibf)-1) \cdot \int \limits_0^\infty (T_\ibf \wedge s \wedge \hat{t}) \cdot \frac{f(s)}{\lambda(T_\ibf)} \cdot v_{\hat{t}}(s) \, \mathrm{d}s \stackrel{\eqref{eq:branch-op-rdcp}}{=}\\
& \frac{d(\ibf)-1}{E(T_\ibf)} \cdot (B_{\hat{t}}v_{\hat{t}})(T_\ibf) \le \Delta \cdot \mu_{\hat{t}} v_{\hat{t}}(T_\ibf) .
\end{split}\\
\begin{split} \label{eq:Yr+1-conditional-variance}
\text{Var} \left( \sum \limits_j v_{\hat{t}}(T_{\ibf j}) \, \Big| \, T_\ibf, \, d(\ibf) \right) &\le (d(\ibf) - 1) \cdot \ev \left( v_{\hat{t}}(T_{\ibf j})^2 \, \Big| \, T_\ibf, \, d(\ibf) \right) \stackrel{\eqref{eq:mtbp-joint-dens}}{=}\\
& (d(\ibf)-1) \cdot \int \limits_0^\infty (T_\ibf \wedge s \wedge \hat{t}) \cdot \frac{f(s)}{\lambda(T_\ibf)} \cdot v_{\hat{t}}(s)^2 \, \mathrm{d}s .
\end{split}
\end{align}

Therefore, by the law of total variance
\begin{align}
\begin{split}\label{eq:Ti-children-variance}
\text{Var} \left( \sum \limits_j v_{\hat{t}}(T_{\ibf j}) \, \Big| \, T_\ibf \right) &\le \ev \left(\text{Var} \left( \sum \limits_j v_{\hat{t}}(T_{\ibf j}) \, \Big| \, T_\ibf, \, d(\ibf) \right) \, \Bigg| \, T_\ibf \right) +\\
& \ev \left[ \left( \ev \left( \sum \limits_j v_{\hat{t}}(T_{\ibf j}) \, \Big| \, T_\ibf, \, d(\ibf) \right) \right)^2 \, \Bigg| \, T_\ibf \right] \stackrel{\eqref{eq:branch-op-rdcp},\, \eqref{eq:Yr+1-condiitonal-expectation},\, \eqref{eq:Yr+1-conditional-variance}}{\le}\\
& \left(B_{\hat{t}}(v_{\hat{t}}^2) \right)(T_\ibf) + \Delta^2 \cdot \mu_{\hat{t}}^2 \cdot v_{\hat{t}}(T_\ibf)^2.
\end{split}
\end{align}

By the law of total variance, we can estimate $\text{Var}(Y^{r+1}(\hat{t}))$:
\begin{align}
\begin{split}\label{eq:law-of-total-variance}
\text{Var}(Y^{r+1}(\hat{t})) &= \ev \left(\text{Var}(Y^{r+1}(\hat{t}) \, | \, \mathcal{F}_r) \right) +\text{Var} \left( \ev( Y^{r+1}(\hat{t}) \, | \, \mathcal{F}_r) \right) \stackrel{\eqref{eq:Yr-expectation}, \, \eqref{eq:Ti-children-variance}}{\le}\\
& \ev \left( \sum \limits_{\ibf \in \widetilde{\mathcal{M}}_{\underline{p}}^r(\hat{t})} \left[ \left(B_{\hat{t}}(v_{\hat{t}}^2) \right)(T_\ibf) + \Delta^2 \cdot \mu_{\hat{t}}^2 \cdot v_{\hat{t}}(T_\ibf)^2 \right] \right) + \mu_{\hat{t}}^2 \cdot \text{Var} \left( Y^r(\hat{t}) \right) \stackrel{\eqref{eq:power-of-branching-op}}{=}\\
& \left\langle \frac{E}{\lambda}, B_{\hat{t}}^r \left(B_{\hat{t}}(v_{\hat{t}}^2) \right) \right\rangle_\rho + \Delta^2 \cdot \mu_{\hat{t}}^2 \cdot \left\langle \frac{E}{\lambda}, B_{\hat{t}}^r \left(v_{\hat{t}}^2 \right) \right\rangle_\rho + \mu_{\hat{t}}^2 \cdot \text{Var} \left( Y^r(\hat{t}) \right) \stackrel{(*)}{\le}\\
& \tilde{c} \mu_{\hat{t}}^r + \mu_{\hat{t}}^2 \cdot \text{Var} \left( Y^r(\hat{t}) \right)
\end{split}
\end{align}
for some positive constant $\tilde{c}$, where at $(*)$ we used that
\begin{equation*}
\left\langle \frac{E}{\lambda}, B_{\hat{t}}^r \left(B_{\hat{t}}(v_{\hat{t}}^2) \right) \right\rangle_\rho \asymp \left\langle \frac{E}{\lambda}, B_{\hat{t}}^r \left(v_{\hat{t}}^2 \right) \right\rangle_\rho \le c \cdot \mu_{\hat{t}}^r.
\end{equation*}
Indeed: note that $\frac{E}{\lambda}$, $v_{\hat{t}}^2 \in L^2(\mathbb{R}_+, \rho)$ (see \refL{lem:L2-elements}), thus the statement follows from $\mu_{\hat{t}}^r = \norm{B_{\hat{t}}^r}$. Using \eqref{eq:law-of-total-variance}, it can be shown that $\text{Var}(Y^r(\hat{t})) \le \text{Var}(Y^1(\hat{t})) \cdot \mu_{\hat{t}}^{2r} + \tilde{c} \cdot \sum_{k=r-1}^{2r-1} \mu_{\hat{t}}^k$. It implies the second inequality of \eqref{eq:Yr-moments}.

Finally, note that the function $\hat{t} \mapsto \mu_{\hat{t}}$ is strictly increasing and continuous (these facts easily follow e.g.\ from the characterization of $\mu_{\hat{t}}$ given in \refL{lem:w-properties}). Thus the second and the third points of \refT{thm:PF-eigenvalue} together with \eqref{eq:t_crit-def} imply that $\norm{B_{\hat{t}_c(\underline{p})}} = \mu_{\hat{t}_c(\underline{p})} = 1$, 
which concludes the proof of \refT{thm:PF-eigenvalue}.
\end{proof}

\textbf{\large Acknowledgements.}
We would like to thank Christian Borgs and Amin Coja-Oghlan for questions that encouraged us to study the existence of local limits in the random \mbox{$d$-process}. 
We would also like to thank the organizers of the R{\'e}nyi 100~conference (Budapest), where this project was initiated in~2022. 
Bal\'azs R\'ath and M\'arton Sz\H{o}ke were partially supported by the grant NKFI-FK-142124 of NKFI (National Research, Development and Innovation Office), and the ERC Synergy Grant No.\ 810115 - DYNASNET.
Lutz Warnke was supported by NSF~CAREER grant~DMS-2225631, a Sloan Research Fellowship, and the RandNET project (Rise project H2020-EU.1.3.3).

\end{document}